\documentclass[oneside, reqno]{amsart}

\usepackage{graphicx}
\graphicspath{{images/}}
\usepackage[dvipsnames]{xcolor}

\usepackage{a4wide}
\usepackage[hidelinks]{hyperref}
\usepackage[alphabetic]{amsrefs}
\usepackage[shortlabels]{enumitem}
\usepackage{csquotes}

\theoremstyle{plain}
\newtheorem{thm}{Theorem}[section]
\newtheorem{prop}[thm]{Proposition}
\newtheorem{lemma}[thm]{Lemma}
\newtheorem{cor}[thm]{Corollary}
\theoremstyle{definition}
\newtheorem{definition}[thm]{Definition}
\newtheorem{notation}[thm]{Notation}

\newtheorem{remark}[thm]{Remark}
\newtheorem{example}[thm]{Example}

\newtheorem{assumption}[thm]{Assumption}

\usepackage{amssymb,amsmath,amsthm,mathrsfs}
\usepackage{graphicx}

\renewcommand{\max}{{\mathrm{max}}}
\renewcommand{\min}{{\mathrm{min}}}

\renewcommand{\u}{{\mathfrak u}}

\newcommand{\R}[0]{\ensuremath{\mathbb{R}}}
\newcommand{\N}{\ensuremath{\mathbb{N}}}
\newcommand{\C}{\ensuremath{\mathbb{C}}}
\newcommand{\g}{{\mathfrak g}}
\newcommand{\de}{\delta}
\renewcommand{\d}{\partial}

\newcommand{\cR}{{\mathcal{R}}}
\newcommand{\grad}{{\mathrm{grad}}}
\newcommand{\md}{{\mathrm{d}}}
\newcommand{\tr}{\mathrm{tr}}
\renewcommand{\div}{\mathrm{div}}
\newcommand{\ldr}[1]{\left\langle #1 \right\rangle}

\renewcommand{\a}{\alpha}
\renewcommand{\b}{\beta}
\newcommand{\p}{{\mathrm p}}
\renewcommand{\q}{q}
\renewcommand{\o}{\omega}
\renewcommand{\O}{{\mathcal O}}
\newcommand{\A}{{\mathcal A}}
\newcommand{\B}{{\mathcal B}}
\newcommand{\Cm}{{\mathcal C}}

\newcommand{\D}{\mathfrak D}
\newcommand{\De}{\mathrm D}

\newcommand{\abs}[1]{|#1|}
\newcommand{\norm}[1]{\left\|#1\right\|}

\newcommand{\res}{{\mathrm{res}}}

\newcommand{\X}{{\mathcal X}}

\newcommand{\T}{{\mathrm T}}
\newcommand{\Ell}{{\mathrm {Ell}}}
\newcommand{\WF}{{\mathrm{WF}}}

\renewcommand{\S}{\Sigma}
\newcommand{\s}{\sigma}
\newcommand{\n}{\nabla}

\newcommand{\im}{\mathrm{im}}
\newcommand{\supp}{\mathrm{supp}}
\renewcommand{\P}{{\mathrm{P}}}
\newcommand{\mP}{{\mathcal{P}}}

\newcommand{\Cr}{{\mathfrak{C}}}
\newcommand{\Vol}{{\mathrm{Vol}}}
\newcommand{\Hess}{{\mathrm{Hess}}}
\renewcommand{\L}{{\mathcal{L}}}
\renewcommand{\cR}{{\mathcal{R}}}
\renewcommand{\H}{{\mathrm{H}}}
\newcommand{\Ric}{{\mathrm{Ric}}}
\newcommand{\mX}{{\mathrm X}}
\newcommand{\mY}{{\mathrm Y}}
\newcommand{\Char}{{\mathrm{Char}}}
\renewcommand{\Re}{\mathrm{Re}}
\renewcommand{\Im}{\mathrm{Im}}
\newcommand{\id}{{\mathrm{id}}}
\newcommand{\oT}{{\overline {T^*}}}

\allowdisplaybreaks
\author{Oliver Petersen}
\address{Department of Mathematics, Stockholm University, Albanovägen 28, 10691 Stockholm, Sweden}
\email{oliver.petersen@math.su.se}

\title{The linear instability of Kasner spacetimes}

\begin{document}

\begin{abstract}
We prove linear stability of Kasner spacetimes in the direction of the Big Bang, up to an explicit finite dimensional space of non-decaying self-similar solutions.
The fastest growing solution is the linearization of Taub's explicit Bianchi II solution, describing a Kasner transition.
All other non-decaying solutions are inhomogeneous analogues of this, and the linearized Kasner metric.
The key novelty is a notion of quasinormal modes on Kasner spacetimes, similar to quasinormal modes on stationary black holes spacetimes.

All quiescent (as opposed to oscillatory) vacuum Big Bang spacetimes of dimension $4$ are asymptotic to a Kasner spacetime from the perspective of a single observer at the Big Bang.
They are expected to be highly unstable due to the oscillations conjectured by Belinski\v{\i}, Khalatnikov and Lifschitz.
This paper provides a complete description of this instability on a linear level without symmetry assumptions.
\end{abstract}

\maketitle

\parskip0ex
\tableofcontents

\section{Introduction} 
\parskip1ex

In a recent breakthrough result, Franco-Grisales and Ringstr\"om \cite{FR2026} prove stable formation of quiescent Big Bang singularities with complete asymptotics for the Einstein-scalar field equations in the non-degenerate subcritical regime, i.e.~for distinct Kasner exponents satisfying
\begin{equation} \label{eq: subextremality}
	p_j + p_k - p_l < 1,
\end{equation}
for all $j, k, l$ with $j \neq k$.
This extends the seminal result by Fournodavlos, Rodnianski and Speck \cite{FRS2023} by obtaining complete asymptotics at the Big Bang.
Franco-Grisales and Ringstr\"om prove that if one perturbs the initial data at a Cauchy hypersurface in a Kasner-scalar field spacetime, with distinct Kasner exponents satisfying \eqref{eq: subextremality}, then the resulting maximal globally hyperbolic Einstein-scalar field solution has a Big Bang as $t \to 0$ with convergent geometry.

Unfortunately, the arguments by Fournodavlos--Rodnianski--Speck and Franco-Grisales--Ringstr\"om do not carry over to the $4$-dimensional vacuum setting, since \eqref{eq: subextremality} is not compatible with the vacuum Kasner relations 
\begin{equation} \label{eq: Kasner relations}
	\sum_{j = 1}^n p_j^2 
		= \sum_{j = 1}^n p_j 
		= 1
\end{equation}
if the spacetime dimension is $n + 1 = 4$.
In fact, it is not expected that a similar result should be true in the vacuum case, due to the conjectures by Belinski\v{\i}, Khalatnikov and Lifschitz (BKL) in \cites{BKL1970, BKL1982}. 
Instead, the Kasner spacetimes are expected to be highly unstable, and each observer towards the Big Bang should experience their own instability.
We verify this expectation for the linearized Einstein equation by explicitly describing the instability from the point of view of a single observer, c.f.~Figure~\ref{fig: Causal}:

\begin{figure*} \label{fig: Causal}
  \begin{center}
    \includegraphics[scale = 0.8]{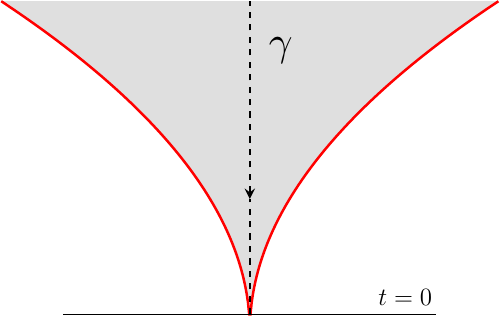}
  \end{center}
  \caption{
  	The shaded region illustrates the causal future $J^+(\gamma)$ of the observer $\gamma$ in a Kasner spacetime. The red curves illustrate the particle horizon of $\gamma$.
  	}
\end{figure*}

\begin{itemize}
	\item The first novelty is that any solution to the linearized Einstein equation in a non-flat Kasner spacetime can be written as an asymptotic sum of self-similar solutions of various growth rates, where all but finitely many decay towards the Big Bang, see Theorem~\ref{thm: main linearized Einstein, general dim}.
	\item The second novelty is that the space of \emph{non-decaying} self-similar solutions to the linearized Einstein equation in a non-flat Kasner spacetime of dimension $4$ can be computed explicitly, see Theorem~\ref{thm: self-similar solutions}.
	\item Theorem~\ref{thm: main linearized Einstein, general dim} and Theorem~\ref{thm: self-similar solutions} combine to give a complete description of the linear instability of $4$-dimensional non-flat Kasner spacetimes, see Theorem~\ref{thm: main linearized Einstein}.
	\item Finally, we note that any $4$-dimensional non-degenerate quiescent vacuum Big Bang spacetime (in the sense of Ringstr\"om~\cite{R2026_2}) is asymptotic to a non-flat Kasner spacetime from the perspective of a single observer, see Section~\ref{subsubsec: asymptotically Kasner}. 
	Our result therefore verifies the expected instability of any quiescent vacuum Big Bang spacetime, as predicted by the BKL conjecture, on a linear level. 
\end{itemize}

The key observation in this paper is the following:
The causal future of the world-line of a single observer travelling to the Big Bang in a non-flat Kasner spacetime shares many structural properties with the exterior of a Kerr-de Sitter black hole or the static patch in the de Sitter spacetime. 
We show that many concepts and ideas from the (by now very well developed) mathematical analysis of wave equations on black holes carry over to the Big Bang setting.
Our approach follows Vasy's microlocal analysis recipe for linear wave equations on black hole spacetimes introduced in \cite{V2013}.
The significant new challenges compared to black holes come from the \emph{anisotropy} of the Kasner spacetimes.

We prove the asymptotic expansion of solutions to the linearized Einstein equations in self-similar solutions by first developing the corresponding results for general systems linear of wave equations, see Theorem~\ref{thm: main QNM} and Theorem~\ref{thm: main as exp}.
As it turns out, only considering self-similar solutions are not quite sufficient for general wave equations.
Indeed, it is easy to construct examples of linear wave equations with real coefficients, that do not admit any self-similar solution (see Example~\ref{ex: scalar wave}).
We therefore have to pass to a generalization of self-similarity, which in black hole analysis is known as \emph{quasinormal modes}, where we allow the growth rate, or \emph{frequency}, to be a \emph{complex} number. 
To the best of our knowledge, quasinormal modes have never before been considered on Big Bang spacetimes.

There are two important differences between quasinormal modes on Big Bang spacetimes compared to quasinormal modes on black hole spacetimes.
Firstly, on black hole spacetimes, one typically expects the quasinormal modes to decay or come from perturbations of the black hole parameters, see for example Whiting's foundational mode stability result in \cite{W1989}.
As we will show, this is not the case in Kasner spacetimes.
Instead, there are unstable non-gauge quasinormal modes and these are precisely describing the linear instability which is the main result of this paper.
Secondly, quasinormal modes and their frequencies on black hole spacetimes are in general very hard to compute explicitly. 
In Kasner spacetimes however, we develop techniques that in principle allow for explicit computations of all quasinormal modes and frequencies of geometric wave operators, like the scalar wave operator or the Lichnerowicz-d'Alembert operator.
The latter is the relevant wave operator for the linearized Einstein equation. 
We use this to compute all unstable non-gauge (generalized) quasinormal modes for the linearized Einstein equations, thereby obtaining a complete characterization of the linear instability in Kasner spacetimes from the point of view of a single observer at the Big Bang, see Theorem~\ref{thm: main linearized Einstein}.

\subsection{Kasner spacetimes} \label{subsec: Kasner spacetimes}

Let us recall the definition of a Kasner spacetime.
\begin{definition} \label{def: Kasner}
Let $n \in \N$. 
A Kasner spacetime $(M, g)$ is given by
\begin{equation} \label{eq: Kasner}
	M
		= (0, \infty)_t \times \R^n_x, 
	\qquad
	g
		= - \md t^2 + \sum_{j = 1}^n t^{2p_j} \md x_j^2,
\end{equation}
such that $p_1, \hdots, p_n \in \R$ satisfy the Kasner relations \eqref{eq: Kasner relations}.
\end{definition}

Kasner spacetimes are solutions to the Einstein vacuum equation $\Ric_g = 0$.
They model a Big Bang at $t = 0$, unless one $p_j = 1$, in which case the Kasner relations force all other $p_j$ to vanish.
The latter special case is called the \emph{flat} Kasner spacetime and will be excluded in this paper. 
Indeed, the Kretschmann scalar is given by
\begin{equation} \label{eq: Krethschmann}
	\mathrm R_{\alpha \beta \gamma \delta} \mathrm R^{\alpha \beta \gamma \delta}
		= \frac4{t^4} \left( \sum_{j = 1}^n p_j^2 (1 - p_j)^2 + \sum_{j < k} p_j^2 p_k^2 \right),
\end{equation}
which is unbounded as $t \to 0$ in non-flat Kasner spacetimes.
The symmetry group of the Kasner spacetime is $n+1$-dimensional and generated by the Killing vector fields $\d_{x_1}, \hdots, \d_{x_n}$ and the homothetic Killing vector field
\begin{equation} \label{eq: T}
	\T
		:= - t \d_t - \sum_{j = 1}^n (1-p_j) x_j \d_{x_j},
\end{equation}
satisfying $\L_\T g = - 2 g$.
In order to measure decay of tensors on the spacetime, we fix a \emph{Riemannian} metric which satisfies the same symmetries as the Kasner spacetime.
Though all choices of such are equivalent, a natural choice is $\tilde g := \md t^2 + \sum_{j = 1}^n t^{2p_j} \md x_j^2$.
This indeed satisfies $\L_\T \tilde g = - 2 \tilde g$ and $\L_{\d_{x_j}}\tilde g = 0$ for $j = 1, \hdots, n$.
For a tensor $u$, we will measure decay of $u$ by the norm $\abs{u} := \sqrt{\tilde g(u, u)}$.
The Kasner metric $g$ does not decay as $t \to 0$, since
\begin{equation} \label{eq: g decay}
	\abs{g}
		= n + 1.
\end{equation}

Let 
\begin{equation} \label{eq: gamma}
	\gamma(\tau) 
		:= (\tau, 0, \hdots, 0) \in M,
\end{equation}
for $\tau \in (0, \infty)$, be the observer that follows the $t$-axis to the Big Bang as $\tau \to 0$.
This observer is invariant under the flow of $\T$, which means that the \emph{causal future} of $\gamma$,
\begin{equation} \label{eq: causal future of obs}
	J^+(\gamma)
		:= \overline{\bigcup_{\tau \in (0, \infty)} J^+(\gamma(\tau))},
\end{equation}
where $J^+(p)$ is the causal future of $p \in M$, is also invariant under the flow of $\T$.
See an illustration of $J^+(\gamma)$ in Figure~\ref{fig: Causal}.
We are interested in the asymptotics of solutions to the linearized Einstein vacuum equation in a region of the form
\[
	J^+_\de(\gamma)
		:= J^+(\gamma) \cap \{t \leq \de\},
\]
for some $\de > 0$.

\begin{remark} \label{rmk: Kasner-scalar field}
The Kasner-scalar field spacetimes $(M, g, \phi)$ are a generalization of the Kasner spacetimes with a metric $g$ on the same form and a scalar field given by $\phi = a \log(t)$, satisfying the Kasner-scalar field relations $\sum_{j = 1}^n p_j^2 + a^2 = \sum_{j = 1}^n p_j = 1$. 
With the extra flexibility of choosing $a \in \R$, there is a large class of Kasner-scalar field exponents that satisfy \eqref{eq: subextremality}.
This is the setting of the results of Franco-Grisales--Ringstr\"om \cite{FR2026} and Fournodavlos--Rodnianski--Speck \cite{FRS2023}.
\end{remark}

\subsection{Main results}

As opposed to \cites{FR2026, FRS2023}, we are here interested in analyzing solutions to the linearized Einstein equation in a $4$-dimensional Kasner spacetime, which is a \emph{vacuum} spacetime.
We therefore localize the analysis to the causal future of an observer, defined in \eqref{eq: causal future of obs}, see also Figure~\ref{fig: Causal}.
Our main result is a complete description of the \emph{linear instability of Kasner spacetime}, Theorem~\ref{thm: main linearized Einstein}.

\subsubsection{Expansion in self-similar solutions}

We consider non-flat Kasner spacetimes of any spacetime dimension $n + 1$.
(The Kasner relations \eqref{eq: Kasner relations} actually imply that $n \geq 3$.)
Our first result says that any solution to the linearized Einstein equation has an asymptotic expansion in terms of self-similar solutions as $t \to 0$, i.e.~solutions $\u \in C^\infty(J^+(\gamma))$ satisfying
\[
	\left( \L_\T - \lambda \right)^k \u
		= 0
\]
for some $\lambda \in \R$ and some $k \in \N_0$, where $\L$ denotes the Lie derivative.
We use the notation
\[
	\L_{t^{1-p}\d_x}^\a
		:= \left( \L_{t^{1-p_1}\d_{x_1}} \right)^{\a_1} \hdots \left( \L_{t^{1-p_n}\d_{x_n}}\right) ^{\a_n}
\]
for the spatial derivatives, where $\a = (\a_1, \hdots, \a_n) \in \N_0^n$.

\begin{thm}[The asymptotic expansion] \label{thm: main linearized Einstein, general dim} 
Let $n \in \N$ and let $(M, g)$ be a non-flat Kasner spacetime of dimension $n+1$, see Definition~\ref{def: Kasner}, and let $\gamma$ be the observer introduced in \eqref{eq: gamma}.
There is a discrete set $\u_1, \u_2, \hdots \in C^\infty(M)$, satisfying
\begin{equation} \label{eq: u_j}
	\De \Ric_g(\u_j) 
		= 0,
	\qquad \left( \L_\T - \lambda_j \right)^{k_j} \u_j 
		= 0,
\end{equation}
for some $k_j \in \N$ and $\lambda_j \in \R$, and the gauge condition 
\begin{equation} \label{eq: gauge condition}
	\div_g \left( \u_j - \frac12 \tr_g(\u_j) g \right) 
		= 0,
\end{equation}
for all $j \in \N$, such that the following holds:
If $l \in \R$, $\de > 0$ and $u \in C^\infty(J_\de^+(\gamma))$ satisfies the linearized Einstein vacuum equation
\[
	\De \Ric_g(u) = 0,
\]
then there is a vector field $X \in C^\infty(M)$ and constants $c_1, \hdots, c_N \in \R$, such that
\begin{equation} \label{eq: asymptotic expansion higher}
	\left| \L_{t\d_t}^k \L_{t^{1-p}\d_x}^\a \left( u - \L_X g - \sum_{j = 1}^N c_j \u_j \right) (t, x) \right| \leq C_{k, \a, l} t^l
\end{equation}
for any $k \in \N_0$ and $\a \in \N_0^n$, for some $C_{k, \a, l} > 0$ and all $(t, x) \in J_\de^+(\gamma)$.
\end{thm}

\begin{remark}[Self-similarity and decay]  \label{rmk: self-similar intro}
If a non-trivial $\u \in C^\infty(M)$ satisfies $(\L_\T - \lambda) \u = 0$, one can check using $\L_\T \tilde g = - 2 \tilde g$ that $\T \left( t^{2 + \lambda} \abs{\u} \right) = 0$.
It follows that
\[
	\abs{\u(t, x)} 
		\leq C t^{- (2 + \lambda)},
\]
for some $C > 0$ and all $(t, x) \in J^+(\gamma)$, and that this cannot be improved to $\abs{\u(t, x)} \leq C t^{- (2 + \lambda) + \epsilon}$ for any $\epsilon > 0$.
If, more generally, $\u$ satisfies $(\L_\T - \lambda)^k \u = 0$, then the optimal decay rate is 
\[
	\abs{\u(t, x)} 
		\leq C t^{- (2 + \lambda)}\abs{\log(t)}^{k'},
\]
for some $k' \leq k - 1$ and $C > 0$, and all $(t, x) \in J^+(\gamma)$.
Since the vector fields $t\d_t$ and $t^{1-p_j}\d_{x_j}$ commute with $\T$, for every $j$, the same is true after applying such spatial derivatives. 
Consequently, the terms $\u_j$ can only be removed from the asymptotic expansion \eqref{eq: asymptotic expansion higher} if $-(2+\lambda_j) > l$ or if $\u_j$ is a gauge solution.
\end{remark}

\subsubsection{The unstable self-similar solutions} \label{subsubsec: USS}

Theorem~\ref{thm: main linearized Einstein, general dim} and Remark~\ref{rmk: self-similar intro} imply that the obstacles to decay for solutions to the linearized Einstein equation are the non-gauge self-similar solutions $\u_j$ with growth rate $\lambda \geq - 2$, satisfying the gauge condition \eqref{eq: gauge condition}.
Our next result is that all these can be computed explicitly if the spacetime dimension is $4$.
We will find them among the following set of solutions:

\textbf{The linearized Kasner metric:}
By differentiating a one-parameter family of Kasner solutions
	\begin{equation} \label{eq: g' def}
		g'
			: = \frac{\md}{\md r}\Big|_{r = 0} \left( -\md t^2 + \sum_{j = 1}^3 t^{2p_j(r)} \left( \md x_j \right)^2 \right)
			= 2 \log(t) \sum_{j = 1}^3 p_j' t^{2p_j} \left( \md x_j \right)^2,
	\end{equation}
	we obtain the \emph{linearized Kasner metric} $g'$, which solves
	\begin{equation} \label{eq: g' QNM}
		\De \Ric_g(g') 
			= 0, 
		\qquad 
		\left( \L_\T + 2 \right)^2 g'
			= 0, 
	\end{equation}
	and satisfies the gauge condition \eqref{eq: gauge condition}, see Theorem~\ref{thm: AVTD system}.
	Moreover, it is not a gauge solution, see Lemma~\ref{le: linearized Kasner non gauge}.
	Hence $g'$ is one of the self-similar solutions in Theorem~\ref{thm: main linearized Einstein, general dim} with $k = 2$ and $\lambda = -2$.
	By Remark~\ref{rmk: self-similar intro}, we have thus found our first unstable non-gauge self-similar solution.
	One can also note that $g'$ is not decaying without using Remark~\ref{rmk: self-similar intro}, by simply computing its norm to be
	\begin{equation} \label{eq: g' decay}
		\abs{g'}
			= 2 \abs{\log(t) \hspace{0.4mm} } \sqrt{\sum_{j=1}^3 \left( p_j'\right)^2},
	\end{equation}
	which grows logarithmically as $t \to 0$.
	In other words, $g'$ is growing at the same polynomial rate in $t$ as $g$ as $t \to 0$, see \eqref{eq: g decay}, but slightly faster due to the logarithmic factor.
	The linearization of the Kasner relations \eqref{eq: Kasner relations} are $\sum_{j = 1}^3 p_j'p_j = \sum_{j = 1}^3 p_j' = 0$, which gives the explicit formula $(p_1', p_2', p_3') = c (p_3 - p_2, p_2 - p_1, p_1 - p_3)$, for some $c \in \R$.

\textbf{An explicit class of self-similar solution:}
Up to permutation, the Kasner relations \eqref{eq: Kasner relations} imply for a non-flat Kasner spacetime of dimension $4$ that $p_1, p_2 \in (0, 1)$ and $p_3 \in [-1/3, 0)$.
Define
\begin{equation} \label{eq: v 1 k}
	v_k(t, x)
		:= \sum_{l = 0}^k a_{lk} t^{l(1-p_1)} \left( x_1 \right)^{k-l} t^{2p_3} \md x^2 \otimes_s \md x^3
\end{equation}
for each integer $1 \leq k < 2 \frac{p_2 - p_3}{1-p_1}$, where $a_{0 k} = \frac1{k!}$ and
\begin{equation} \label{eq: a m K}
	a_{lk}
		= \frac1{(k-l)!} \prod_{q = 1}^{\frac l2} \frac1{\left( 2q (1 - p_2) + p_3 - p_1 \right)^2 - \left( p_3 - p_1 \right)^2},
\end{equation}
when $l$ is even and positive, and $a_{lk} = 0$ if $l$ is odd.
We similarly define $w_k$ for each integer $1 \leq k \leq 2 \frac{p_1 - p_3}{1-p_2}$ by swapping the indices $1$ and $2$ in equations~\eqref{eq: v 1 k} and \eqref{eq: a m K}.
Then 
\begin{align*}
	\De \Ric_g(v_k) 
		= & \ 0, \\
	\De \Ric_g(w_k)
		= & \ 0,
\end{align*}
for all $k$ satisfying the above bounds, see Theorem~\ref{thm: explicit solutions} below.
All $v_k$ and $w_k$ are \emph{self-similar} with respect to $\T$, with
\begin{equation} \label{eq: explicit QNM decay}
\begin{split}
	\left( \L_\T + 2 + p_3 - p_2 + k (1-p_1) \right) v_k
		= & \ 0, \\
	\left( \L_\T + 2 + p_3 - p_1 + k (1-p_2) \right) w_k
		= & \ 0.
\end{split}
\end{equation}
Moreover, all these solutions satisfy the gauge condition \eqref{eq: gauge condition}, see Theorem~\ref{thm: explicit solutions}. 
By Remark~\ref{rmk: self-similar intro}, these will be \emph{unstable} (i.e.~non-decaying) as $t \to 0$ precisely when $1 \leq k \leq \frac{p_2 - p_3}{1-p_1}$ and $1 \leq k \leq \frac{p_1 - p_3}{1-p_2}$, respectively.
Note that $v_1$ and $w_1$ are gauge related, since
\begin{equation} \label{eq: v1 w1}
	\frac12 \L_{x_1 x_2 \d_{x_3}}g
		= v_1 + w_1.
\end{equation}
Apart from this, all $v_k$ and $w_k$ are linearly independent, even up to adding gauge solutions, see Proposition~\ref{prop: linear independence}.

\textbf{An extra self-similar solution when two Kasner exponents coincide:}
In the special case that two Kasner exponents coincide, there is one more solution that will be needed. 
Up to permutation, we are left with the case
\[
	p_1 = p_2 = \frac23, \quad p_3 = -\frac13.
\]
In this particular case, we have to include one more to the list of unstable solutions:
\[
	z(t, x)
		:= x_1 \left( \frac16x_2^2t^{- \frac23} - 1 \right) \md x_1 \otimes_s \md x_3 + x_2\md x_2 \otimes_s \md x_3.
\]
By Theorem~\ref{thm: LRS case}, this satisfies $\De \Ric_g(z) = 0$ and is self-similar with $\L_\T z = - 2 z$, so Remark~\ref{rmk: self-similar intro} implies that $z$ is not decaying.

\begin{thm}[The unstable self-similar solutions] \label{thm: self-similar solutions}
Let $(M, g)$ be a non-flat Kasner spacetime of dimension $4$, see Definition~\ref{def: Kasner}, and let $\gamma$ be the observer introduced in \eqref{eq: gamma}.
Up to permutation of the indices, we may assume that $p_1, p_2 \in (0, 1)$ with $p_1 > p_2$ and $p_3 \in [-1/3, 0)$.
Let $K \in \N_0$ be the largest integer such that $K \leq \frac{p_2 - p_3}{1-p_1}$.
Assume that $\u \in C^\infty(M)$ satisfies
\[
	\De \Ric_g(\u) 
		= 0,
	\qquad 
	\left( \L_\T - \lambda \right)^k \u
		= 0,
\]
for some $k \in \N$ and $\lambda \geq -2$.
\begin{itemize}
	\item If $p_3 \in \left( -\frac27, 0 \right)$, then there are unique constants $a, b_1, \hdots, b_K \in \R$ and a vector field $X \in C^\infty(M)$, such that
	\[
		\u
			= \L_X g + a g' + \sum_{k = 1}^K b_k v_k.
	\]
	\item If $p_3 \in \left( -\frac13, -\frac27 \right]$, then there are unique constants $a, b_1, \hdots, b_K, c \in \R$ and a vector field $X \in C^\infty(M)$, such that
	\[
		\u
			= \L_X g + a g' + \sum_{k = 1}^K b_k v_k + c w_2.
	\]
	\item If $p_3 = - \frac13$, then there are unique constants constants $a, b_1, b_2, b_3, c_2, c_3, d \in \R$ and a vector field $X \in C^\infty(M)$, such that
	\[
		\u
			= \L_X g + a g' + \sum_{k = 1}^3 b_k v_k + \sum_{k = 2}^3 c_k w_k + d z.
	\]
\end{itemize}
\end{thm}

Note that $w_1$ is not appearing here, since \eqref{eq: v1 w1} implies that it relates to $v_1$ by a gauge solution.

\begin{remark}
With the convention $p_1, p_2 \in (0, 1)$ with $p_1 > p_2$ and $p_3 \in [-1/3, 0)$, as in Theorem~\ref{thm: self-similar solutions}, the Kasner relations \eqref{eq: Kasner relations} imply that
\[
	\frac{p_2 - p_3}{1-p_1}
		= 1 - \frac{4 p_3}{1 + p_3 + \sqrt{(1+p_3)^2 - 4p_3^2}}.
\]
In concrete situations, the number $K$ is therefore easily computed. 
\end{remark}

\begin{remark}
One can in practice conclude more information about the constants appearing in Theorem~\ref{thm: self-similar solutions}.
For example, since the growth rate of all $v_k$ are different, at most one of the constants $b_k$ can be non-zero for a given $\u$.
However, none of the terms appearing in the decomposition of $\u$ can be removed in general, since the constants are unique.
\end{remark}

\subsubsection{The linear instability of Kasner spacetimes}

Combining Theorem~\ref{thm: main linearized Einstein, general dim} with Theorem~\ref{thm: self-similar solutions}, we may formulate the main result of this paper:

\begin{thm}[The linear instability of Kasner spacetimes] \label{thm: main linearized Einstein} 
Let $(M, g)$ be a non-flat Kasner spacetime of dimension $4$, see Definition~\ref{def: Kasner}, and let $\gamma$ be the observer introduced in \eqref{eq: gamma}.
Up to permutation of the indices, we may assume that $p_1, p_2 \in (0, 1)$ with $p_1 > p_2$ and $p_3 \in [-1/3, 0)$.
Let $K \in \N_0$ be the largest integer such that $K \leq \frac{p_2 - p_3}{1-p_1}$.
There is an $\epsilon > 0$, such that the following holds: 
If $\de > 0$ and $u \in C^\infty(J_\de^+(\gamma))$ solves the linearized Einstein vacuum equation
\[
	\De \Ric_g(u) = 0,
\]
then the following holds:
\begin{itemize}
	\item If $p_3 \in \left( -\frac27, 0 \right)$, then there are unique constants $a, b_1, \hdots, b_K \in \R$ and a vector field $X \in C^\infty(M)$, such that
	\begin{equation} \label{eq: asymptotic expansion 1}
	\left| \L_{t\d_t}^k \L_{t^{1-p}\d_x}^\a \left( u - \L_X g - a g' - \sum_{k = 1}^K b_k v_k \right) (t, x) \right| \leq C_{k, \a} t^{\epsilon}
\end{equation}
	for all $k \in \N_0$ and $\a \in \N_0^3$, for some $C_{k, \a} > 0$ and all $(t, x) \in J_\de^+(\gamma)$, 
	\item If $p_3 \in \left( -\frac13, -\frac27 \right]$, there are unique constants $a, b_1, \hdots, b_K, c \in \R$ and a vector field $X \in C^\infty(M)$, such that
	\begin{equation} \label{eq: asymptotic expansion 2}
	\left| \L_\T^k \L_{t^{1-p}\d_x}^\a \left( u - \L_X g - a g' - \sum_{k = 1}^K b_k v_k - c w_2 \right) (t, x) \right| \leq C_{k, \a} t^{\epsilon}
\end{equation}
for all $k \in \N_0$ and $\a \in \N_0^3$, for some $C_{k, \a} > 0$ and all $(t, x) \in J_\de^+(\gamma)$, 
	\item If $p_3 = - \frac13$, there are unique constants constants $a, b_1, b_2, b_3, c_2, c_3, d \in \R$ and a vector field $X \in C^\infty(M)$, such that
\begin{equation} \label{eq: asymptotic expansion 3}
	\left| \L_\T^k \L_{t^{1-p}\d_x}^\a \left( u - \L_X g - a g' - \sum_{k = 1}^3 b_k v_k - \sum_{k = 2}^3 c_k w_k - d z \right) (t, x) \right|
		\leq C_{k,\a} t^\epsilon,
\end{equation}
for all $k \in \N_0$ and $\a \in \N_0^3$, for some $C_{k, \a} > 0$ and all $(t, x) \in J_\de^+(\gamma)$.
\end{itemize}
Conversely, the estimates \eqref{eq: asymptotic expansion 1}, \eqref{eq: asymptotic expansion 2}, and \eqref{eq: asymptotic expansion 3} are false if one removes any term.
\end{thm}

Strictly speaking, Theorem~\ref{thm: main linearized Einstein} is not a direct consequence of combining Theorem~\ref{thm: main linearized Einstein, general dim} and Theorem~\ref{thm: self-similar solutions}, since it needs a slightly stronger uniqueness statement provided by Proposition~\ref{prop: linear independence}.

The question that arises is whether there are one-parameter families of metrics $s \mapsto g_s$ solving the Einstein vacuum equation, with linearization given by our unstable solutions $v_k$, $w_k$ and $z$ in Theorem~\ref{thm: self-similar solutions} and Theorem~\ref{thm: main linearized Einstein}. 
As shown by Moncrief in his linearization stability results in \cites{M1975, M1976}, this is a subtle question when the background spacetime has Killing vector fields like the Kasner spacetime.
The next two remarks address this question.

\begin{remark}[Non-linearized version of $v_1$ and $w_1$] \label{rmk: Bianchi II}
Consider the Taub's explicit family of Bianchi II vacuum metrics
\[
	g_s
		= - A(t)^2 \md t^2 + t^{2p_1}A(t)^2\md x_1^2 + t^{2p_2}A(t)^2 \md x_2^2 + t^{2 p_3}A(t)^{-2}\left( \md x_3 + 4p_3 s x_1 \md x_2 \right)^2,
\]
where $A(t)^2 := 1 + s^2 t^{4 p_3}$ and $\sum_{j = 1}^3 p_j = \sum_{j = 1}^3 p_j^2 = 1$, where $s \in \R$.
They satisfy the vacuum equation $\Ric(g_s) = 0$, and we compute that
\[
	\frac{\md}{\md s}\Big|_{s = 0} g_s
		= 8p_3 x_1 t^{2 p_3} \md x_2 \otimes_s \md x_3
		= 8p_3 v_1.
\]
In other words, $v_1$ is, up to a constant factor, the linearization of the family $g_s$ of solutions to Einstein's vacuum equation.
This particular family $g_s$ describes a \emph{Kasner transition}, where the expansion-normalized Weingarten map (the natural generalization of Kasner exponents) converge to
\begin{align*}
	(p_1, p_2, p_3)
		\text{ as } t \to \infty, \qquad
	\left( \frac{p_1 + 2p_3}{1 + 2p_3}, \frac{p_2 + 2 p_3}{1 + 2p_3}, \frac{- p_3}{1 + 2p_3} \right)
		\text{ as } t \to 0.
\end{align*}
Consequently, $v_1$ is the linearization of a Kasner transition.
The solution $w_1$ is similarly the linearization of an isometric copy of the same Kasner transition.
Theorem~\ref{thm: main linearized Einstein} imply on a linear level that these transition maps give the leading order instability.
\end{remark}

\begin{remark}[Non-linearized version of $v_k$ and $w_k$ for $k \geq 2$]
Assume that $g_s$ is a curve of metrics such that $g_0$ is a non-flat Kasner metric and that $\frac{\md}{\md s} g_s |_{s = 0} = v_k$.
We claim that $g_s$ is an \emph{inhomogeneous} metric if $k \geq 2$.
For this, assume that $X_s$ is a family of Killing vector fields depending smoothly on $s$. 
It follows that
\[
	0
		= \frac{\md}{\md s} \Big |_{s = 0} \L_{X_s}g_s
		= \L_{X'}g + \L_{X_0} v_k
\]
where $X' : = \frac{\md}{\md s}|_{s = 0} X_s$.
Since $X_0$ is a Killing vector field with respect to the Kasner metric, it follows that $X_0 = \sum_{j = 1}^3 b_j \d_{x_j}$ for some $b_1, b_2, b_3 \in \R$.
Hence
\[
	0
		= \L_{X'}g + b_1 \L_{\d_{x_1}} v_k
		= \L_{X'}g + b_1 k v_{k-1}.
\]
By Proposition~\ref{prop: linear independence}, $v_{1k-1}$ is not a gauge solution if $k - 1 \geq 1$, so it follows that $b_1 = 0$.
But this means that $X_0$ is spanned by $\d_{x_2}$ and $\d_{x_3}$. 
Therefore there cannot be three linearly independent Killing vector fields depending smoothly on $s$, and the conclusion is that $g_s$ cannot be a homogeneous family of metrics for any $s \neq 0$.
\end{remark}

\subsubsection{General quiescent vacuum spacetimes} \label{subsubsec: asymptotically Kasner}
Let us now show why Kasner spacetimes indeed are the natural limits of \emph{any} $4$-dimensional non-degenerate quiescent vacuum spacetime after `zooming in' with the vector field $\T$.
Let $(M, g_*)$ be a non-degenerate quiescent vacuum spacetime (with a Gaussian foliation) in the sense of Ringstr\"om's \cite{R2026_2}*{Def.1.7}.
By \cite{R2026_2}*{Lem.~3.3}, we can write the metric as
\[
	g_*
		:= - \md t^2 + \sum_{j = 1}^3 a_{jk}t^{2\max(p_j, p_k)} \md x_j \otimes_s \md x_k,
\]
where $\norm{a_{jk}(\cdot, t) - c_{jk}}_{C^l} \leq C t^{2\eta}$ for some $\eta > 0$.
The proof of \cite{R2026_2}*{Lem.~3.3} also shows that $c_{ij}|_{x = 0} = \de_{ij}$ for $i, j = 1, 2, 3$, see also the proof of \cite{R2026_2}*{Lem.~3.1}.
Therefore, we conclude that
\begin{equation} \label{eq: asymptotically Kasner}
	g_*
		= g + \sum_{j = 1}^3 (a_{jj} - 1) t^{2p_j} \md x_j^2 + \sum_{j \neq k} a_{jk}t^{2\max(p_j-p_k, p_k-p_j)}  t^{p_j}\md x_j \otimes t^{p_j}\md x_j,
\end{equation}
where $g$ is the Kasner metric with Kasner exponents $p_1, p_2, p_3$.
Since both the second and the third term in \eqref{eq: asymptotically Kasner} decay faster than the Kasner metric along $\T$ (as both $t$ \emph{and} all $\abs{x_j}$ decrease in the direction of $\T$), this shows that the Kasner metric $g$ is the limit after zooming in with $\T$.
In this sense, the results of this paper is the natural linear version of the instability of non-degenerate quiescent vacuum Big Bang spacetimes conjectured by Belinski\v{\i}, Khalatnikov and Lifschitz.

\subsubsection{Quasinormal modes at the Big Bang}
As stated in our main results for the linearized Einstein equation, Theorem~\ref{thm: main linearized Einstein, general dim}, Theorem~\ref{thm: self-similar solutions} and Theorem~\ref{thm: main linearized Einstein}, we find in this paper asymptotic expansions of solutions to the linearized Einstein equation in self-similar solutions.
We will prove this by providing a general theory for such asymptotic expansions for solutions to systems of linear wave equations on Kasner-type spacetimes.

It turns out that for many linear wave equations, the solutions cannot be written as a sum of self-similar solutions $v$ satisfying $(\T - \lambda)^k v = 0$, where $k \in \N$ and $\lambda$ is a real number.
See Example~\ref{ex: scalar wave} for a simple instance of this. 
Instead, we need to also allow for complex scale factors.
This leads us to the notion of \emph{quasinormal modes} (QNMs), where we allow for a complex frequencies $\lambda$.

We will show for any linear system of wave equations that scales naturally with $\T$, all smooth solutions have asymptotic expansions in terms of QNMs. 
The result will hold for more general spacetimes than Kasner spacetimes, we call these the \emph{Kasner-type spacetimes}:

\begin{definition} \label{def: Kasner-type spacetime}
Let $n \geq 1$. 
A Kasner-type spacetime $(M, g)$ is given by
\[
	M
		= (0, \infty)_t \times \R^n_x, 
	\qquad
	g
		= - \md t^2 + \sum_{j = 1}^n t^{2p_j} \md x_j^2,
\]
such that $p_1, \hdots, p_n \in \R$ with $p_j < 1$ for $j = 1, \hdots, n$.
\end{definition}

In other words, we do not ask the Kasner relations to hold, but we still exclude the flat Kasner spacetime by demanding that $p_j < 1$ for all $j$.
Similar to Kasner spacetimes, the symmetry group of a Kasner-type spacetime is $n+1$-dimensional, generated by the same Killing vector fields $\d_{x_1}, \hdots, \d_{x_n}$ and the homothetic Killing vector field $\T$ given by the formula \eqref{eq: T}, satisfying $\L_{\T}g = - 2 g$.
We also define the causal future $J^+(\gamma)$  of the world-line of an observer $\gamma$ as in \eqref{eq: gamma} analogously to \eqref{eq: causal future of obs}.

In our analysis, we will need the linear wave operators $\P$ we consider to satisfy
\begin{equation} \label{eq: T commutator first}
	[\T, t^2 \P]
		= 0.
\end{equation}
Since $\L_\T g = - 2 g$, linear wave operators associated with the Kasner-type geometry (for example $\P = \n^*\n$ on tensor fields) will typically satisfy \eqref{eq: T commutator first}.
Our first result for linear wave equations is the construction of (generalized) QNMs:

\begin{thm}[Generalized quasinormal modes] \label{thm: main QNM}
Let $(M, g)$ be a Kasner-type spacetime, see Definition \ref{def: Kasner-type spacetime}.
Assume that $\P$ is a linear wave operator on $\C^m$-valued functions, for some $m \in \N$, such that $[\T, t^2 \P] = 0$.
Then there exists a discrete set $\s_1, \s_2, \hdots \in \C{}$, such that the following holds:
\begin{itemize}
\item For every $\a > 0$, the set of all $\s_j$ satisfying $\Im(\s_j) \geq - \a$ is finite.
\item If $u \in C^\infty(M; \C^m)$ is a non-trivial solution to
\begin{equation} \label{eq: gen mode equation}
	\P u
		= 0, 
	\qquad 
	\left( \T + i \s \right)^k u
		= 0,
\end{equation}
for a $\s \in \C$ and $k \in \N$, then $\s = \s_j$ for some $j \in \N$.
\item For any fixed $\s \in \C{}$, the vector space of $u \in C^\infty(M; \C^m)$ solving \eqref{eq: gen mode equation}, for some $k \in \N$, is finite dimensional.
\end{itemize}
\end{thm}

In other words, there is a certain invariantly defined discrete set of frequencies $\s_1, \s_2, \hdots \in \C$ associated with the operator $\P$, such that \eqref{eq: gen mode equation} admits a non-trivial solution.
For a concrete example, we refer to Example~\ref{ex: scalar wave}.

\begin{definition} \label{def: QNM}
The numbers $\s_1, \s_2, \hdots \in \C{}$ are called the \emph{quasinormal mode frequencies} of $\P$.
For a $\s \in \C{}$, a solution $u \in C^\infty(M; \C^m)$ to \eqref{eq: gen mode equation} is called a \emph{quasinormal mode} (QNM) of $\P$ if $k = 1$, and a \emph{generalized quasinormal mode} of $\P$ if $k \geq 1$. 
\end{definition}

A QNM is in particular a generalized QNM.
Another natural name for the numbers $\s_1, \s_2, \hdots \in \C{}$ is \emph{resonance frequencies} and \emph{resonant states} for the corresponding solutions to \eqref{eq: gen mode equation}.
Definition~\ref{def: QNM} is completely parallel to the definition of quasinormal modes on black hole spacetimes, c.f.~for example \cite{W1989} or \cite{V2013}.
For black holes, the stationary Killing vector field is used in place of the homothetic Killing vector field $\T$ used here.

\begin{remark}
A quasinormal mode satisfies by definition the condition $\T u = - i \s$, for some $\s \in \C$.
The reason for writing `$- i\s$' instead of simply `$\s$' or `$\lambda$' is that we follow the standard conventions for quasinormal modes on black holes, see e.g.~\cite{V2013} and references therein.
A QNM frequency $\s$ with $\Im(\s) > 0$ implies that the corresponding (generalized) QNM is growing as $t \to 0$, while $\Im(\s) < 0$ implies that it decays.
\end{remark}

\begin{remark}
Note that any generalized QNM can be written $u = t^{i \s} \sum_{j = 0}^{k-1} \log(t)^j a_j$, where $a_j \in C^{\infty}(M; \C^m)$ with $\T a_j = 0$.
Comparing with Remark~\ref{rmk: self-similar intro}, where we computed the decay rate of self-similar solutions, we similarly conclude that 
\[
	\abs{u(t, x)} 
		\leq C t^{\Im(\s) + \epsilon},
\]
for all $(t, x) \in J^+(\gamma)$, for any $\epsilon > 0$, and we may set $\epsilon = 0$ if and only if $k = 1$.
\end{remark}

\begin{remark} \label{rmk: geometric wave}
Linear wave operators associated with the geometry of Kasner-type spacetimes will typically satisfy $[\d_{x_j}, \P] = 0$ in addition to $[\T, t^2\P] = 0$.
Examples include the d'Alembert operator and the Lichnerowicz-d'Alembert operator, studied in Section~\ref{sec: QNM lin Ein}.
This class of \emph{geometric wave operators} is studied in Section~\ref{sec: QNM geometric waves}, and one of the main features, see Corollary~\ref{cor: structure of QNMs}, is that the generalized QNMs for such operators are polynomials in $x$, which makes them particularly easy to compute. 
\end{remark}

\begin{remark} \label{rmk: gen QNM are QNM}
If $u \in C^\infty(M; \C^m)$ is a generalized QNM with frequency $\s \in \C$ and order $k \in \N_0$, then $v := \left( \T + i \s \right)^{k-1} u$ is a QNM, since
\[
	t^2\P v
		= \left( \T + i \s \right)^{k-1} t^2 \P u
		= 0, 
	\qquad 
	\left( \T + i \s \right) v
		= \left( \T + i \s \right)^k u
		= 0.
\]
In other words, any generalized QNM gives rise to a QNM of the same frequency.
Consequently, if we want to show that $\s \in \C{}$ is \emph{not} a QNM frequency, it is enough to show that there is no QNM with that frequency.
\end{remark}

Our second result for linear wave equations is that solutions to linear wave equations satisfying \eqref{eq: T commutator first} have an asymptotic expansion in the generalized QNM given by Theorem~\ref{thm: main QNM}. 
We use the notation
\[
	\left( t^{1-p}\d_x \right)^\a 
		:= \left( t^{1-p_1} \d_{x_1}\right)^{\a_1} \hdots \left( t^{1-p_n} \d_{x_n}\right)^{\a_n}
\]
for any $\a \in \N_0^n$.

\begin{thm}[Asymptotic expansion] \label{thm: main as exp}
Let $(M, g)$ be a Kasner-type spacetime, see Definition \ref{def: Kasner-type spacetime}.
Assume that $\P$ is a linear wave operator on $\C^m$-valued functions, for some $m \in \N$, such that $[\T, t^2 \P] = 0$.
Let $v_1, v_2, \hdots \in C^\infty(M)$ be the generalized QNMs associated with $\P$ provided by Theorem~\ref{thm: main QNM} (see Definition~\ref{def: QNM}), ordered so that $\Im\left(\s_{v_j}\right) \geq \Im\left(\s_{v_k}\right)$ for all $j \leq k$, where $\s_{v_j}$ is the quasinormal mode frequency of $v_j$.
Given any $l \in \R$, let $N \in \N$ be the smallest number such that $\Im\left(\s_{v_{N+1}}\right) < l$.
Then, for any $f \in C_c^\infty(M)$, there are unique constants $c_1, \hdots, c_N \in \C$, such that the unique backward propagated solution $u \in C^\infty(M)$ to $\P u = f$ satisfies
\begin{equation} \label{eq: asym exp estimate}
	\left| \left( t \d_t \right) ^k \left( t^{1-p}\d_x \right)^\a \left( u
		- \sum_{j = 1}^N c_j v_j \right) (t, x) \right|
			\leq C_{k,\a, l} t^l,
\end{equation}
for all $k \in \N_0$ and $\a \in \N_0^n$ and some $C_{k, \a, l} > 0$, and all $(t, x) \in J^+(\gamma)$.
\end{thm}

In particular, we may always choose $l = \Im \left( \s_{v_1} \right) + \epsilon$, and conclude that
\begin{equation} \label{eq: estimate without as exp}
	\left| \left( t \d_t \right) ^k \left( t^{1-p}\d_x \right)^\a u (t, x) \right|
			\leq C_{k,\a, \epsilon} t^{\Im \left( \s_{v_1} \right) + \epsilon}
\end{equation}
for all $k \in \N_0$, $\a \in \N_0^n$ and $\epsilon > 0$, and some $C_{k, \a, \epsilon} > 0$, and all $(t, x) \in J^+(\gamma)$. 
For \emph{geometric operators}, see Remark~\ref{rmk: geometric wave} and Definition~\ref{eq: x commutators} below, we may use this to get improved derivative estimates. 
Indeed, since $\P \d_{x_j} u = \d_{x_j} \P u = \d_{x_j} f \in C_c^\infty(M)$, the estimate \eqref{eq: estimate without as exp} also holds for $\d_{x_j} u$ in place of $u$ with some new constant $\tilde C_{k, \a, \epsilon} > 0$. 
The following corollary is therefore an immediate consequence of Theorem~\ref{thm: main as exp}.
We use the notation $\left( \d_x \right)^\a := \left( \d_{x_1}\right)^{\a_1} \hdots \left( \d_{x_n}\right)^{\a_n}$ for any $\a \in \N_0^n$.

\begin{cor}[Derivative estimates for geometric wave equations] \label{cor: derivative estimates}
Let $(M, g)$, $\P$ and $v_1$ be as in Theorem~\ref{thm: main as exp}, with the additional assumption that $[\d_{x_j}, \P] = 0$ for $j = 1, \hdots, n$.
Then, for any $f \in C_c^\infty(M)$, the unique backward propagated solution $u \in C^\infty(M)$ to $\P u = f$ satisfies
\[
	\left| \left( t \d_t \right) ^k \left( \d_x \right)^\a u (t, x) \right|
			\leq C_{k,\a, \epsilon} t^{\Im \left( v_1 \right) + \epsilon},
\]
for all $k \in \N_0$ and $\a \in \N_0^n$, any $\epsilon > 0$, some $C_{k, \a, \epsilon} > 0$, and all $(t, x) \in J^+(\gamma)$.
\end{cor}

It is desirable to show for example that $\Im(\s) \leq 0$ for all QNM frequencies $\s$, in which case Corollary~\ref{cor: derivative estimates} also provides derivative estimates.
Such statements are known in the black hole literature as \emph{mode stability}, see for example \cites{W1989, S2015, AMPW2017, CTC2021, H2025_2, PV2026}.

\begin{remark}[Asymptotic expansions and derivative estimates]
We claim that one cannot in general (even for geometric wave equations) remove the $t^{1-p_j}$-factors from the $\d_{x_j}$-derivatives in \eqref{eq: asym exp estimate}.
Indeed, consider the following example.
Let $p_1 = \hdots = p_n = 0$, i.e.~$(M,g)$ is the future half of the Minkowski spacetime.
Then $\Box = \d_t^2 - \sum_{j = 1}^n \d_{x_j}^2$ and $\T = - t\d_t - \sum_{j = 1}^n x_j \d_{x_j}$.
It follows that $v_1 := 1$ and $v_2 := x_1$ are QNMs with QNM frequency $0$ and $-i$, respectively.
All other QNM frequencies have smaller or equal imaginary parts. 
Defining $u := \chi \left( 1 + x_1 \right)$, for some $\chi \in C^\infty(M)$ with $\chi(t, x) = 0$ for $t \geq 2$ and $\chi(t, x) = 1$ for $t \leq 1$, we note that $\Box u = [\Box, \chi] (1 + x_1)$, which has support where $1 \leq t \leq 2$.
By finite speed of propagation, $u$ coincides in $J_2^+(\gamma) = \{ \abs x \leq t \leq 2\}$ with the unique backward propagated solution $\tilde u \in C^\infty(M)$ to $\Box \tilde u = f$, for any $f \in C_c^\infty(M)$ that coincides with $[\Box, \chi](1+x_1)$ in $J^+(\gamma)$.
We are thus precisely in the setting of Theorem~\ref{thm: main as exp}.
Moreover, in $J_1^+(\gamma) = \{ \abs x \leq t \leq 1\}$, $u - v_1 = x_1$, which implies that 
\[
	\abs{(u - v_1)(t, x)}
		\leq C t
\]
for all $(t, x) \in J^+(\gamma)$.
However, in $J_1^+(\gamma)$, $\d_{x_1} \left( u - v_1 \right) = \d_{x_1}x_1 = 1$, so 
\[
	\abs{\d_{x_1}(u - v_1)(t, x)} 
		\nleq C t^\de
\]
for all $(t, x) \in J^+(\gamma)$, for any $C > 0$ and $\de > 0$.
This shows the claim.
Note also that this does not contradict Corollary~\ref{cor: derivative estimates}, which is even true with $\epsilon = 0$ in this example.
\end{remark}

\subsubsection{Relation to previous literature}

The most general setting for linear wave equations on Big Bang spacetimes is considered by Ringstr\"om in \cites{R2020, R2025} (see also related results in \cites{P2016, R2019, AFF2019,FG2025, R2026, S2025}, and the sharp scattering result for the scalar wave equation by Warren Li in \cite{L2024}).
In \cite{R2020}, Ringstr\"om obtains optimal energy estimates without loss of derivatives for systems of wave equations on (in particular) Kasner spacetimes.
He also obtains the leading-order asymptotics, which for e.g.~for the scalar wave equation are
\begin{equation} \label{eq: R asymptotics}
	u(t, x) 
		\sim a(x) \ln(t) + b(x),
\end{equation}
for some functions $a$ and $b$, as $t \to 0$ in a non-flat Kasner spacetime. 
In \cite{R2025}, Ringstr\"om localizes these asymptotics to only require information in a region of the form $J^+(\gamma)$ (and thereby allow for a general class of spacetimes).
Asymptotics analogous to \eqref{eq: R asymptotics} for the linearized Einstein-scalar field equations in the subcritical range, which is known to be stable by \cites{FRS2023, OGPR2023, FR2026}, were obtained by Warren Li in \cite{L2024}.
Ringstr\"om's method was later iterated by Franco-Grisales in \cite{FG2025} in order to higher-order improvements of e.g.~\eqref{eq: R asymptotics} of the form
\begin{equation} \label{eq: FG asymptotics}
	u(t, x) 
		\sim a(x) \ln(t) + b(x) + \sum_{j = 1}^n t^{2(1-p_j)} \left( \a_j(x) + \b_j(x) \right) + \hdots,
\end{equation}
where the $\a_j$ and $\b_j$ (and the corresponding higher-order coefficients) in principal can be computed from $a$ and $b$.
The main difference between the asymptotics obtained in \eqref{eq: R asymptotics} and \eqref{eq: FG asymptotics} and our Theorem~\ref{thm: main as exp} is that $a$ and $b$ lie in an \emph{infinite-dimensional} space, whereas our asymptotic terms (for a given decay order in $J^+(\gamma)$) come from a \emph{finite-dimensional} space of generalized QNMs. 
The asymptotic information we obtain in Theorem~\ref{thm: main as exp} therefore gives more precise information.
This is the key that makes it possible to explicitly characterize all unstable perturbations of the linearized Einstein equations in non-flat Kasner spacetimes, Theorem~\ref{thm: main linearized Einstein}.

In recent years, there has been a fascinating development of the natural scattering problem for Einstein-scalar field solutions with a Big Bang in the \emph{subcritical regime}.
This means that the Kasner exponents, or more generally the eigenvalues of the expansion-normalized Weingarten map, satisfy the subcriticality condition \eqref{eq: subextremality} for all $j, k, l = 1, \hdots, n$ with $j \neq k$.
(Recall that the subextremality condition is not compatible with vacuum spacetimes of dimension $n + 1 = 4$.)
In modern terms, the scattering question in the subcritical regime for the Einstein-scalar field equations is the following:
\begin{enumerate}
\item Given any initial data on the Big Bang singularity in the sense of Ringstr\"om \cite{R2026_2}, does there exist a unique (up to isometry) maximal globally hyperbolic development?
\item What is the class of initial data at a Cauchy hypersurface (as opposed to the singularity) whose maximal globally hyperbolic development contains a Big Bang singularity in the sense of Ringstr\"om \cite{R2026_2}?
\end{enumerate}
Through the efforts of many contributors, both these questions are to a large extent resolved in the subcritical regime.
The pioneering work on question (1) is by Andersson and Rendall in \cite{AR2001} under the assumption of real analyticity.
Analyticity was removed by Fournodavlos and Luk in the vacuum setting \cite{FL2023}.
This was extended in \cite{NRT2020} to the Einstein-scalar field equations in the subcritical setting.
All these three results were significantly generalized in \cite{FG2026}, which in particular allows for either situation at different parts of the initial hypersurface.
The result in \cite{FL2023} was localized in \cite{AF2025}.
See also \cites{ABIL2013, ABIL2017} for related analysis.
The pioneering work for question (2) without symmetry or analyticity in the subcritical regime is a sequence of results by Rodnianski and Speck in \cites{RS2018, RS2018_2, R2022} and Speck in \cite{S2018}, which deals with near isotropic Einstein-scalar field solutions. 
The first breakthrough for the full subcritical regime is the work of Fournodavlos, Rodnianski and Speck in \cite{FRS2023}, already mentioned above.
Fajman and Urban extended this to negative spatial curvature in \cites{FU2025, FU2026}.
See also \cites{BO2024, AHS2025, BIGO2026} for further extensions.
In \cite{OGPR2023}, Oude Groeniger, Ringstr\"om and the author identified general conditions on initial data that produce a quiescent Big Bang singularity, without any reference to a background solution (e.g.~Kasner-scalar field spacetime).
Under these conditions, Franco-Grisales and Ringstr\"om improved on \cite{OGPR2023} by obtaining complete data on the Big Bang, nicely giving a general answer to (2).
An interesting further development are the localized results by Beyer and Oliynyk in \cite{BO2024_2} and Beyer, Oliynyk and Zheng in \cite{BOZ} and the recent improvement in the framework of \cite{OGPR2023} by Franco-Grisales in \cite{FG2026_2}.

By the result mentioned above, problem (1) is actually well understood even for vacuum spacetimes.
However, as opposed to the subcritical regime, general vacuum Big Bang spacetimes are not expected to be stably \emph{quiescent}, which makes it remarkably difficult to answer question (2) in the vacuum setting.
The expectation by the celebrated heuristic analysis by Belinski\v{\i}, Khalatnikov and Lifschitz in \cites{BKL1970, BKL1982} and the significant advances in the homogeneous setting in e.g.~\cites{R2000, R2001, B2010, LHWG2011, BD2023}, is that small perturbations of quiescent vacuum spacetimes typically will cause an oscillating behaviour for the metric, where each observer towards the Big Bang would experience their own chaotic gravitational dynamics.
There is also important results supporting this under Gowdy symmetry, see~\cites{R2009, L2024_2}, and polarized $U(1)$ symmetry, see \cites{FRS2023, D2026}. 
To the best of our knowledge, the present paper is the first in which the asymptotics for general linear gravitational perturbation of a vacuum Big Bang spacetime are derived without symmetry or analyticity, indicating on a linear level the instabilities predicted by Belinski\v{\i}, Khalatnikov and Lifschitz.

The methods of this paper rely on the microlocal Fredholm analysis framework for non-elliptic operators put forth by Vasy in \cite{V2013}.
See the outline of the proof in the next section for an explanation in some detail.
Linear wave equations in the \emph{isotropic} case, i.e.~when all Kasner exponents in a Kasner-type spacetime coincide, can be treated directly with the methods of Vasy in \cite{V2013}, even though this has not been observed before.
There is, however, no isotropic Kasner spacetime, so novel analysis in the anisotropic setting is necessary for the Einstein vacuum equations.
One of the new challenges is that the bicharacteristic flow does not have the source/sink structure (see the next section for more details), but instead a more involved saddle point structure (see Section~\ref{subsec: geometry}).
Since the work of Vasy in \cite{V2013}, various type of saddle point estimates have been developed, see e.g.~the work by Hintz and Vasy in \cite{HV2015}*{Sec.~2}, in the b-calculus setting, and the recent work by Hintz, Vasy and the author in \cite{HPV2025}*{Sec.~4}, in the semiclassical b-calculus setting.
The saddle point estimates we need are proven in Appendix~\ref{sec: saddle point estimates}.

To the best of our knowledge, quasinormal modes have not been defined in this way for Kasner-type spacetimes before.
However, QNMs are the fundamental objects in the analysis of black holes spacetimes and a vital part of astrophysics and scattering theory.
The study of black hole stability has made quasinormal modes important also in mathematical general relativity.
In addition to the references already mentioned above, see for example \cites{W1989, SZ1997, BH2008, V2010, V2013, S2015, AMPW2017, HV2018_1, HV2018_2, HHV2019, HHV2019_2, DHRT2021, CTC2021, PV2023, PV2024, PV2025, H2025_2, H2026, PV2026} and references therein.

\subsection{Outline of the proof}

The proof of our main results, Theorem~\ref{thm: main linearized Einstein, general dim} and Theorem~\ref{thm: main linearized Einstein}, consists of two main steps:
\begin{enumerate}
	\item[Step 1] Prove the asymptotic expansion of solutions to linear wave equations in QNMs.
	This is Theorem~\ref{thm: main as exp} and is proven Section~\ref{sec: asymptotic expansion}.
	\item[Step 2] Compute all unstable generalized QNMs for the linearized Einstein vacuum equation in a non-flat Kasner spacetime.
	This is Theorem~\ref{thm: self-similar solutions} and is proven in Section~\ref{sec: QNM lin Ein}, based on the theory obtained in Section~\ref{sec: QNM geometric waves}.
\end{enumerate}

\textbf{Strategy for Step 1:}
The main novelty in Step 1 is the introduction of QNMs near Big Bang and the microlocal analysis of the corresponding spectral problem in an anisotropic spacetime.
The outline for Step 1 is the following:
\begin{enumerate}
	\item We analyze the light propagation (the bicharacteristic flow) in a neighbourhood of $J^+(\gamma)$.
	This is needed for propagation of singularities techniques in microlocal analysis. 
	See Theorem~\ref{thm: saddle point flow} for the saddle point structure of the bicharacteristic flow near the equilibrium points and Theorem~\ref{thm: bicharacteristic flow} for the global structure of the bicharactieristic flow.
	\item Fourier transforming the wave equation along the integral curves of $\T$, which corresponds to formally replacing $\T$ by the frequency $-i \s$ in the linear wave operators, gives a spectral family of operators on the spacelike hypersurface given by e.g.~$t = 1$.
	We show in Theorem~\ref{thm: Fredholm property} that this gives an analytic family of Fredholm operators.
	This leads to the meromorphic extension of the resolvent, see Corollary~\ref{cor: meromorphic extension}, which in turn is used to construct the QNMs in the subsequent Corollary~\ref{cor: QNM freq}.
	This step follows the recipe from Vasy's analysis on Kerr-de Sitter spacetimes in \cite{V2013}, with the addition of microlocal saddle point estimates.
	\item In order to prove Theorem~\ref{thm: main QNM} and Theorem~\ref{thm: main as exp}, we need high-energy estimates in the sense of \cite{V2013} for the resolvent, i.e.~uniform estimates in regions where $\abs{\Re(\s)} \to \infty$, using semiclassical saddle point estimates.
	This is achieved in Theorem~\ref{thm: high-energy estimate}, using the full knowledge about the bicharacteristic flow.
\end{enumerate}
Let us look closer how we deal with the geometry in Kasner-type spacetimes.
It turns out that there is another set of coordinates which are much better adapted to the homothetic Killing vector field $\T$.
We introduce the new coordinates
\begin{equation} \label{eq: new coordinates}
	s
		:= t, \qquad 
	y_j
		:= \frac{x_j}{t^{1-p_j}},
\end{equation}
for $j = 1, \hdots, n$. 
The main point of these coordinates is that $\T = s\d_s$.
A Kasner-type metric takes the form
\[
	g
		= - \md t^2 + \sum_{j = 1}^n t^{2p_j}\md x_j^2
		= - \md s^2 + \sum_{j = 1}^n \left( \q_j y_j \md s + s \md y_j \right)^2,
\]
on the same manifold $M$, with new coordinates $(s, y_1, \hdots, y_n)$, and where $\q_j := 1 - p_j > 0$ for $j = 1, \hdots, n$.
See Figure~\ref{fig: Straight} for an illustration of how the $J^+(\gamma)$ has now been `straightened out' in the new coordinates.
\begin{figure*} \label{fig: Straight}
  \begin{center}
    \includegraphics[scale = 0.8]{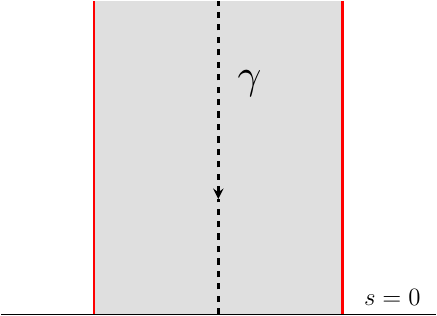}
  \end{center}
  \caption{
  	The gray region denotes the causal future $J^+(\gamma)$ of the observer $\gamma$ with respect to the new coordinates $(s, y_1, \hdots, y_n)$. The red curves illustrating the \emph{particle horizon} of the observer have been straightened out.
  	}
\end{figure*}
We will work with the rescaled metric,
\begin{equation} \label{eq: g rescaled}
	\g
		:= s^{-2} g
		= \left( 1- \sum_{j = 1}^n \q_j^2 y_j^2 \right)^2 \left( \frac{\md s}s \right)^2 + 2 \sum_{j = 1}^n \q_j y_j \md y_j \otimes_s \frac{\md s}s + \sum_{j = 1}^n \md y_j^2
\end{equation}
This is a b-metric, generalizing the set-up of \cite{V2013}, see specifically \cite{HV2015}*{p.~1811}.
Note that $s \d_s$ is a Killing vector field with respect to $\g$ and it is timelike if $\sum_{j = 1}^n \q_j^2 y_j^2 < 1$, lightlike if $\sum_{j = 1}^n \q_j^2 y_j^2 = 1$, and spacelike otherwise.

\textbf{Problem 1:}
\emph{There is an ergoregion-like region near the Big Bang singularity in anisotropic Kasner-type spacetimes.} 

Recall that in the Kerr(-anti de Sitter/de Sitter) spacetime, there is a small region beyond the event horizon where the stationary Killing vector field is \emph{spacelike}.
This makes the analysis of linear wave equations significantly more involved than for Schwarzschild(-anti de Sitter/de Sitter) spacetime, where the stationary Killing vector field is lightlike precisely at the horizon, timelike outside and spacelike inside.
For Kasner-type spacetimes, the situation is similar.
For \emph{isotropic} Kasner-type spacetime,  rescaled as in \eqref{eq: g rescaled}, the Killing vector field $s\d_s$ is timelike everywhere in $J^+(\gamma)$ and lightlike precisely on the boundary of $J^+(\gamma)$.
However, if not all $p_1, \hdots, p_n$ coincide, then there is a region inside $J^+(\gamma)$ where $s \d_s$ is \emph{spacelike}.
Like for rotating black holes, this significantly complicates the analysis, and requires a method that can handle an ergoregion-like region.
Microlocal analysis is particularly suitable in such a situation, since it only requires asymptotic knowledge about the light propagation (the bicharacteristic flow). 

\textbf{Problem 2:}
\emph{The bicharacteristic flow does not have a source/sink structure.}

It is only in isotropic Kasner-type spacetimes that the bicharacteristic flow has a source/sink structure, like in the Kerr-de Sitter spacetime (which instead is complicated by trapping). 
This implies that the source/sink propagation estimates that are used for Kerr-de Sitter spacetimes, and similar situations, cannot be applied here.
Instead, we have to carefully analyze the saddle point structure and apply the saddle point estimates proven in Appendix~\ref{sec: saddle point estimates}.
A difficulty related to this is that there seems to be fewer monotone quantities for the bicharacteristic flow on Kasner-type spacetimes than on Kerr-de Sitter spacetimes, c.f.~\cites{V2013, PV2025}.
We instead need to find appropriate asymptotically monotone quantities for the bicharacteristic flow, see Section~\ref{subsec: geometry} in general and Lemma~\ref{le: A min max} in particular.

\textbf{Strategy for Step 2:}
We define \emph{geometric wave operators} in Section~\ref{sec: QNM geometric waves} to be the linear wave operators $\P$ that share all symmetries of the spacetime, i.e.
\begin{equation} \label{eq: commutation rules}
	[\T, t^2 \P] = 0, \qquad [\d_{x_j}, \P] = 0,
\end{equation}
for $j = 1, \hdots, n$.
These assumptions are natural, as they will be satisfied for any linear wave equation associated with the geometry of a Kasner-type spacetime.
For this paper, the most important such operator is the Lichnerowicz-d'Alembert operator $\Box_L$, which is the main object of study in Section~\ref{sec: QNM lin Ein}.
In Section~\ref{sec: QNM geometric waves}, we develop an algorithm on how to compute the generalized QNMs for geometric wave operators, that we outline here:
Let $\P$ be a system of linear wave equations, which satisfies \eqref{eq: commutation rules}.
Assume now that $u \in C^\infty(M)$ is a quasinormal mode for $\P$, i.e.~that $\P u = 0$ and $\T u = - i \s u$, for some $\s \in \C$.
Then the key is that $\P \d_{x_j} u = \d_{x_j} \P u + [\P, \d_{x_j}] u = 0$ and $\T \d_{x_j} u = [\T, \d_{x_j}] u + \d_{x_j} \T u = (1-p_j) \d_{x_j} u - i \s u = -i \left( \s + (1-p_j)i \right) u$.
Consequently, $\d_{x_j} u$ is a new QNM for $\P$ with the frequency $\s + (1-p_j)i$.
Iterating this several times, we may produce a QNM of \emph{arbitrary} growth as $t \to 0$.
But a (very rough) energy estimate forbids arbitrary growth, and we conclude that $u$ must vanish after taking a finite amount of $x$-derivatives.
Consequently, any generalized QNM must be a polynomial in $x$.
Moreover, the condition $[\T, t^2\P] = 0$ and the defining property of generalized QNMs, $\left(\T + i\s \right)^k u = 0$, characterizes also the possible time-dependence of the coefficients, see Corollary~\ref{cor: structure of QNMs}.
In this context, we therefore call $\d_{x_j}$, mapping (generalized) QNMs to (generalized) QNMs, the \emph{annihilation} operator. 
Choosing a consistent way of finding a primitive function in $x_j$, we also obtain the \emph{creation} operator, c.f.~Theorem~\ref{thm: creation and annihilation}.
There are two immediate conclusions from this observation:
\begin{enumerate}
	\item The fastest growing QNMs to a linear geometric wave equation will be independent of $x$. 
	Indeed, the fastest growing QNMs necessarily lie in the kernel of annihilation operator $\d_{x_j}$, for otherwise there would be a faster growing QNM, see Corollary~\ref{cor: structure of QNMs}.
	This can be seen as the manifestation of the \emph{asymptotically velocity term dominated (AVTD)} behaviour of solutions in this context.
	\item Every generalized quasinormal mode is a linear combination of terms produced by the creation operator of the $x$-independent solutions.
	Using this, we characterize all possible QNM frequencies in Corollary~\ref{cor: QNM frequencies}.
\end{enumerate}
The scalar linear wave equation in Kasner spacetimes is given by 
\[
	t^2 \Box 
		= \left( t\d_t \right)^2 - \sum_{j = 1}^n t^{2(1-p_j)} \d_{x_j}^2.
\]
If we restrict to $x$-independent solutions, we end up with the ODE $(t\d_t)^2 u(t) = 0$, which has the general solution $u(t) = c_1 + c_2 \log(t)$.
This solutions dictates the decay for general solutions, as all other quasinormal modes will decay faster, see Corollary~\ref{cor: structure of QNMs} and Example~\ref{ex: scalar wave}.
Analyzing the linearized Einstein equation is significantly more involved, but the basic idea is the same.
We compute the $x$-independent QNMs in Theorem~\ref{thm: AVTD system}.
These need to be carefully analyzed to distinguish what solutions are guage solutions (or gauge related), even after applying the creation operator.
We refer to Section~\ref{sec: QNM lin Ein} for the details.

\subsubsection*{Acknowledgements}

The author would like to thank Andres Franco-Grisales, Hans Oude Groeniger, Alberto Richtsfeld, Hans Ringstr\"om and Andr\'as Vasy for helpful discussions, and gratefully acknowledges support from the Swedish Research Council under grant number 2021-04269 and from Knut och Alice Wallenbergs Stiftelse under grant number KAW 2021.0239.

\section{Quasinormal modes at the Big Bang} \label{sec: asymptotic expansion}

The goal of this subsection is to prove Theorem~\ref{thm: main QNM} and Theorem~\ref{thm: main as exp}.
Let $(M, g)$ be a Kasner-type spacetime, see Definition~\ref{def: Kasner-type spacetime}, and let $(s, y_1, \hdots y_n)$ be the coordinates introduced in \eqref{eq: new coordinates}.
We will often work with the conformally rescaled metric $\g$, introduced in \eqref{eq: g rescaled}.
Though we use Vasy's framework introduced in \cite{V2013}, we will mostly refer to the detailed exposition by Hintz in \cite{H2025} for the background on microlocal analysis that we will need.
To simplify the comparison with \cite{H2025}*{Sec.~11.1}, we introduce the coordinate $\tau := - \log(s)$, and obtain
\begin{align*}
	\g
		&= - \md \tau^2 + \sum_{j = 1}^n \left( - \q_j y_j \md \tau + \md y_j \right)^2
\end{align*}
on the manifold $\R^{n+1}$ with coordinates $\tau, y_1, \hdots, y_n$.
The Big Bang is now located at $\tau = \infty$.
We recall that $q_j := 1 - p_j$ for $j = 1, \hdots, n$.
The metric $\g$ thus generalizes \cite{H2025}*{Eq.~(11.1)} in which $\q_j = 1$, or equivalently $p_j = 0$, for all $j = 1, \hdots, n$.
Moreover, this generalizes the set-up in \cite{HV2015}*{p.~1811}, which would require all $\q_j$ to coincide - that is the \emph{isotropic} case. 
The dual metric is given by
\begin{equation} \label{eq: dual metric normalized}
	\g^{-1}
		= - \left( \d_\tau + \sum_{j = 1}^n \q_j y_j \d_{y_j} \right)^2 + \sum_{j = 1}^n \d_{y_j}^2.
\end{equation}

\subsection{Geometry of Kasner-type spacetimes}
\label{subsec: geometry}

The goal of this subsection is to compute the asymptotic behaviour of all lightlike geodesics of $\g$, which are just reparametrizations of lightlike geodesics of $g$, from the point of view of one observer, using the expression for the dual metric $\g^{-1}$ in \eqref{eq: dual metric normalized}.

\subsubsection{The Hamiltonian vector field}

Recall from e.g.~\cite{H2025}*{Prop.~9.9} that geodesics are given by the projections of the integral curves of the Hamiltonian vector field $H_G$, where $G$ is the Hamiltonian defined as $G(\zeta) := \g^{-1}(\zeta, \zeta)$, for any $\zeta \in T^*M$.
More specifically, since we are interested in the null geodesics, we will restrict the $H_G$ to the set of non-zero lightlike co-vectors
\[
	\S^\circ
		:= \S^\circ_+ \sqcup \S^\circ_-, \qquad \S^\circ_\pm := \left\{ \zeta \in T^*\R^{n+1} \mid G(\zeta) = 0, \ \pm \g^{-1}(\zeta, \md \tau) > 0 \right\}.
\]
Note here that $\md \tau$ is timelike, ensuring that $\S_\pm^\circ$ are connected.
Following the sign convention in \cite{H2025}*{Sec.~11.1}, let us define $\s$ and $\xi_j$ by $\zeta = - \s \md \tau + \sum_{j = 1}^n \xi_j \md y_j$, so that
\begin{equation}\label{eq: G definition}
	G(\zeta)
		= - \left( \s - \sum_{j = 1}^n \q_j y_j \xi_j \right)^2 + \sum_{j = 1}^n \xi_j^2.
\end{equation}
The Hamiltonian vector field is given by
\begin{align}
	\frac12 H_G
		= & \ \frac12 \left( - \left( \d_\s G \right) \d_\tau + \left(\d_\xi G \right) \cdot \d_y + \left( \d_\tau G \right) \d_\s - \left( \d_y G \right) \cdot \d_\xi \right) \nonumber \\
		= & \ \left( \s - \sum_{j = 1}^n \q_j y_j \xi_j \right) \d_\tau + \sum_{k = 1}^n \left( \left( \s - \sum_{j = 1}^n \q_j y_j \xi_j \right)\q_k y_k + \xi_k \right) \d_{y_k} \nonumber \\*
		& \ - \left( \s - \sum_{j = 1}^n \q_j y_j \xi_j \right) \sum_{k = 1}^n \q_k \xi_k \d_{\xi_k}. \label{eq: Hamiltonian vector field}
\end{align}
By definition, $0 < \pm \g^{-1}(\zeta, \md \tau) = \pm \left( \s - \sum_{j = 1}^n \q_j y_j \xi_j \right)$ in $\S_\pm^\circ$, which implies using $G(\zeta) = 0$ that
\begin{equation} \label{eq: lightlike comp}
	\s - \sum_{j = 1}^n \q_j y_j \xi_j 
		= \pm \abs{\xi}
\end{equation}
in $\S_\pm^\circ$.
Since there is no $\d_\s$-component in the Hamiltonian vector field (since $\d_\tau G = 0$), it follows that $\s$ is constant along the Hamiltonian flow.
Moreover, the $\tau$-values of each integral curve can be recovered by integrating $\s - \sum_{j = 1}^n \q_j y_j \xi_j$. 
It therefore suffices to study the flow of the vector field 
\begin{equation} \label{eq: Ham sigma}
	\frac12 H_\s
		:= \sum_{k = 1}^n \left( \left( \s - \sum_{j = 1}^n \q_j y_j \xi_j \right)\q_k y_k + \xi_k \right) \d_{y_k} - \left( \s - \sum_{j = 1}^n \q_j y_j \xi_j \right) \sum_{k = 1}^n \q_k \xi_k \d_{\xi_k}
\end{equation}
on the set
\[
	\S^\circ_\s
		:= \S^\circ_{\s,+} \sqcup \S^\circ_{\s, -}, \qquad \S^\circ_{\s, \pm} := \left\{ (y, \xi) \in T^*\R^n \mid (0, y, \s, \xi) \in \S^\circ_\pm \right\},
\]
where we have identified $\{0\} \times \R^n \cong \R^n$.
Note that $\Sigma_\s^\circ$ is homogeneous in the fibers of $T^* \R^{n+1}$ if and only if $\s = 0$.
We refer to \cite{H2025}*{p.~332--333} for a geometric interpretation of these sets when $\q_1 = \hdots = \q_n = 1$.

In order to formulate the next lemma, we group all $\q_1, \hdots, \q_n$ that coincide.
Let
\[
	0 < \check \q_1 < \hdots < \check \q_m
\]
be such that each $\q_j$ is equal to precisely one $\check \q_l$.
Consequently, $m \leq n$.
In the maximally anisotropic case, where all $\q_j$ are distinct, then $m = n$.
In the isotropic case, where all $\q_j$ coincide, then $m = 1$.
For each $l = 1, \hdots, m$, we define
\begin{align*}
	\L_{l, \pm}
		:= \left\{ \frac{\xi_j}{\abs \xi} = \mp \q_j y_j \mid j: \q_j = \check \q_l \right\} \cap \{ y_j = \xi_j = 0 \mid j: \q_j \neq \check \q_l\} \subset \S_{0, \pm}^\circ.
\end{align*}

\begin{lemma}[Radial points for $H_0$.] \label{le: radial points}
The set of points in $\S_{0, \pm}^\circ$ at which the vector field $H_0$ is radial is given by $\L_\pm := \cup_{l = 1}^m \L_{l, \pm}$, and
\[
	\frac1{2\abs \xi}H_0|_{\L_{l, \pm}}
		= \mp \check \q_l \sum_{k: \q_k = \check \q_l}^n \xi_k \d_{\xi_k}
\]
for $l = 1, \hdots, m$.
\end{lemma}
\begin{remark}
Note that $\sum_{j = 1}^n \q_j^2 y_j^2|_{\L_\pm} = 1$, which implies that $\d_\tau$ is lightlike at $\L_\pm$.
In the isotropic case, when all $\q_j =: \q$ coincide, then $\L_+ \cup \L_-$ is the \emph{conormal bundle} of the particle horizon (or cosmological horizon in the case of the de Sitter spacetime, when $\q = 1$) given by $\sum_{j = 1}^n y_j^2 = \q^{-2}$.
\end{remark}
\begin{proof}
If $H_0$ is radial at a point $(y, \xi)$, then there is a $c \in \R$, such that $\sum_{k = 1}^n \q_k \xi_k \d_{\xi_k} = c \sum_{k = 1}^n \xi_k \d_{\xi_k}$, which is equivalent to $\sum_{k = 1}^n \left( \q_k - c \right) \xi_k \d_{\xi_k} = 0$.
It follows that $c = \q_l$ for some $l$ and therefore $\xi_j = 0$ for all $j$ such that $\q_j \neq \q_l$.
Moreover, if the vector field $H_0$ is radial at a point $(y, \xi) \in \S_{0, \pm}^\circ$, then the vanishing of the $\d_{y_k}$-components in \eqref{eq: Ham sigma} and \eqref{eq: lightlike comp} imply that
\[
	\xi_k
		= \q_k y_k \sum_{j = 1}^n \q_j y_j \xi_j
		= \mp \q_k y_k \abs{\xi},
\]
for any $k = 1, \hdots, n$.
Since $\abs \xi \neq 0$ in $\S_0^\circ$, it follows that $y_k = 0$ for all $k$ such that $\q_k \neq \q_l$.
Conversely, it is clear that $H_0$ is radial at all such points, and given by the asserted expression.
\end{proof}

The next step is to extend the set of lightlike vectors, as well as the Hamiltonian vector field, to the \emph{radial compactification} of the cotangent bundle $\oT \R^n$, see for example \cite{H2025}*{Sec.~6.5}, and its boundary $S^*\R^n := \d \oT \R^n$, which we call \emph{fiber infinity}.
We will use the following notation for the closures inside the radial compactification of the cotangent bundle:
\begin{align*}
	\S_0
		:= & \ \overline {\S^\circ_0} \cap \left( \oT \R^n \backslash o \right), && 
	\S_{0, \pm}
		:= \ \overline{\S_{0,\pm}^\circ} \cap \left( \oT \R^n \backslash o \right), \\
	\S_\s
		:= & \ \overline {\S_\s^\circ} \subseteq \oT \R^n, && 
	\S_{\s, \pm}
		:= \ \overline{\S_{\s,\pm}^\circ} \subseteq \oT \R^n, && \s \neq 0.
\end{align*}
One readily checks that $\d \S_{0, \pm} = \S_{0, \pm} \cap S^*\R^n = \S_{\s, \pm} \cap S^*\R^n = \d \S_{\s, \pm}$ for all $\s \in \R$.
We define the \emph{radial sets} by
\[
	\cR_{l, \pm}
		:= \overline{\L_{l, \pm}} \cap S^*\R^n
		\subset \d \S_{\s, \pm},
\]
for any $l = 1, \hdots, m$, and any $\s \in \R$.
Let us now extend the Hamiltonian vector field $H_\s$ smoothly to fiber infinity $S^*\R^n$.
Since we only care about the lightlike geodesics, it suffices to extend $H_\s$, defined in $\S_\s^\circ \subset T^*\R^n$, to $\S_\s \subset \oT \R^n$.
We therefore define
\begin{equation}\label{eq: Ham vf}
	\pm \mathsf H_\s|_{\S_{\s,\pm}^\circ}
		:= \pm \frac1{2\abs\xi} H_\s|_{\S_{\s,\pm}^\circ} 
		= \sum_{k = 1}^n \left( \q_k y_k \pm \frac{\xi_k}{\abs \xi} \right) \d_{y_k} - \sum_{k = 1}^n \q_k \xi_k \d_{\xi_k}
\end{equation}
which is the restriction of a vector field which is invariant under dilations in $\xi$ (even though $\S_\s$ is not) and therefore extends smoothly to all of $\S_\s$.
Since $\d \S_{\s, \pm} = \d \S_{0, \pm} \subset S^*\R^n$ for all $\s \in \R$ and the expression for $\mathsf H_\s|_{\S_{\s,\pm}^\circ}$ is independent of $\s$, we note that 
\[
	\mathsf H_\s|_{\d \S_{\s,\pm}}
		= \mathsf H_0|_{\d \S_{0,\pm}}.
\]
\begin{prop}[Stationary points for the Hamiltonian flow] \label{prop: Stationary points}
Let $\s \in \R$.
The (smoothly extended to fiber infinity) Hamiltonian vector field $\mathsf H_\s$, defined on $\S_{\s,\pm}$, vanishes precisely in the set 
\[
	\cR_\pm = \cup_{l = 1}^m \cR_{l, \pm}.
\]
\end{prop}
\begin{proof}
In the interior, $\S^\circ_{\s, +}$, the Hamiltonian vector field does not vanish, since the zero section is not in $\S^\circ_{\s, +}$.
At fiber infinity, $\d S_{\s, +}$, the Hamiltonian vector field vanishes precisely at the radial points.
The statement therefore follows by Lemma~\ref{le: radial points}.
\end{proof}

\subsubsection{The saddle points}

We now turn to analyzing the structure of the flow near the stationary points in $\cR_\pm$. 
For this, let us first note that
\begin{equation} \label{eq: negative of Hamiltonian VF}
	\mathsf H_\s|_{(y, \xi)}
		= - \mathsf H_{-\s}|_{(y, -\xi)}, \qquad 
	(y, \xi) \in \S_{\s, \pm} \Leftrightarrow (y, - \xi) \in \S_{- \s, \mp}.
\end{equation}
For simplicity of presentation, we therefore restrict to $(y, \xi) \in \S_{\s, +}$, in particular to $\cR_+$, when describing the saddle point structure.
The saddle point structure at $\cR_-$ is obtained by applying the above identification. 
We will use the notation
\[
	\rho 
		:= \frac1{\abs \xi}, \qquad 
	\hat \xi_j 
		:= \frac{\xi_j}{\abs \xi},
\]
for $j = 1, \hdots, n$. 
Note that these functions extend smoothly up to to fiber infinity $S^*\R^n$.
For a fixed $l \in \{1, \hdots, m\}$, define the functions
\begin{equation} \label{eq: Vl Wl}
\begin{split}
	V_l
		&:= \sum_{j : \q_j \leq \check \q_	{l-1}} \hat \xi^2_j + \sum_{j = 1}^n \left( \hat \xi_j + \left( 2\q_j - \check \q_l \right) y_j \right)^2 + \sum_{j : 2 \q_j = \check \q_l} \q_j^2 y_j^2, \\
	W_l
		&:= \sum_{j : \q_j \geq \check \q_{l+1}} \hat \xi_j^2
\end{split}
\end{equation}
defined on $S^*\R^n$.
Note that there might well be no $j$ such that $2 \q_j = \check \q_l$, in which case the last sum in $V_l$ is void; it is just there to cover for the potential degeneracy in the second last term in $V_l$.
\begin{remark} \label{rmk: quadratic def}
Both $V_l$ and $W_l$ vanish quadratically at $\cR_{l,+}$, but $\Hess \left(V_l + W_l\right)$ is positive semi-definite on $\d \S_{\s, +}|_{\cR_{l,+}}$, with kernel given by $T\cR_{l,+} \subset T\d \S_{\s, +}$.
In other words, $V_l + W_l$ is a \emph{quadratic defining function} for $\cR_{l,+}$ in $\d \S_{\s, +}$ (see \cite{H2025}*{Def.~10.3}).
\end{remark}

\begin{thm}[The saddle point structure] \label{thm: saddle point flow}
Let $l \in \{1, \hdots, m\}$ be fixed. 
Then there is a $c > 0$, such that
\[
	\mathsf H_\s V_l
		\geq c V_l + \Phi_l, \qquad
	\mathsf H_\s W_l
		\leq - c W_l + \Psi_l,
\]
in $\d \S_{\s,+}$, where $\Phi_l$ and $\Psi_l$ are smooth functions on $\d \S_{\s,+}$, which are cubically vanishing at $\cR_{l,+}$.
\end{thm}

The proof will be based on the following two lemmas:

\begin{lemma} \label{le: diagonal linearization}
For $k, l = 1, \hdots, n$, we have
\begin{align}
	\mathsf H_\s \rho
		&= \left(\check \q_l + A_l \right) \rho, \label{eq: Hp rho} \\
	\mathsf H_\s \hat \xi_k
		&= \left( \check \q_l - \q_k + A_l \right) \hat \xi_k, \label{eq: bich diagonal xi} \\
	\mathsf H_\s \left( \hat \xi_k + \left( 2\q_k - \check \q_l \right) y_k \right)
		&= \q_k \left( \hat \xi_k + \left( 2 \q_k - \check \q_l \right) y_k \right) + A_l \hat \xi_k, \label{eq: bich diagonal xi y}
\end{align}
in $\S_{\s, +}$, where 
\begin{equation}\label{eq: def A l}
	A_l
		:= \sum_{j : \q_j \neq \check \q_l} (\q_j - \check \q_l) \hat \xi_j^2.
\end{equation}
\end{lemma}

\begin{proof}
We first recall from \eqref{eq: Ham vf} that $\mathsf H_\s|_{\S_{\s,+}^\circ} = \sum_{j = 1}^n \left( \q_j y_j + \hat \xi_j \right) \d_{y_j} - \sum_{j = 1}^n \q_j \xi_j \d_{\xi_j}$, which smoothly extends to $\S_{\s,+}$.
The computations are straightforward.
\end{proof}

Since $A_l$ vanishes at $\cR_{+, l}$, Lemma~\ref{le: diagonal linearization} provides the \emph{linearization} of $\mathsf H_\s$ at $\cR_{+, l}$, unless there is a $k$, such that $\q_k = \frac12 \check \q_l$, in which case \eqref{eq: bich diagonal xi} and \eqref{eq: bich diagonal xi y} coincide. 
We therefore need a separate equation in this case:

\begin{lemma} \label{le: degeneration linearization}
If $\q_k = \frac12 \check \q_l$, for some $k,l \in \{1, \hdots, n\}$, then
\[
		\mathsf H_\s \left( \hat \xi_k^2 + \q_k^2 y_k^2 \right) 
		= \q_k \left( \hat \xi_k^2 + \q_k^2 y_k^2 \right) + \q_k \left( \hat \xi_k + \q_k y_k \right)^2 + 2 A_l \hat \xi_k^2
\]
in $\S_{\s, +}$, where $A_l$ is defined in \eqref{eq: def A l}.
\end{lemma}
\begin{proof}
Since $\mathsf H_\s|_{\S_{\s,\pm}^\circ} y_k = \q_k y_k + \hat \xi_k$ and $\check \q_l = 2 \q_k$, and by Equation \eqref{eq: bich diagonal xi}, 
\begin{align*}
	\mathsf H_\s \left( \hat \xi_k^2 + \q_k^2 y_k^2 \right)
		& = 2 \left( \q_k + A_l \right) \hat \xi_k^2 + 2 \q_k^3 y_k^2 + 2 \q_k^2 \hat \xi_k y_k \\
		& = \q_k \left( \hat \xi_k^2 + \q_k^2 y_k^2 \right) + \q_k \left( \hat \xi_k + \q_k y_k \right)^2 + 2 A_l \hat \xi_k^2,
\end{align*}
in $\S_{\s, +}$, proving the assertion.
\end{proof}

\begin{proof}[Proof of Theorem~\ref{thm: saddle point flow}]
We choose $c := \min \left(\q_1, \check \q_l - \check \q_{l-1}, \check \q_{l+1} - \check \q_l \right) > 0$.
We start with the computation for $W_l$.
By Lemma~\ref{le: diagonal linearization}, $\mathsf H_\s W_l = 2 \sum_{j : \q_j \geq \check \q_{l+1}} \left( \check \q_l - \q_j + A_l \right) \hat \xi_j^2 \leq - c W_l + \Psi_l$, where $\Psi_l = -2A_l W_l$.
Next, we similarly compute, using Lemma~\ref{le: diagonal linearization} and Lemma~\ref{le: degeneration linearization},
\begin{align*}
	\mathsf H_\s V_l
		= & \ 2 \sum_{j : \q_j \leq \check \q_{l-1}, 2\q_j \neq \check \q_l} \left( \check \q_l - \q_j \right) \hat \xi^2_j 
		+ 2 \sum_{j = 1}^n \q_j \left( \hat \xi_j + \left( 2 \q_j - \check \q_l\right) y_j \right)^2 \\*
		& \ + \sum_{j : 2 \q_j = \check \q_l} \q_j \left( \q_j^2 y_j^2 + \hat \xi_j^2 \right) + \q_j \left( \hat \xi_j + \q_j y_j \right)^2 + \Phi_l \\
		\geq & \ c V_l + \Phi_l,
\end{align*}
where $\Phi_l = 2 A_l \left( \sum_{j = 1}^n \left( \hat \xi_j + \left( 2 \q_j - \check \q_l\right) y_j \right) \hat \xi_j + \sum_{j : \q_j \leq \check \q_{l-1}} \hat \xi^2_j \right)$.
\end{proof}

\subsubsection{The global behaviour of lightlike geodesics}

Proposition~\ref{prop: Stationary points} described the stationary points of the flow of lightlike geodesics and Theorem~\ref{thm: saddle point flow} described the flow near the stationary points.
The following theorem describes the global structure of the flow.

\begin{thm} \label{thm: bicharacteristic flow}
Let $\mathring \gamma \in \S_{\s, +}$ and consider the integral curve of $\mathsf H_\s$ with initial condition $\gamma(0) = \mathring \gamma$.
If $\mathring \gamma \in \cR_+$, then $\gamma(r) = \mathring \gamma$ for all $r \in \R$.
If instead $\mathring \gamma \in \S_{\s, +} \backslash \cR_+$, then the following holds:
\begin{enumerate}[(a)]
	\item Let $l_\max \in \{1, \hdots, m\}$ be the largest integer such that $\hat \xi_j (\mathring \gamma) \neq 0$ for some $j$ with $\q_j = \check \q_{l_\max}$.
	Then 
	\[
		\lim_{r \to - \infty} \gamma(r) \in \cR_{l_\max, +}.
	\]
	\item Let $l_\min \in \{1, \hdots, m\}$ be the smallest integer such that $\hat \xi_j(\mathring \gamma) \neq 0$ for some $j$ with $\q_j = \check \q_{l_\min}$. 
	Then \textbf{either}
	\[
		\lim_{r \to \infty} \gamma(r) \in \cR_{l_\min, +},
	\]
	\textbf{or} $\gamma$ intersects all spacelike hypersurfaces given by
	\[
		\mathcal C_a := 
			\left\{ \sum_{j = 1}^n \q_j y_j^2 = \frac1{\min_j \q_j} + a \right\} 
			\subset \R^n,
	\]
	for any $a > 0$, transversally and outward pointing, as $r \to \infty$.
	\label{item: future dichotomy}
\end{enumerate}
\end{thm}

The idea for the proof is to combine the identities in Lemma~\ref{le: diagonal linearization} and Lemma~\ref{le: degeneration linearization}, with the following exponential bounds.

\begin{lemma} \label{le: A min max}
Let $l_\min, l_\max, \gamma$ be as in Theorem~\ref{thm: bicharacteristic flow}.
There is a $C > 0$, such that
\[
	\left|A_{l_\max} \left(\gamma(r) \right)\right|
		\leq C e^{2 r}, \text{ for all } r \leq 0, 
	\qquad 
	\left|A_{l_\min} \left(\gamma(r) \right)\right|
		\leq C e^{-2 r}, \text{ for all } r \geq 0. 
\]
If $l_\min = l_\max$, then in fact $A_{l_\min} = A_{l_\max} = 0$.
Moreover, if $A_{l_\min}\left(\mathring \gamma \right) < 1$, or $A_{l_\min}\left(\mathring \gamma \right) < 1$, then $\abs{A_{l_\max}(\gamma(r))}$ is strictly growing for all $r \in (-\infty, 0)$ and $\abs{A_{l_\min}(\gamma(r))}$ is strictly decreasing for all $r \in (0, \infty)$. 
\end{lemma}

\begin{proof}
Let us first note that \eqref{eq: bich diagonal xi} implies that $\hat \xi_j = 0$ for all $j < l_\min$ and $j > l_\max$.
It follows that $A_{l_\max} \leq 0$ and $A_{l_\min} \geq 0$, with equalities if and only if $l_\min = l_\max$.
We may therefore assume that $l_\min < l_\max$.
By \eqref{eq: bich diagonal xi}, for any $l = 1, \hdots, m$ and $j$ such that $\q_j = \check \q_l$, $\mathsf H_\s \log \abs{\hat \xi_j} = A_l$.
Since $\abs{\hat \xi_j} \leq 1$, it follows that $\log \abs{\hat \xi_j} \leq 0$.
Since $A_{l_\min} > 0$, this equation therefore implies that $A_{l_\min}(\gamma(r)) \to 0$ as $r \to \infty$, for otherwise $\log \abs{\hat \xi_j}$ would eventually have to become positive for those $j$ such that $\q_j = \check \q_{l_\min}$. 
One further computes, using the Cauchy-Schwartz inequality, that
\begin{align*}
	\mathsf H_\s A_{l_\min}
		= & \ - 2 \sum_{j : \q_j \neq \check \q_{l_\min}} (\q_j - \check \q_{l_\min})^2 \hat \xi_j^2 + 2 A_{l_\min}^2
		\leq - 2 A_{l_\min} + 2 A_{l_\min}^2.
\end{align*}
Since $A_{l_\min}(\gamma(r)) \to 0^+$, integrating this provides the exponential bound.
The other inequality is proven similarly.
\end{proof}

In the proof of Theorem~\ref{thm: bicharacteristic flow}, we will use the following elementary remark:

\begin{remark} \label{rmk: ODE cases}
Let $f, h : \R \to \R$ be smooth functions satisfying $f'(r) + \a f(r) = h(r)$, where $\a \in \R$ is a constant.
We need the asymptotics of the solution
\[
	f(r)
		= e^{- \a r} \left( \int_0^r e^{\a s} h(s) \md s + f(0) \right),
\]
under the assumption that $\abs{h(r)} \leq \frac 1c e^{- c r}$ for some constant $c > 0$ and all $r \geq 0$.
There is the following trichotomy:
\begin{enumerate}
	\item If $\a > 0$, then $\lim_{r \to \infty}f(r) = 0$. \label{eq: alpha positive}
	\item If $\a = 0$, then $\lim_{r \to \infty}f(r) \in \R$. \label{eq: alpha zero}
	\item If $\a < 0$, then \emph{either} $\lim_{r \to \infty} f(r) = 0$ \emph{either} or $\abs{f(r)} \to \infty$ as $r \to \infty$. \label{eq: alpha negative}
\end{enumerate}
\end{remark}

\begin{proof}[Proof of Theorem~\ref{thm: bicharacteristic flow}]
Since $\mathsf H_\s \rho = - \sum_{j = 1}^n \q_j \hat \xi_j^2 \rho$, it follows that
\[
	e^{-\max_j \q_j r} \rho(\mathring \gamma)
		\leq \rho(\gamma(r))
		\leq e^{-\min_j \q_j r} \rho(\mathring \gamma).
\]
If $\rho(\mathring \gamma) = 0$, then $\rho(\gamma(r)) = 0$ for all $r \in \R$.
If $\rho(\mathring \gamma) > 0$, then $\lim_{r \to \infty}\rho(\gamma(r)) = 0$ and $\lim_{r \to - \infty}\rho(\gamma(r)) = \infty$.

Equations \eqref{eq: bich diagonal xi} and \eqref{eq: bich diagonal xi y} imply that
\begin{equation} \label{eq: integral form}
\begin{split}
	\hat \xi_k \left( \gamma(r) \right)
		= & \ e^{\left( \check \q_l - \q_k \right)r}  \int_0^r e^{-\left( \check \q_l - \q_k \right)s} \left( A_l(\gamma(s)) \hat \xi_k(\gamma(s)) \right) \md s 
		+ e^{\left( \check \q_l - \q_k \right)r}  \hat \xi_k(\mathring \gamma), \\
	\hat \xi_k(\gamma(r)) + \left( 2 \q_k - \check \q_l \right) y_k(\gamma(r))
		= & \ e^{\q_kr} \int_0^r e^{-\q_ks} \left( A_l(\gamma(s)) \hat \xi_k(\gamma(s)) \right) \md s, \\*
		& \ + e^{\q_kr} \left( \hat \xi_k(\mathring \gamma) + \left( 2 \q_k - \check \q_l \right) y_k(\mathring \gamma) \right).
\end{split}
\end{equation}
Since $\abs{\hat \xi_k} \leq 1$, Lemma~\ref{le: A min max} implies that
\[
	\left|A_{l_\max} \left(\gamma(r) \right) \hat \xi_k (\gamma(r)) \right|
		\leq C e^{ 2 r}
\]
for all $r \leq 0$.
Remark~\ref{rmk: ODE cases} applies, with the reversed conditions for $r \to -\infty$, and we conclude that $\lim_{r \to - \infty} \hat \xi_k(\gamma(r))$ exists for each $k$, and
\begin{align*}
	\lim_{r \to - \infty} \hat \xi_k(\gamma(r))
		&= 0, \quad \text{ if } \q_k \neq \check \q_{l_\max}, \\
	\lim_{r \to - \infty} \left( \hat \xi_k(\gamma(r)) + \left( 2 \q_k - \check \q_l \right) y_k(\gamma(r)) \right)
		&= 0, \quad \text{ for all } k. 
\end{align*}
We draw conclusions from this in three separate cases:

\textbf{Case 1:} If $k$ is such that $\q_k = \check \q_{l_\max}$, then $\lim_{r \to - \infty} \hat \xi_k(\gamma(r)) = - \q_k \lim_{r \to - \infty} y_k(\gamma(r))$. \\
\textbf{Case 2:}
If $k$ is such that $\q_k \neq \check \q_{l_\max}$ and $2 \q_k \neq \check \q_{l_\max}$, then we conclude that
\[
	\lim_{r \to - \infty} \hat \xi_k(\gamma(r))
		= \lim_{r \to - \infty} y_k(\gamma(r))
		= 0.
\]
\textbf{Case 3:}
If $k$ is such that $2\q_k = \check \q_{l_\max}$, then Lemma~\ref{le: degeneration linearization} and Lemma~\ref{le: A min max} imply that
\[
	\frac{\md}{\md r} \left( \hat \xi_k^2 + \q_k^2 y_k^2 \right)(\gamma(r))
		\geq \q_k \left( \hat \xi_k^2 + \q_k^2 y_k^2 \right)(\gamma(r)) - C e^{2 r},
\]
for some $C > 0$, and all $r \leq 0$.
Integrating, this implies that $\lim_{r \to - \infty} \left( \hat \xi_k^2 + \q_k^2 y_k^2 \right) (\gamma(r)) = 0$, and hence $\lim_{r \to - \infty} \hat \xi_k(\gamma(r)) = \lim_{r \to - \infty} y_k(\gamma(r)) = 0$. \\
\textbf{Case 1-3} taken together, we have thus shown that $\lim_{r \to - \infty} \gamma(r) \in \cR_{l_\max, +}$, as claimed.

Considering now $r \to \infty$, since $\abs{\hat \xi_k} \leq 1$, Lemma~\ref{le: A min max} implies that
\[
	\left|A_{l_\min} \left(\gamma(r) \right)\right|
		\leq C e^{-2 r},
\]
for all $r \geq 0$.
Remark~\ref{rmk: ODE cases}  implies that $\lim_{r \to \infty}\hat \xi_k(\gamma(r))$ exists for each $k$, and
\[
	\lim_{r \to \infty} \hat \xi_k(\gamma(r))
		= 0, \quad \text{ if } \q_k \neq \check \q_{l_\min}.
\]
The second equation in \eqref{eq: integral form} with $l = l_\min$ is of type \eqref{eq: alpha negative} in Remark~\ref{rmk: ODE cases}, and we conclude that \emph{either}
\begin{equation} \label{eq: diagonalized second limit}
	\lim_{r \to \infty} \left( \hat \xi_k + \left( 2 \q_k - \check \q_{l_\min} \right) y_k \right)(\gamma(r))
		= 0
\end{equation}
\emph{or}
\begin{equation} \label{eq: diagonalized divergent}
	\left| \left( \hat \xi_k + \left( 2 \q_k - \check \q_{l_\min} \right) y_k \right)(\gamma(r)) \right|
		\to \infty,
\end{equation}
as $r \to \infty$.
We again divide this into three separate cases.
Note that since $\q_k \geq \check \q_{l_\min}$ for all relevant $k$, it follows that $2 \q_k - \check \q_{l_\min} \neq 0$. 

\textbf{Case 1:}
If $k$ is such that $\q_k = \check \q_{l_\min}$ and \eqref{eq: diagonalized second limit} holds, then \eqref{eq: diagonalized second limit} and the fact that $\lim_{r \to \infty}\hat \xi_k(\gamma(r))$ exists imply that $\lim_{r \to \infty} \hat \xi_k(\gamma(r)) = - \q_k \lim_{r \to \infty} y_k (\gamma(r))$.

\textbf{Case 2:}
If $k$ is such that $\q_k \neq \check \q_{l_\min}$ and \eqref{eq: diagonalized second limit} holds, then $\lim_{r \to \infty} \hat \xi_k(\gamma(r)) = \lim_{r \to \infty} y_k(\gamma(r)) = 0$.

\textbf{Case 3:} 
Assume that \eqref{eq: diagonalized divergent} holds for some $k$.
Since $\hat \xi_k(\gamma(r))$ converges as $r \to \infty$, we conclude that $\abs{y(\gamma(r))} \to \infty$ as $r \to \infty$.
By the Cauchy-Schwartz inequality, we get
\[
	\mathsf H_\s \left( \sum_{k = 1}^n \q_k y_k^2 \right)
		= 2 \sum_{k = 1}^n \q_k^2 y_k^2 + 2 \sum_{k = 1}^n \q_k y_k \hat \xi_k
		\geq 2 \left( \sqrt{\sum_{j = 1}^n \q_j^2 y_j^2} - 1 \right) \sqrt{\sum_{j = 1}^n \q_j^2 y_j^2},
\]
in $\S_{+, \s}$.
Since, for $y \in \mathcal C_a$, we have
\[
	\sum_{j = 1}^n \q_j^2 y_j^2|_{\mathcal C_a}
		\geq \min_j \q_j \sum_{j = 1}^n \q_j y_j^2|_{\mathcal C_a}
		= 1 + \frac a{\min_j \q_j}
		> 1,
\]
we conclude that indeed $\mathsf H_\s \left( \sum_{k = 1}^n \q_k y_k^2 \right) |_{\mathcal C_a \cap \S_{+, \s}} > 0$.
This finishes the proof. 
\end{proof}

In order to prove the necessary saddle point estimates, we need the following extra information.

\begin{cor} \label{cor: propagation into larger source}
Let $l_\min, l_\max, \gamma$ be as in Theorem~\ref{thm: bicharacteristic flow}.
If $W_l(\mathring \gamma) > 0$, for some $l \in \{1, \hdots, m\}$, then 
\[
	\lim_{r \to - \infty} \gamma(r)
		\in \cR_{\tilde l, +}
\]
for some $\tilde l \in \{l+1, \hdots, m\}$.
\end{cor}
\begin{proof}
This follows by Theorem~\ref{thm: bicharacteristic flow} after noting that the assumption implies that $\hat \xi_j(\mathring \gamma) \neq 0$ for some $j$ such that $\q_j > \check \q_l$.
\end{proof}

\begin{cor} \label{cor: backward propagation smaller larger source}
Let $l_\min, l_\max, \gamma$ be as in Theorem~\ref{thm: bicharacteristic flow}.
There is a $\delta_0 > 0$, such that if $\de \in (0, \de_0)$ and $V_l(\mathring \gamma) \in (\de/2, \de)$ and $W_l(\mathring \gamma) \in (0, \de)$ for some $l \in \{l_\min, \hdots, l_\max\}$, then 
\begin{equation} \label{eq: r to infty l tilde}
	\lim_{r \to \infty} \gamma(r)
		\in \cR_{\tilde l, +}
\end{equation}
for some $\tilde l \in \{1, \hdots, l-1\}$, or $\abs{y(\gamma(r))} \to \infty$, as $r \to \infty$.
\end{cor}
\begin{proof}
If $\hat \xi_j(\mathring \gamma) \neq 0$ for some $j$ such that $\q_j \leq \check \q_{l-1}$, then Theorem~\ref{thm: bicharacteristic flow} implies \eqref{eq: r to infty l tilde}, where $\tilde l$ is such that $\check \q_{\tilde l} = \q_j$.
We may therefore assume that $\hat \xi_j(\mathring \gamma) = 0$, and consequently that $\hat \xi_j(\gamma(r)) = 0$ for all $r \in \R$, for all $j$ such that $\q_j \leq \check \q_{l-1}$.
If now $y_j(\mathring \gamma) \neq 0$ for some $j$ such that $\q_j \leq \check \q_{l-1}$, then \eqref{eq: Ham vf} and the fact that $\hat \xi_j(\gamma) = 0$ imply that $\abs{y_k(\gamma(r))} \to \infty$ as $r \to \infty$.
We may therefore assume that also $y_j(\gamma) = 0$ for all $j$ such that $\q_j \leq \check \q_{l-1}$.
The expression for $V_l(\mathring \gamma)$ reduces to
\begin{equation} \label{eq: Vl gamma exp}
	V_l(\mathring \gamma)
		= \sum_{j: \q_j \geq \check \q_l} \left( \hat \xi_j + \left( 2\q_j - \check \q_l \right) y_j \right)(\mathring \gamma)^2,
\end{equation}
and we may assume that $l = l_\min$.
Theorem~\ref{thm: bicharacteristic flow} implies that $\lim_{r \to \infty} \gamma(r) \in \cR_{l, +}$ or $\abs{y(\gamma(r))} \to \infty$ as $r \to \infty$.
Equation~\eqref{eq: bich diagonal xi y} implies that
\begin{align*}
	\hat \xi_k(\gamma(r)) + \left( 2 \q_k - \check \q_{l_\min} \right) y_k(\gamma(r))
		= & \ e^{\q_kr} \int_0^r e^{-\q_ks} \left( A_l(\gamma(s)) \hat \xi_k(\gamma(s)) \right) \md s, \\*
		& \ + e^{\q_kr} \left( \hat \xi_k(\mathring \gamma) + \left( 2 \q_k - \check \q_l \right) y_k(\mathring \gamma) \right),
\end{align*}
for $k = 1, \hdots, n$.
If the left-hand side is diverging for some $k$, then the proof is complete.
Assume therefore that the left-hand side converges for every $k$, which means that
\begin{equation} \label{eq: integral equality}
	- \int_0^\infty e^{-\q_ks} A_{l_\min}(\gamma(s)) \hat \xi_k(\gamma(s)) \md s
		= \hat \xi_k(\mathring \gamma) + \left( 2 \q_k - \check \q_{l_\min} \right) y_k(\mathring \gamma).
\end{equation}
Since $\abs{A_{l_\min}(\mathring \gamma)} < \left( \check \q_n - \check \q_1 \right) W_l(\mathring \gamma) \leq \left( \check \q_n - \check \q_1 \right) \de < 1$ if $\de > 0$ is small enough, Lemma~\ref{le: A min max} implies that $\abs{A_{l_\min}(\gamma(r))}$ is strictly decreasing in $r$.
We can therefore estimate the left-hand side of \eqref{eq: integral equality} by
\[
	\int_0^\infty e^{-\q_ks} \abs{A_{l_\min}(\gamma(s)) \hat \xi_k(\gamma(s))} \md s
		\leq \abs{A_{l_\min}(\mathring \gamma)} \int_0^\infty e^{-\q_ks} \md s
		\leq \frac{\check \q_n - \check \q_1}{\check \q_1} \de.
\]
On the other hand, since $\sum_{j = 1}^n \left( \hat \xi_k(\mathring \gamma) + \left( 2 \q_k - \check \q_{l_\min} \right) y_k(\mathring \gamma) \right)^2 = V_l(\mathring \gamma) \geq \frac \de 2$, there is at least one $k$ with $\abs{\hat \xi_k(\mathring \gamma) + \left( 2 \q_k - \check \q_{l_\min} \right) y_k(\mathring \gamma)} \geq \sqrt{\frac{\de}{2n}}$.
This is a contradiction to \eqref{eq: integral equality} for at least one $k$, if $\de$ is small enough. 
\end{proof}

\subsection{The spectral family of Fredholm operators}

The purpose of this subsection is to prove that the \emph{mode operator}, or \emph{Fourier-transformed in time} wave wave operator is an analytic family of Fredholm operator between appropriate Hilbert spaces. 
The main input is the saddle point structure for the lightlike geodesics (lifted to phase space) in the previous subsection, and the microlocal saddle point estimates proven in Theorem~\ref{thm: saddle points} and Theorem~\ref{thm: saddle points semiclassical}.
Let us first fix the setting for the rest of the section:
\begin{assumption} \label{ass: setting Fredholm}
Let $(M, g)$ be a Kasner-type spacetime and let $\P$ be as in Theorem~\ref{thm: main QNM} and let $\mP := t^2 \P$.
For a fixed $a > 0$, define
\[
	\mX_a
		:= \left\{ y \in \R^n \mid \sum_{j = 1}^n \q_j y_j^2 \leq \frac1{\min_j \q_j } + a \right\}.
\]
\end{assumption}
We note that $[\T, \mP] = [\T, t^2\P] = 0$.
Following the philosophy of \cite{V2013}, we will study the wave equation in the open subset 
\begin{equation} \label{eq: Y a}
	\mY_a 
		:= \R_\tau \times \mX_a \subset M,
\end{equation}
which contains $J^+(\gamma)$, by studying a spectral problem on $\mX_a$.
We define the spectral family of mode operators
\begin{equation}\label{eq: hat P sigma}
	\hat \mP(\s): \ C^\infty(\mX_a) \to C^\infty(\mX_a),
	\quad
	\hat \mP(\s) v(y)
		:= e^{i\s \tau} \mP \left( e^{- i\s \tau} v(y) \right), \ y \in \mX_a.
\end{equation}
\begin{remark} \label{rmk: QNMs and mode operator}
Recall that $\d_\tau = \T$.
The point is that $u \in \mY_a$ satisfies 
\[
	\mP u 
		= 0,
	\qquad 
	\left(\T + i \s \right) u
		= 0,
\]
i.e.~$u$ is a QNM for $\mP$, if and only if $u(\tau, x) = e^{-i\s\tau}v(y)$ with $\hat \mP (\s) v = 0$.
\end{remark}
In this section, we find appropriate Hilbert spaces on which $\hat \mP(\s)$ is an analytic Family of Fredholm operators, with a meromorphic continuation of the resolvent $\hat \mP(\s)^{-1}$.

\subsubsection{The Fredholm operators}

Note that by \eqref{eq: dual metric normalized},
\[
	\mP
		= \left( \d_\tau + \sum_{j = 1}^n \q_j y_j \d_{y_j} \right)^2 - \sum_{j = 1}^n \d_{y_j}^2 + \A(y) \d_\tau + \sum_{j = 1}^n \B_j(y) \d_{y_j} + \Cm(y),
\]
where $\A, \B_1, \hdots, \B_n, \Cm$ are square matrices independent of $\tau$.
It follows that
\begin{equation}\label{eq: hat P sigma explicit}
	\hat \mP(\s)
		= - \left(\s - \sum_{j = 1}^n \q_j y_j \frac1i \d_{y_j} \right) - \sum_{j = 1}^n \d_{y_j}^2 - i \s \A(y) + \sum_{j = 1}^n \B_j(y) \d_{y_j} + \Cm(y),
\end{equation}
The principal symbol of $\hat \mP(\s)$ is given by 
\begin{equation} \label{eq: principal symbol mode operator}
	p 
		= \sum_{j = 1}^n \xi_j^2 - \left( \sum_{j = 1}^n \q_j y_j  \xi_j\right)^2.
\end{equation}
In order to verify the hypotheses of Section~\ref{sec: saddle point estimates}, we need to rescale the principal symbol to $\tilde p := \frac1{\abs \xi^2}p = 1 - \left( \sum_{j = 1}^n \q_j y_j  \hat \xi_j\right)^2$, which now extends smoothly up to fiber infinity $S^*\R^n$.
Note that the characteristic set of $\hat \mP(\s)$ is therefore given by the two components $\d \S_{\s, \pm} = \d \S_{0, \pm} \subset S^*\R^n$.
The Hamiltonian vector field, rescaled by the factor $\frac1{2\abs \xi}$ in order for it to smoothly extend (and be tangent to) fiber infinity $S^*M$, is given by
\[
	\tilde H_p|_{\d \S_{0,\pm}}
		= \mathsf H_\s|_{\d \S_{0,\pm}}
		= \mathsf H_0|_{\d \S_{0,\pm}},
\]
where $\mathsf H_0$ is given by \eqref{eq: Ham vf}.

\begin{lemma} \label{le: saddle points in our setting}
Let $a > 0$ and let $\mP$ and $\mX_a$ be as in Assumption~\ref{ass: setting Fredholm}.
Let $l \in \{1, \hdots, m\}$ and $\s \in \C$.
Then $\cR_{+, l} \subset \oT \R^n$ is a saddle manifold for $\hat \mP(\s)$, defined in \eqref{eq: hat P sigma}, in the sense of Definition~\ref{def: saddle manifold}.
\end{lemma}
\begin{proof}
Note first that $\cR_{+,l} \subset \oT \R^n$ is a smooth closed submanifold. 
Proposition~\ref{prop: Stationary points} implies that condition (A.1) in Definition~\ref{def: saddle manifold} is satisfied.
We compute
\[
	\md \tilde p
		= 2 \sum_{j = 1}^n \q_j y_j  \hat \xi_j \sum_{k = 1}^n \q_k \left(\hat \xi_k \md y_k + y_k \md \hat \xi_k \right),
\]
which is non-zero at $\cR_{+,l}$. 
Hence condition (A.2) in Definition~\ref{def: saddle manifold} is also satisfied.
Condition (A.3) in Definition~\ref{def: saddle manifold} is a consequence of Theorem~\ref{thm: saddle point flow}.
Lemma~\ref{le: diagonal linearization} implies that $\tilde H_p \rho = \left( \check \q_l + A_l \right) \rho$.
The condition (A.4) in Definition~\ref{def: saddle manifold} is therefore satisfied with
\begin{equation} \label{eq: beta 0}
	\b_0 
		:= \check \q_l + A_l,
\end{equation}
since $\b_0|_{\cR_{+, l}} = \check \q_l > 0$.
\end{proof}

It remains to compute the threshold quantity for $\hat \mP(\s)$ as in Definition~\ref{def: saddle manifold threshold}.
Fixing the volume form $\mathrm{Vol} := \md y_1 \wedge \hdots \wedge \md y_n$, \eqref{eq: hat P sigma explicit} implies that
\[
	\frac{\hat \mP(\s) - \hat \mP(\s)^*}{2i}
		= - 2 \Im(\s) \sum_{j = 1}^n \q_j y_j \frac1i \d_{y_j} + \sum_{j = 1}^n \Im(\B_j(y)) \frac1i\d_{y_j} + \text{ 0th order terms},
\]
which has principal symbol $p_1 = - 2 \Im(\s) \sum_{j = 1}^n \q_j y_j \xi_j + \sum_{j = 1}^n \B_j(y) \xi_j$.
Rescaling this appropriately to extend smoothly to fiber infinity $S^*\R^n$ (see Section~\ref{sec: saddle point estimates}), we get 
\[
	\tilde p_1 = \frac1{\abs \xi} p_1 = - 2 \Im(\s) \sum_{j = 1}^n \q_j y_j \hat \xi_j + \sum_{j = 1}^n \hat \xi_j \Im(\B_j(y)).
\]
Restricting this to the saddle manifold $\cR_{+, l}$ gives $\tilde p_1|_{\cR_{+, l}} = \sum_{\q_j = \check \q_l} \hat \xi_j \Im(\B_j(y))|_{\cR_{+, l}} - 2 \Im(\s)$.
We have therefore computed the threshold quantity $\tilde \b$ on $\cR_{+, l}$ in Definition~\ref{def: saddle manifold threshold} to be
\[
	\tilde \b_\s
		= \frac1{\b_0} \tilde p_1|_{\cR_{+, l}}
		= \frac1{\check \q_l} \sum_{\q_j = \check \q_l} \hat \xi_j \Im(\B_j)|_{\cR_{+, l}} - \frac2{\check \q_l} \Im(\s).
\]
Let us define $\tilde \b_{\max}(\s) := \max_{l = 1, \hdots, m}\max_{x \in \cR_{+, l}} \norm{\tilde \b_\s(x)}_{\mathrm{op}}$, and note that 
\[
	\tilde \b_{\max}(\s) 
		\leq \tilde \b_{\max}(0) - \frac2{\check \q_m} \Im(\s),
\]
since $\check \q_m \geq \check \q_1, \hdots, \check \q_{m-1}$.

We are now ready to state the main theorem. 

\begin{thm}[Fredholm property of $\hat \mP(\s)$] \label{thm: Fredholm property}
Let $a > 0$ and let $\mP$ and $\mX_a$ be as in Assumption~\ref{ass: setting Fredholm}.
Let $\a \in \R$, and let $s > \frac12 + \frac2{\check \q_m} \a + \tilde \b_\max(0)$.
Then 
\begin{equation} \label{eq: Fredholm map}
	\hat \mP(\s) : \{u \in \bar H^s(\mX_a) \mid \hat \mP(0) u \in \bar H^{s-1}(\mX_a) \} \to \bar H^{s-1}(\mX_a)
\end{equation}
is an analytic family of Fredholm operators for $\s \in \C$ with $\Im(\s) > \a$.
Moreover, 
\begin{enumerate}
	\item $\ker_{\bar H^s(\mX_a)} \hat \mP(\s) \subset \bar C^\infty(\mX_a)$, 
	\item $\ker_{\dot H^{-s +1}(\mX_a)} \hat \mP(\s)^* \subset \cap_{\epsilon > 0} \dot H^{\frac12 - \tilde \b_\max(\s) - \epsilon}(\mX_a)$, 
	\item $f \in \bar H^{s-1}(\mX_a)$ lies in the range of \eqref{eq: Fredholm map} if and only if $\ldr{f, u^*}_{L^2(\mX_a)} = 0$ for all $u^* \in \dot H^{-s + 1}(\mX_a)$ with $\hat \mP(\s)^* u^* = 0$.
\end{enumerate}
\end{thm}

\begin{remark}
In the isotropic case, when $\q_1 = \hdots = \q_n$, there are no saddle point estimates needed, as the bicharacteristic flow has only a normal source/sink structure (just set $\q_1 = \hdots = \q_n$ in Theorem~\ref{thm: bicharacteristic flow} and Theorem~\ref{thm: saddle point flow}).
In this case, Theorem~\ref{thm: Fredholm property} follows immediately by the analysis of Vasy in \cite{V2013} (though Big Bang spacetimes were not considered in \cite{V2013}), see also \cite{HV2015}.
The technical novelty of Theorem~\ref{thm: Fredholm property} is that we allow for \emph{anisotropy}, where we do not assume that all $\q_j$ coincide.
\end{remark}

Lemma~\ref{le: saddle points in our setting} in combination with Theorem~\ref{thm: saddle points}, together with the global description of the bicharacteristic flow in Theorem~\ref{thm: bicharacteristic flow}, are the essential ingredients in the proof.

\begin{proof}[Proof of Theorem~\ref{thm: Fredholm property}]
Let $a > 0$ and define $b := a/6$ in order to simplify the notation.
Fix $\psi, \chi \in C_c^\infty(\mX_a)$, such that $\psi(y) = 1$ for $y \in \mX_{2b}$ and $\supp(\psi) \subset \mX_{3b}$, and such that $\chi(y) = 1$ for $y \in \mX_{4b}$ and $\supp(\chi) \subset \mX_{5b}$.
We set $s_0 := \frac12 + \frac2{\check \q_m} \a + \tilde \b_\max(0) > \frac12 + \tilde \b_\max(\sigma)$.

The idea is to extend the source/sink arguments of Vasy in \cite{V2013} to this saddle point flow by successively establish the desired estimate on larger and larger subsets of $\supp(\psi)$, by iteratively applying the saddle point estimate, Theorem~\ref{thm: saddle points}.
Lemma~\ref{le: saddle points in our setting} implies that all assumptions of Theorem~\ref{thm: saddle points} holds for the saddle manifold $\cR_{+, l}$, for $l = 1, \hdots, m$, with the threshold regularity given by $\tilde \b_\max(\sigma)$.
By \eqref{eq: negative of Hamiltonian VF}, the corresponding properties holds at $\cR_{-, l}$, for $l = 1, \hdots, m$, where we instead propagate in the direction of $- \mathsf H_\s$.
We want to use this to establish microlocal control of the regularity at every saddle manifold $\cR_{\pm, l}$, for $l = 1, \hdots, m$.
Once we have done that, then Theorem~\ref{thm: bicharacteristic flow}, together with standard propagation of singularities and elliptic estimates away from the characteristic set, will prove our conclusion.

Recall the definition of $V_l$ and $W_l$ in \eqref{eq: Vl Wl}.
Since $W_m = 0$, the saddle manifold $\cR_{+, m}$ is actually a normal source.
We can therefore apply Theorem~\ref{thm: saddle points} with $E = 0$ and conclude that there is a pseudodifferential operator $B_m$, which is elliptic in $\cR_{+, m}$, such that for all $s, N \in \R$, with $s > s_0$, there exists a $C > 0$, such that
\begin{equation} \label{eq: elliptic max}
	\norm{B_{+,m} u}_{H^s}
		\leq C \left( \norm{\chi \hat \mP(\s) u}_{H^{s-1}} + \norm{\chi u}_{H^{s_0}} + \norm{\chi u}_{H^{-N}} \right),
\end{equation}
for any $u \in \D'(\R^n)$ in the strong sense that if the right-hand side is finite, then the left-hand side is finite as well.
Note here that we have replaced the $G$ in the estimate of Theorem~\ref{thm: saddle points} by $\chi$, which is possible by elliptic regularity, since we can choose $G$ such that $\WF'(G) \subset \supp(\chi)^\circ$.

We have thus established microlocal control of the regularity in a neighbourhood of $\cR_{+, m}$ and would like to propagate this to the other saddle manifolds. 
For this, we note that Theorem~\ref{thm: saddle points} implies that there is a pseudodifferential operator $B_l$, which is elliptic in $\cR_{+, l}$, such that for all $s, N \in \R$, with $s > s_0$, there exists a $C > 0$, such that
\[
	\norm{B_{+,l} u}_{H^s}
		\leq C \left( \norm{\chi \hat \mP(\s) u}_{H^{s-1}} + \norm{\chi u}_{H^{s_0}} + \norm{E u}_{H^s} + \norm{\chi u}_{H^{-N}} \right),
\]
where $E$ is any pseudodifferential operator where
\[
	\{\tilde p \leq \de^2 \} \cap \{ \rho \leq \de^2 \} \cap \{ V_l \leq \de \} \cap \{ \de/2 \leq W_l \leq \de \} 
		\subset \Ell(E),
\]
for a suitably small $\de > 0$ and for any $u \in \D'(\R^n)$ in the strong sense that if the right hand side is finite, then the left-hand side is finite as well.
In order to conclude regularity at $\cR_{+, l}$, we therefore need to have an estimate in $\Ell(E)$.
We therefore choose $E$ such that 
\[
	\Ell(E)
		\subset \{\tilde p \leq 2 \de^2 \} \cap \{ \rho \leq 2 \de^2 \} \cap \{ V_l \leq 2 \de \} \cap \{ 2 \de/2 \leq W_l \leq \de \}.
\]
By elliptic regularity theory, it suffices to prove that we have regularity in
\[
	\Char(P) \cap \Ell(E)
		\subset K_l := \{\tilde p = 0 \} \cap \{ \rho = 0\} \cap \{ V_l \leq 2 \de \} \cap \{ 2 \de/2 \leq W_l \leq \de \}
		\subset S^*\R^n.
\]
The point is that, by Corollary~\ref{cor: propagation into larger source}, any bicharacteristic in $\d \S_{0, +} \subset S^*\R^n$ through this set will approach to a saddle manifold $\cR_{+, \tilde l}$ to the past, for some $\tilde l > l$.

In order to perform an induction argument over $l$, let us therefore assume that we have an estimate of the form \eqref{eq: elliptic max} in $\cup_{\tilde l > l} \cR_{+, \tilde l}$.
Since $\mathrm K_l$ is compact, standard propagation of singularities, see e.g.~\cite{H2025}*{Thm.~8.10}, implies that
\[
	\norm{Eu}_{H^s} 
		\leq C \left( \norm{\chi \hat \mP(\s) u}_{H^{s-1}} + \sum_{\tilde l > l} \norm{B_{+,\tilde l} u}_{H^s} + \norm{\chi u}_{H^{-N}} \right)
\]
for some $C > 0$, and therefore, for all $s, N \in \R$, with $s > s_0$, there exists a $C > 0$, such that
\begin{equation} \label{eq: saddle point + at l}
	\norm{B_{+,l} u}_{H^s}
		\leq C \left( \norm{\chi \hat \mP(\s) u}_{H^{s-1}} + \norm{\chi u}_{H^{s_0}} + \norm{\chi u}_{H^{-N}} \right),
\end{equation}
in the usual strong sense, in case the corresponding estimate holds for all $\tilde l > l$.
Since we have shown the base case for the induction over $l = m, m-1, \hdots, 1$ in \eqref{eq: elliptic max}, we conclude by induction the estimate \eqref{eq: saddle point + at l} in the strong sense, for each $l = 1, \hdots, m$.

By propagating backwards along the Hamiltonian vector field, we similarly obtain an analogous estimate for all $s, N \in \R$, with $s > s_0$, there exists a $C > 0$, such that
\begin{equation} \label{eq: saddle point - at l}
	\norm{B_{-,l} u}_{H^s}
		\leq C \left( \norm{\chi \hat \mP(\s) u}_{H^{s-1}} + \norm{\chi u}_{H^{s_0}} + \norm{\chi u}_{H^{-N}} \right),
\end{equation}
where $B_{-,l}$ is elliptic at $\cR_{-, l}$, for $l = 1, \hdots, m$.

We have thus established control of the desired estimate at every saddle point. 
Fix operators $B_\pm \in \Psi^0(\mX_a^\circ)$ with Schwartz kernels supported in $\mX_{3b} \times \mX_{3b}$, such that $B_\pm$ is elliptic on $\d \S_{\s, \pm} \cap S^*_{\supp(\psi)} \mX_a$ and $\WF'(B_\pm) \cap \d \S_{\s, \mp} = \emptyset$.
Standard propagation of singularities, e.g.~\cite{H2025}*{Prop.~8.10}, together with our obtained estimates in \eqref{eq: saddle point + at l} and \eqref{eq: saddle point - at l}, and the structure of the bicharacteristic flow in Theorem~\ref{thm: bicharacteristic flow} implies that for all $s, N \in \R$, with $s > s_0$, there exists a $C > 0$, such that
\begin{equation} \label{eq: global char set estimate}
	\norm{B_\pm u}_{H^s}
		\leq C \left( \norm{\chi \hat \mP(\s) u}_{H^{s-1}} + \norm{\chi u}_{H^{s_0}} + \norm{\chi u}_{H^{-N}} \right),
\end{equation}
in the usual strong sense.
Moreover, fix an operator $B \in \Psi^0(\mX_a^\circ)$ with Schwartz kernels supported in $\mX_{3b} \times \mX_{3b}$, such that $\WF'(B) \cap \d \S_{\s, \pm} = \emptyset$.
Standard elliptic regularity, e.g.~\cite{H2025}*{Exercise~6.12}, implies the estimate for all $s, N \in \R$, there exists a $C > 0$, such that
\begin{equation} \label{eq: first elliptic estimate}
	\norm{B u}_{H^s}
		\leq C \left( \norm{\chi \hat \mP(\s) u}_{H^{s-2}} + \norm{\chi u}_{H^{-N}} \right),
\end{equation}
in the usual strong sense. 
Finally note that we can choose this $B$ such that $\Ell(B_+) \cup \Ell(B_-) \cup \Ell(B) \supset \supp(\psi)$.
Another application of standard elliptic regularity, \cite{H2025}*{Exercise~6.12}, implies that the esimates \eqref{eq: global char set estimate} and \eqref{eq: first elliptic estimate} combine to imply that for all $s, s_0, N \in \R$, with $s > s_0$ on $\cR$, there exists a $C > 0$, such that
\begin{equation} \label{eq: X 3b estimate}
	\norm{u}_{\bar H^s(\mX_{2b})}
		\leq \norm{\psi u}_{H^s}
		\leq C \left( \norm{\chi \hat \mP(\s) u}_{H^{s-1}} + \norm{\chi u}_{H^{s_0}} + \norm{\chi u}_{H^{-N}} \right),
\end{equation}
in the usual strong sense, that if the right-hand side is finite, so is the left-hand side.

We finally want to estimate $u$ in the region $\mX_a \backslash \mX_b$, which together with \eqref{eq: X 3b estimate} would give the desired global semi-Fredholm estimate. 
For this, we simply generalize the corresponding step in the proof of \cite{H2025}*{Thm.~11.16} using a simple application of a standard energy estimate for wave equations, see \cite{H2025}*{Cor.~9.38}.
We claim that with $f := \sum_{j = 1}^n \q_j y_j^2$, then $\md f$ is \emph{timelike} in $\mX_a \backslash \mX_b$.
Indeed, with $\g^{-1}$ given in \eqref{eq: dual metric normalized}, we compute
\[
	\g^{-1}(\md f, \md f)
		= 4 \g^{-1} \left( \sum_{j = 1}^n \q_j y_j \md y_j, \sum_{j = 1}^n \q_j y_j \md y_j \right)
		= 4 \left( 1 - \sum_{j = 1}^n \q_j^2 y_j^2 \right) \sum_{k = 1}^n \q_k^2 y_k^2
		< 0,
\]
since $\sum_{j = 1}^n \q_j^2 y_j^2 > 1$ in $\mX_a \backslash \mX_b$.
It follows that $\mX_a \backslash \mX_b^\circ$ is an admissible domain in the sense of \cite{H2025}*{Def.~9.14} with $f_1 = f - b$ and $f_2 = a - f$.
Therefore, \cite{H2025}*{Cor.~9.38} applies with $\Omega_{- \de, 0} = \mX_a \backslash \mX_{2b}^\circ$ and $\Omega = \mX_a \backslash \mX_b^\circ$, and implies that the energy estimate
\begin{align*}
	\norm{u}_{\bar H^s(\mX_a \backslash \mX_b)}
		&\leq C \left( \norm{\hat \mP(\s) u}_{\bar H^{s-1}(\mX_a \backslash \mX_b)} + \norm{u}_{\bar H^s(\mX_{2b} \backslash \mX_b)} \right) \\
		&\leq C \left( \norm{\hat \mP(\s) u}_{\bar H^{s-1}(\mX_a)} + \norm{u}_{\bar H^s(\mX_{2b})} \right),
\end{align*}
in the usual strong sense.
This combines with \eqref{eq: X 3b estimate} to conclude that
\[
	\norm{u}_{\bar H^s(\mX_a)}
		\leq C \left( \norm{\hat \mP(\s) u}_{\bar H^{s-1}(\mX_a)} + \norm{u}_{\bar H^{s_0}(\mX_a)} \right),
\]
in the usual strong sense.
This is the desired semi-Fredholm estimate for $\hat \mP(\s)$.
\end{proof}

\subsubsection{The meromorphic continuation of the resolvent}
The spectral family $\hat \mP(\s)$ is invertible if $\Im(\s)$ is large enough, which relies on the following rough energy estimate:
\begin{prop}[A rough energy estimate] \label{prop: rough energy estimate}
Let $a > 0$ and let $\mP$ and $\mX_a$ be as in Assumption~\ref{ass: setting Fredholm}.
Let $s \in \N_0$. 
There is a constant $C_0 > 0$, such that if $w \in C^\infty\left(\mY_a\right)$ with $\supp(w) \subset \{\tau > 0\}$, the unique $v \in C^\infty\left(\mY_a\right)$ to $\mP v = w$ with $\supp(v) \subset \{\tau > 0\}$, satisfies
\[
	\norm{e^{- C_0 \tau}v}_{H^s\left(\mY_a \right) }
		 \leq C_0 \norm{e^{- C_0 \tau}w}_{H^{s-1}\left(\mY_a \right) }.
\]
\end{prop}
\begin{proof}
The first step is to get a local in time energy estimate.
We would like to apply \cite{H2025}*{Cor.~9.38} in the spacetime domains $\Omega = [\tau_0 - 1, \tau_0 + 1] \times \mX_a$ and $\Omega_{0, -1} = [\tau_0 - 1, \tau_0] \times \mX_a$.
For this, we need to first check that $\Omega$ is admissible in the sense of \cite{H2025}*{Def.~10.14}.
Firstly, we choose
\[
	f_1 
		= \tau - (\tau_0-1),
	\quad
	f_2
		= \tau_0 + 1 - \tau, 
	\quad
	f_3
		=  \frac1{\min_j \q_j} + a - \sum_{j = 1}^n \q_j y_j^2. 
\]
Then $\Omega = \cap_{j = 1}^3 \{f_j \geq 0\}$, the boundary components are connected and the differentials of the $f_1, f_2, f_3$ are linearly independent at the corners of $\Omega$.
Moreover, note that
\begin{align*}
	\g^{-1}(\md f_1, \md f_1)|_{f_1 = 0}
		= & \ \g^{-1}(\md f_1, \md \tau)|_{f_1 = 0}
		= - 1, \\
	\g^{-1}(\md f_2, \md f_2)|_{f_2 = 0}
		= & \ - \g^{-1}(\md f_2, \md \tau)|_{f_2 = 0}
		= - 1, \\
	\g^{-1}(\md f_3, \md f_3)|_{f_3 = 0}
		= & \sum_{j = 1}^n \q_j^2 y_j^2 - \left( \sum_{j = 1}^n \q_j^2 y_j^2 \right)^2|_{f_3 = 0} 
		< 0, \\
	\g^{-1}(\md f_3, \md \tau)|_{f_3 = 0}
		= & \ 2 \sum_{j = 1}^n \q_j^2 y_j^2|_{f_3 = 0}
		> 0.
\end{align*}
Therefore $\Omega \cap f_1^{-1}(0)$ is a spacelike \emph{initial hypersurface}, and the hypersurfaces $\Omega \cap f_2^{-1}(0)$ and $\Omega \cap f_3^{-1}(0)$ are spacelike \emph{final} hypersurfaces to the \emph{admissible region} $\Omega$.
Following the notation on \cite{H2025}*{Sec.~9.4}, we note that $\Omega_{-1, 0} = \{f_1 \geq 1\} \cap_{j = 2}^3 \{ f_j \geq 0 \}$.
By \cite{H2025}*{Cor.~9.38} and since $\mP$ is invariant under translations in $\tau$, we thus conclude the energy estimate
\[
	\norm{u}_{H^s([\tau_0, \tau_0 +1] \times \mX_a)}
		\leq C \left( \norm{\mP u}_{H^{s-1}([\tau_0-1, \tau_0 +1] \times \mX_a)} + \norm{u}_{H^s([\tau_0-1, \tau_0] \times \mX_a)} \right),
\]
for any $\tau_0 \in \R$.
Note that for $\tau_0 = 0$, the support conditions imply that 
\[
	\norm{u}_{H^s([0, 1] \times \mX_a)}
		\leq C \norm{\mP u}_{H^{s-1}([0, 1] \times \mX_a)}.
\]
The rest of the proof is the same as the proof of \cite{H2025}*{Lem.~11.14}.
\end{proof}

Theorem~\ref{thm: Fredholm property} combines with Proposition~\ref{prop: rough energy estimate} to give the meromorphic extension of the resolvent.
We present this through several corollaries.

\begin{cor}[Invertibility for large $\Im(\s)$] \label{cor: resolvent large im sigma}
There exists $C \in \R$ such that for all $\s \in \C$ with $\Im(\s) > C$, the operator $\hat \mP(\s)$ in \eqref{eq: Fredholm map} is invertible.
\end{cor}
\begin{proof}
The proof is exactly the same as the proof of \cite{H2025}*{Prop.~11.18} with Theorem~\ref{thm: Fredholm property} replacing \cite{H2025}*{Thm.~11.16}, and with Proposition~\ref{prop: rough energy estimate} replacing \cite{H2025}*{Lem.~11.14}.
\end{proof}

\begin{cor}[Meromorphic continuation of the resolvent] \label{cor: meromorphic extension}
Fix $a > 0$ and $\a \in \R$, and let $s > \frac12 + \frac2{\check \q_m} \a + \tilde \b_\max(0)$.
Then $\hat \mP(\s)$ in \eqref{eq: Fredholm map} has index $0$ and the resolvent 
\begin{equation} \label{eq: resolvent Sobolev}
	\hat \mP(\s)^{-1}: \bar H^{s-1}(\mX_a) \to \bar H^s(\mX_a)
\end{equation}
extends from $\Im(\s) \gg 1$ to a finite-meromorphic family of operators on $\left\{\s \in \C \mid \Im(\s) > - \a \right\}$.
The set of poles with $\Im(\s) > - \a$ is independent of the value of $s > \frac12 + \frac2{\check \q_m} \a + \tilde \b_\max(0)$.
\end{cor}
\begin{proof}
The first statement follows by Theorem~\ref{thm: Fredholm property} and Corollary~\ref{cor: resolvent large im sigma} and the analytic Fredholm theorem, e.g.~\cite{H2025}*{Thm.~5.77}.
The argument for the second statement is the same as for the proof of \cite{H2025}*{Cor.~11.19}:
Since $\hat \mP(\s)$ has index $0$, it is invertible if and only if it has a trivial nullspace. 
But, by Theorem~\ref{thm: Fredholm property}, the nullspace is a subspace of $\bar C^\infty(\mX_a)$ and is thus independent of $s$.
\end{proof}

\begin{cor}[QNM frequencies as poles; generalized QNMs residues] \label{cor: discrete QNM} \
Let $a > 0$.
\begin{enumerate}
\item The set of all $\s \in \C$, such that there is a $u \in C^\infty(\mY_a)$ and a $k \in \N$, such that
\begin{equation} \label{eq: QNM t squared modified}
	\mP u = 0, \qquad \left( \T + i \s \right)^k u,
\end{equation}
is \emph{discrete}.
A number $\s_0 \in \C$ is a resonance if and only if $\hat \mP(\s)^{-1}$ as in \eqref{eq: resolvent Sobolev} has a pole at $\s = \s_0$ (for suitably large regularity $s$ in \eqref{eq: resolvent Sobolev}). 
\item For all $\s \in \C$, the space of all $u \in C^\infty(\mY_a)$, satisfying \eqref{eq: QNM t squared modified} for some $k \in \N$, is finite-dimensional and equal to
\begin{align*}
	& \Big \{\res_{\s = \s_0} \left( e^{-i \s \tau} \hat \mP(\s)^{-1} f(\s) \right) \mid f \text{ is a polynomial in } \s \text{ with values in } \bar C^\infty(\mX_a) \Big \}.
\end{align*}
\end{enumerate}
Moreover, any smooth generalized QNM in $\mY_a$ extends uniquely to a smooth generalized QNM in $M$.
\end{cor}
\begin{proof}
The proofs of (1) and (2) are exactly the same as the proof of \cite{H2025}*{Prop.~11.20} with Corollary~\ref{cor: meromorphic extension} replacing \cite{H2025}*{Cor.~11.19}.
The last assertion follows by solving a Cauchy problem for a linear wave equation, since $\d_\tau$ is timelike in $M \backslash \mY_a$.
\end{proof}

\begin{remark} \label{rmk: relating QNMs}
Note that if $u \in C^\infty(M)$ is a generalized quasinormal mode for $\P$, with $\P$ as in Theorem~\ref{thm: main QNM}, then also $\mP u = t^2 \P u = 0$.
The assertion in Theorem~\ref{thm: main QNM} that the QNM frequencies form a discrete set therefore follows immediately by Corollary~\ref{cor: discrete QNM}.
However, we still need the high-energy estimates from Section~\ref{thm: high-energy estimate} to prove that there are only finitely many quasinormal mode frequencies with imaginary part over a fixed threshold.
\end{remark}

\subsubsection{Solving wave equations for quasinormal modes}

In Section~\ref{sec: QNM geometric waves} we will show how to compute the QNMs and the QNM frequencies, by introducing what we call \emph{annihilation} and \emph{creation} operators (see Theorem~\ref{thm: creation and annihilation}). 
The construction of the creation operator will rely on the following:

\begin{prop} \label{prop: QNM wave equations}
Let $\mP$ be as in Assumption~\ref{ass: setting Fredholm} and let $a > 0$.
Let $f \in C^\infty(\mY_a)$ be a generalized quasinormal mode with frequency $\s \in \C$.
Then there is a $u \in C^\infty(\mY_a)$ such that $\mP(u) = f$ and 
\[
	\left( \d_\tau + i \s \right)^{k'} u 
		= 0
\]
for some $k' \in \N$.
Moreover, $u$ is unique up to adding QNMs with frequency $\s$.
\end{prop}
\begin{proof}
Let $f \in C^\infty(M)$ be a generalized quasinormal mode, i.e.~we assume that $\left( \d_\tau + i \s \right)^k f = 0$ for some $\s \in \C$ and $k \in \N$.
This is equivalent to 
\[
	f(\tau, x)
		= e^{-i\s \tau} \sum_{j = 0}^{k-1} \frac{f_j(x) \tau^j}{j!}.
\]
We define
\[
	u(\tau, x)
		= \sum_{j = 0}^{k-1} \frac{i^j}{2\pi i} \oint_\s \frac{e^{- i z \tau}}{(z - \s)^{j+1}} \hat \mP(z)^{-1} f_j(x) \md z.
\]
This is well-defined since $\hat \mP(z)^{-1}$ is meromorphic on suitably regular function spaces (see Theorem~\ref{thm: Fredholm property}), implying that there is a number $l \in \N_0$, such that $(z-\s)^l \hat \mP(z)^{-1} f_j$ is holomorphic near $\s$ for all $j$.
We now check that
\begin{align*}
	\mP u (\tau, x)
		= & \ \sum_{j = 0}^{k-1} \frac{i^j}{2\pi i} \oint_\s \frac1{(z - \s)^{j+1}} \mP \left( e^{- i z \tau} \hat \mP(z)^{-1} f_j(x) \right) \md z \\
		= & \ \sum_{j = 0}^{k-1} \frac{i^j}{2\pi i} \oint_\s \frac{e^{- i z \tau}}{(z - \s)^{j+1}} \hat \mP(z) \hat \mP(z)^{-1} f_j(x) \md z
		= \ \sum_{j = 0}^{k-1} \frac{i^j}{2\pi i} \oint_\s \frac{e^{- i z \tau}}{(z - \s)^{j+1}} \md z f_j(x) \\
		= & \ e^{-i \s\tau} \sum_{j = 0}^{k-1} \frac{\tau^j}{j!} f_j(x)
		= \ f(\tau, x),
\end{align*}
as claimed, and
\begin{align*}
	\left( \d_\tau + i \s \right)^{j + l + 1} u(\tau, x)
		= \sum_{j = 0}^{k-1} \frac{i^j}{2\pi i} \oint_\s e^{- i z \tau} (z - \s)^l \hat \mP(z)^{-1} f_j(x) \md z 
		= 0,
\end{align*}
since $(z-\s)^l \hat \mP(z)^{-1} f_j$ is holomorphic near $\s$ for all $j$.
This proves the first assertion.
If we now both $u_1$ and $u_2$ are generalized quasinormal mode solutions satisfying $\mP(u_1) = f = \mP(u_2)$, then $u_1 - u_2$ is a generalized quasinormal mode, proving the second assertion.
\end{proof}

\subsection{The high-energy estimates}

Following Vasy's scheme for wave equations on black hole spacetimes in \cite{V2013}, we need to prove \emph{high-energy} estimates for the spectral family $\hat \mP(\s)$ in order to prove Theorem~\ref{thm: main QNM} and Theorem~\ref{thm: main as exp}.
That is, by means of semiclassical analysis, we will prove uniform estimates when $\Im(\s)$ is bounded and $\abs{\Re (\s)} \to \infty$.

We use similar notation as in \cite{H2025}*{Thm.~11.21}, which is the case $\q_1 = \hdots = \q_n = 1$.

\begin{thm}[High-energy estimates] \label{thm: high-energy estimate}
Let $a > 0$ and let $\mP$ and $\mX_a$ be as in Assumption~\ref{ass: setting Fredholm}.
Fix $\a_+ \geq \a_- \in \R$ and let $s > \frac12 + \frac2{\check \q_m} \a_+ + \tilde \b_\max(0)$. 
Then there are constants $C, C' > 0$ such that 
\[
	\norm{u}_{\bar H^s_{{\abs{\Re(\s)}}^{-1}}(\mX_a)}
		\leq C \abs{\Re(\s)}^{-1} \norm{\hat \mP(\s) u}_{\bar H^{s-1}_{\abs{\Re(\s)}^{-1}}}
\]
for $- \a_+ \leq \Im(\s) \leq - \a_-$ and $\abs{\Re(\s)} \geq C'$.
In particular, for all QNM frequencies $\s \in \C$ (as in Corollary~\ref{cor: discrete QNM}) with $-\a_+ \leq \Im(\s) \leq - \a_-$, we have $\abs{\Re(\s)} \leq C'$.
\end{thm}

\begin{proof}
We first reformulate the problem in semiclassical terms.
Defining $h := \abs{\Re(\s)}^{-1}$ and writing $\s = \Re(\s) - i \a = h^{-1} \left( \pm 1 - i h \a \right)$ for $\pm \Re(\s) > 0$, we introduce the semiclassical operator
\[
	\mP_{\pm, \a}
		:= h^2 \hat \mP\left(h^{-1} (\pm 1 - i h \a) \right),
\]
for a fixed $\a \in \R$.
Combining this with \eqref{eq: hat P sigma explicit}, we see that the (semiclassical) principal symbol is given by
\[
	\p_\hbar
		:= \s_{\hbar}\left( \mP_{\pm, \a} \right)(\xi)
		= G\left(\mp \md \tau + \xi \right),
\]
for $\pm \Re(\s) > 0$, where $G$ was introduced in \eqref{eq: G definition}.
The semiclassical characteristic set of $\mP_{\pm, \a} $ is $\S_{\pm 1} \subset \oT \mX_a$, and the Hamiltonian vector field $H_{\p_\hbar}$ is given by $\mathsf H_{\pm 1}$ defined in \eqref{eq: Ham vf}.
Since $\mathsf H_1|_{(y, \xi)} = - \mathsf H_{-1}|_{(y, -\xi)}$, it suffices to consider the case with the $+$ sign, i.e.~the case $\Re(\s) > 0$. 
The estimate in the assertion can now be rewritten as the semiclassical estimate
\[
	\norm{u}_{\bar H^s_h (\mX_a)}
		\leq C h^{-1} \norm{\mP_{+, \a} u}_{\bar H^{s-1}_h(\mX_a)}
\]
for $h < h_0$ and $\a \in [\a_-, \a_+]$, where $h_0 = \frac1{C'}$.

The proof of this estimate is analogous to the proof of the Fredholm estimate for Theorem~\ref{thm: Fredholm property}, but with semiclassical estimates.
Let $a > 0$ and $b := a/6$.
Fix $\psi, \chi \in C_c^\infty(\mX_a)$, such that $\psi(y) = 1$ for $y \in \mX_{2b}$ and $\supp(\psi) \subset \mX_{3b}$, and such that $\chi(y) = 1$ for $y \in \mX_{4b}$ and $\supp(\chi) \subset \mX_{5b}$.
We set $s_0 := \frac12 + \frac2{\check \q_m} \a + \tilde \b_\max(0) > \frac12 + \tilde \b_\max(\sigma)$.

Starting at the source $\cR_{+, m}$, where we get a priori semiclassical regularity by Theorem~\ref{thm: saddle points semiclassical}, we propagate regularity inductively to all saddle manifolds $\cR_{+, l}$ for $l = 1, \hdots, m-1$ along the Hamiltonian vector field $H_{\p_\hbar}$, using the semiclassical saddle point estimate Theorem~\ref{thm: saddle points semiclassical} and standard semiclassical propagation of singularities \cite{H2025}*{Thm.~8.27}, by the dynamical information of Theorem~\ref{thm: bicharacteristic flow}.
We similarly obtain and propagate regularity from the sink $\cR_{-, m}$ and to all saddle manifolds $\cR_{-, l}$ for $l = 1, \hdots, m-1$ along $-\H_{\p_\hbar}$.
Finally, standard propagation of singularities \cite{H2025}*{Thm.~8.27} and semiclassical elliptic regularity theory \cite{H2025}*{Prop.~6.61}, together with the dynamical information of Theorem~\ref{thm: bicharacteristic flow}, implies that
\begin{equation} \label{eq: semiclassical interior estimate}
	\norm{u}_{\bar H_h^s(\mX_{2b})}
		\leq C \left( h^{-1} \norm{\chi \mP_{+, \a} u}_{\bar H^{s-1}_h(\mX_a)} + h^N \norm{\chi u}_{\bar H_h^{s_0}(\mX_a)} \right)
\end{equation}
analogous to \eqref{eq: X 3b estimate}.
We would like to get rid of the last term.
Since $\d_\tau$ is \emph{spacelike} in the region $\R \times \left( \mX_a \backslash \mX_b \right)$, \cite{H2025}*{Prop.~9.42} implies that $\mP_{+, \a}$ is a semiclassical wave-type operator in this region. We may therefore apply the semiclassical energy estimate \cite{H2025}*{Cor.~9.46} in the region $\mX_a \backslash \mX_b$ to conclude that
\[
	\norm{u}_{\bar H_h^s(\mX_a \backslash \mX_b)}
		\leq C \left( h^{-1} \norm{\mP_{+, \a} u}_{\bar H^{s-1}_h(\mX_a)} + \norm{u}_{\bar H_h^s(\mX_{2b})} \right).
\]
Combining this with \eqref{eq: semiclassical interior estimate} proves the asserted estimate for a fixed $\a$.
The proof that we can choose the constant $C > 0$ independently of the choice of $\a \in [\a_-, \a_+]$ is the same as in the proof of \cite{H2025}*{Thm.~11.21}.
\end{proof}

\begin{cor}[Quasinormal mode frequencies] \label{cor: QNM freq}
For every $\a \in \R$, the set of $\s \in \C$ with $\Im(\s) > - \a$, such that there is a $u \in C^\infty(M)$ satisfying \eqref{eq: QNM t squared modified} for some $k \in \N_0$ is finite.
\end{cor}
\begin{proof}
By Corollary~\ref{cor: resolvent large im sigma}, there is a $C > 0$, such that all quasinormal mode frequencies $\s \in \C$ satisfy $\Im(\s) \leq C$.
We may therefore apply Theorem~\ref{thm: high-energy estimate} with $\a_- := - C$ and $\a_+ := \a$, to conclude that $\abs{\Re(\s)} \leq C'$ for some $C' > 0$ and all quasinormal mode frequencies $\s \in \C$ with $- \a \leq \Im(\s) \leq C$.
By Corollary~\ref{cor: discrete QNM}, we know that the set of such $\s$ is discrete, which proves the assertion.
\end{proof}

We may now prove the first main theorem for linear wave equations. 

\begin{proof}[Proof of Theorem~\ref{thm: main QNM}]
The proof is an immediate consequence of Corollary~\ref{cor: discrete QNM}, Remark~\ref{rmk: relating QNMs} and Corollary~\ref{cor: QNM freq}.
\end{proof}

\begin{cor}[Late-time asymptotics] \label{cor: late-time asymptotics}
Let $a > 0$ and let $\mP$ and $\mX_a$ be as in Assumption~\ref{ass: setting Fredholm}.
Let $\mathrm{QNM}(\mP)$ denote the quasinormal mode frequencies with respect to $\mP$, provided by Corollary~\ref{cor: QNM freq}.
For each $\s \in \mathrm{QNM}(\mP)$, choose a basis $v_{\s, 1}, \hdots, v_{\s, l(\s)}$ for the finite dimensional space of generalized QNMs, provided by Corollary~\ref{cor: discrete QNM}.
Let $\a > 0$ and let $w \in C^\infty_c\left( \mY_a \right)$ with $\supp(w) \subset \{ \tau \geq 0 \}$.
Let $v \in C^\infty(\mY_a)$ denote the unique solution to $\mP u = w$ with $\supp(u) \subset \{ \tau \geq 0 \}$. 
Then there are unique constants $c_{\s, k}$ for $1 \leq k \leq l(\s)$, and all $\s \in \mathrm{QNM}(\mP)$, such that
\[
	u(\tau, y)
		= \sum_{\underset{\s \in \mathrm{QNM}(\mP)}{\Im(\s) > - \a}} c_{\s, k}v_{\s, k}(\tau, y) + \tilde v_\a(\tau, y),
\]
where
\[
	\abs{\d_\tau^j \d_y^\kappa \tilde v_\a(\tau, y)}
		\leq C_{\a, j, \kappa} e^{-\a \tau},
\]
for all $j \in \N_0$, $\kappa \in \N_0^n$, $\tau \geq 0$, $y \in \X_a$.
\end{cor}

\begin{proof}
This is proven via the same complex contour deformation argument developed by Vasy in \cite{V2013}, with a detailed presentation in \cite{H2025}*{Sec.~10.5}. 
Specifically, we refer to the proofs of \cite{H2025}*{Thm.~11.24} and \cite{H2025}*{Thm.~11.24}, which is completely the same in our case, where our high-energy estimate Theorem~\ref{thm: high-energy estimate} is used in place of \cite{H2025}*{Thm.~11.21}.
\end{proof}

\begin{remark}
With little effort, Corollary~\ref{cor: late-time asymptotics} can be improved by allowing for suitably decaying source functions $w$ with only Sobolev regularity, c.f.~\cite{H2025}*{Thm.~11.24}.
For the applications in this paper, we however only need this $C^\infty$-version with compactly supported source function $w$.
\end{remark}

Our second main result for linear wave equations is a consequence of Corollary~\ref{cor: late-time asymptotics}. 

\begin{proof}[Proof of Theorem~\ref{thm: main as exp}]
Since $J^+(\gamma) \subseteq \mY_a$, this is an immediate consequence of Corollary~\ref{cor: late-time asymptotics}.
\end{proof}

\section{Geometric wave equations} \label{sec: QNM geometric waves}

Let in this section $(M, g)$ denote a Kasner-type spacetime with standard coordinates $t$ and $x_1, \hdots, x_n$, see Definition~\ref{def: Kasner-type spacetime}.
Let moreover $\P$ denote a linear wave operator, i.e.~a linear system of differential operators with principal symbol given by $g^{-1}(\xi, \xi)$.
Many linear wave operators of interest are \emph{geometric} in the sense that they are derived from the geometry. 
Since $\T$ is a homothetic Killing vector field such that $\L_\T g = - 2 g$, it is natural to assume that
\begin{equation} \label{eq: T commutator}
	[\T, t^2 \P]
		= 0,
\end{equation}
as we did in Section~\ref{sec: asymptotic expansion}.
Since $\d_{x_j}$ are Killing vector fields for $j = 1, \hdots, n$, a linear wave operator which is associated with the geometry should also satisfy
\begin{equation}\label{eq: x commutators}
	[\d_{x_j}, \P]
		= 0,
\end{equation}
for $j = 1, \hdots, n$.
\begin{definition}
A \emph{geometric wave operator} $\P$ is a linear wave operator satisfying \eqref{eq: T commutator} and \eqref{eq: x commutators}.
\end{definition}
An important example of a geometric wave operator is the Lichnerowicz-d'Alembert operator that we study in the next subsection.

Given any $a > 0$, recall the definition of the region $\mY_a \subset M$ in \eqref{eq: Y a}, which contains $J^+(\gamma)$.
The following is the key to compute both QNMs and their frequencies.

\begin{thm}[The annihilation and creation operators] \label{thm: creation and annihilation}
Assume that $(M, g)$ is a Kasner-type spacetime and that $\P$ is a geometric wave operator on $\C^m$-valued functions, for some $m \in \N$.
Let $a > 0$ and let $\A \subset C^\infty(\mY_a; \C^m)$ be the discrete subset of generalized QNMs, as provided by Corollary~\ref{cor: discrete QNM}.
Then the following holds:
\begin{itemize}
	\item For any $j = 1, \hdots, n$, the operator $\d_{x_j}$ defines a linear map $\d_{x_j}: \A \to \A$, such that if $u \in \A$ with frequency $\s \in \C$, then $\d_{x_j} u$ is a generalized QNM of frequency $\s + (1-p_j)i$.
	\item For any $j = 1, \hdots, n$, there is a linear map $\Cr_j: \A \to \A$, such that if $u \in \A$ with frequency $\s \in \C$, then $\Cr_j (u)$ is a generalized QNM of frequency $\s - (1 - p_j) i$.
\end{itemize}
Moreover, $\d_{x_j} \Cr_j = \id$ and $[\d_{x_j}, \Cr_k] = 0$ for all $j \neq k$.
\end{thm}

\begin{definition} \label{def: ann + creat}
We call $\d_{x_j}$ the \emph{annihilation operator}, and $\Cr_j$ the \emph{creation operator}.
\end{definition}

The creation operator $\Cr_j$ is a particular choice of primitive function with respect to the $x_j$-variable, as we will see in the proof.

\begin{proof}[Proof of Theorem~\ref{thm: creation and annihilation}]
We begin by studying the annihilation operator. 
Note first that, since $[\d_{x_j}, \P] = 0$ and $\P u = 0$, it follows that $\P \left( \d_{x_j} u \right) = 0$.
To see that $\d_{x_j} u$ is a QNM of the shifted frequency, we proceed by induction over the order $k$, using that $[\d_{x_j}, \T + i \s] = - (1 - p_j) \d_{x_j}$.
In case $k = 1$, if $\left( \T + i \s \right) u = 0$, then $0 = \d_{x_j} \left( \T + i \s \right) u = \left( \T + i \left( \s + (1 - p_j)i \right) \right) \d_{x_j}u$.
Now assume that for some fixed $k$, $\left( \T + i \s \right)^k u = 0$ implies that $\left( \T + i \left( \s + (1 - p_j)i \right) \right)^k \d_{x_j}u = 0$.
If now $\left( \T + i \s \right)^{k+1} u = 0$, then $\left( \T + i \s \right)^k \left( \T + i \s \right) u = 0$, and the induction assumption implies that
\[
	0
		= \left( \T + i \left( \s + (1 - p_j)i \right) \right)^k \d_{x_j} \left( \T + i \s \right) u 
		= \left( \T + i \left( \s + (1 - p_j)i \right) \right)^{k+1} u,
\]
as claimed.

We now want to construct the creation operator.
For simplicity of notation, let us assume that $j = 1$ in the proof.
We start with the ansatz
\[
	\Cr_1 u(t, x_1, \hdots, x_n)
		= \int_0^{x_1} u(t, y, x_2, \hdots, x_n) \md y + v(t, x_2, \hdots, x_n),
\]
and compute that
\begin{align*}
	\T \int_0^{x_1} u(t, y, x_2, \hdots, x_n) \md y
		= & \ - (1 - p_1) x_1 \d_{x_1} \int_0^{x_1} u(t, y, x_2, \hdots, x_n) \md y \\
		& \ + \int_0^{x_1} \left( \T + (1-p_1)y \d_y \right) u(t, y, x_2, \hdots, x_n) \md y \\
		= & \ \int_0^{x_1} \left( \T - (1 - p_1) \right) u(t, y, x_2, \hdots, x_n) \md y.
\end{align*}
It follows that $\left( \T + i \s \right)^k u = 0$ if and only if $\left( \T + i \left( \s - (1-p_1)i \right) \right)^k \int_0^{x_1} u(t, y, x_2, \hdots, x_n) \md y = 0$.
Consequently, 
\begin{align*}
	\left( \T + i \left( \s - (1-p_1)i \right) \right)^k \Cr_1 u 
		= 0 \quad 
		& \Longleftrightarrow \quad \left( \T + i \left( \s - (1-p_1)i \right) \right)^k v
		= 0, \\
	\P \Cr_1 u = 0 \quad 
		&\Longleftrightarrow \quad
		\P v = - \P \int_0^{x_1} u(t, y, x_2, \hdots, x_n) \md y.
\end{align*}
This is solvable by Proposition~\ref{prop: QNM wave equations}, with some order $k' \geq k$, since $- \P \int_0^{x_1} u(t, y, x_2, \hdots, x_n) \md y$ is a generalized QNM with the shifted frequency.
\end{proof}

Theorem~\ref{thm: creation and annihilation} gives a way of producing QNMs and frequencies for geometric wave operators.
This is the key in the proof of the next corollary.
In order to formulate it, we introduce the notation
\[
	\abs{\o}_p
		:= \sum_{j = 1}^n \o_j (1-p_j)
\]
for any $\o \in \N_0^n$.

\begin{cor}[Structure of quasinormal modes for wave equations] \label{cor: structure of QNMs}
Let $(M, g)$ and $\P$ be as in Theorem~\ref{thm: creation and annihilation}.
If $u \in C^\infty(\mY_a; \C^m)$ is a generalized QNM with frequency $\s \in \C$ and order $k$, then there is a $K \in \N_0$, such that
\begin{equation} \label{eq: QNM form}
	u(t, x)
		= t^{i \s} \sum_{\abs{\o}_p \leq K} t^{-\abs{\o}_p} x^\o a_\o(t), 
\end{equation}
where $a_\o(t) = \sum_{j = 0}^{k-1} a_{\o, j} \log(t)^j$ for some constant $a_{\o,j} \in \C^m$ for all $j = 0, \hdots, k-1$ and all $\o$, and where $x^\o := x_1^{\o_1} \cdots x_n^{\o_n}$.
Consequently, $u$ is real analytic and extends uniquely to a generalized QNM on all of $M$. 
Moreover, if $\a \in \N_0^n$ is such that $a_\o(t) = 0$ for all $\o > \a$, then $\P\left( t^{i\s - \abs{\a}_p} a_\a(t) \right) = 0$.

\end{cor}

This shows in particular that every quasinormal mode of a geometric wave operator is a \emph{polynomial} in the variables $x_1, \hdots, x_n$.
The key observation in the proof of this corollary is that the annihilation operator indeed \emph{annihilates} quasinormal modes after a finite number of applications, due to Corollary~\ref{cor: resolvent large im sigma}.

\begin{proof}[Proof of Corollary~\ref{cor: structure of QNMs}]
Let $u$ be a quasinormal mode with respect to $\s \in \C$.
For any $\a \in \N_0^n$, Theorem~\ref{thm: creation and annihilation} implies that $\d_{x_1}^{\a_1} \hdots \d_{x_n}^{\a_n} u$ is a generalized QNM of frequency $\s + i \sum_{j = 1}^n \a_j(1-p_j)$.
Moreover, Corollary~\ref{cor: resolvent large im sigma} implies that there is a $K \in \N_0$, only dependent on $\P$, such that $\Im \left( \s + i \sum_{j = 1}^n \a_j(1-p_j) \right) \leq K$, or equivalently that $\abs{\a}_p = \sum_{j = 1}^n \a_j(1-p_j) \leq K - \Im(\s)$.
Therefore, any generalized QNM $u$ with frequency $\s \in \C$ must satisfy $\d_{x_1}^{\a_1} \hdots \d_{x_n}^{\a_n} u = 0$, for any $\a \in \N_0^n$ such that $\abs{\a}_p > K - \Im(\s)$.
Consequently, $u$ has to be a polynomial in $x_1, \hdots, x_n$ with coefficients depending only on $t$.
We may therefore write
\[
	u(t, x)
		= t^{i \s} \sum_{\abs{\o}_p \leq K - \im(\s)} t^{- \abs{\o}_p} x^\o a_\o(t).
\]
Since $\T \left( t^{- \abs{\o}_p} x^\o\right) = 0$, the generalized QNM condition $(\T + i \s)^k u(t, x) = 0$ is equivalent to $\T^k a_\o(t) = 0$, which is equivalent to $\left( t\d_t \right)^k a_\o(t) = 0$, for all $\abs{\o}_p \leq K - \Im(\s)$.
Integrating this, we find the expressions for $a_\o(t)$.
For the second assertion, if $\a \in \N_0^n$ is such that $a_\o(t) = 0$ for all $\o > \a$, then 
\[
	\P\left( t^{i\s - \abs{\a}_p} a_\a(t) \right)
		= \P \left( \d_{x_1}^{\a_1} \cdots \d_{x_n}^{\a_n} u \right) 
		= \d_{x_1}^{\a_1} \cdots \d_{x_n}^{\a_n} \P u
		= 0,
\]
as claimed.
\end{proof}
The next corollary will be an algebraic characterization of all possible QNM frequencies. 
For this, the following observation will be useful:

\begin{remark} \label{rmk: wave const coeff}
Note that any linear wave operator $\P$ on a Kasner-type spacetime takes the form
\begin{equation} \label{eq: geom wave char}
	t^2 \P
		= \id \left( t\d_t \right)^2 - \id \sum_{j = 1}^n t^{2(1-p_j)} \d_{x_j}^2 + \A t\d_t + \sum_{j = 1}^n \B_j t^{1-p_j} \d_{x_j} + \Cm,
\end{equation}
for some matrices $\A, \B_1, \hdots, \B_n, \Cm$, where \eqref{eq: T commutator} implies that the coefficients are independent of $t$.
For a \emph{geometric} operator, the additional assumption \eqref{eq: x commutators} implies that the coefficient matrices $\A, \B_1, \hdots, \B_n, \Cm$ are in fact constant. 
\end{remark}

\begin{cor}[Quasinormal mode frequencies for wave equations] \label{cor: QNM frequencies}
Let $(M, g)$ and $\P$ be as in Theorem~\ref{thm: creation and annihilation}.
The quasinormal mode frequencies are precisely the numbers
\[
	\s
		= - i \left( \lambda + \abs{\a}_p \right),
\]
where $\a \in \N_0^n$ and $\lambda \in \C$ is such that 
\[
	\lambda^2 \id + \lambda \A + \Cm
\]
is non-invertible, where $\A$ and $\Cm$ are as in Remark~\ref{rmk: wave const coeff}.
Consequently, the real part of the frequency of any generalized QNM coincides with the real part of the frequency of a generalized QNM with no $x$-dependence.
\end{cor}

\begin{proof}[Proof of Corollary~\ref{cor: QNM frequencies}]
Let $u$ be a generalized QNM and let $\s \in \C$ be the QNM frequency.
By Remark~\ref{rmk: gen QNM are QNM}, we may assume that $u$ is a (not generalized) QNM.
By Corollary~\ref{cor: structure of QNMs}, if $u$ is not identically zero, then there is an $\a \in \N^n_0$ such that $a_\o = 0$ for all $\o > \a$ and $a_\a \neq 0$, and
\begin{align*}
	0
		= \P \left( t^{i \s} a_\a t^{-\abs{\a}_p} \right) 
		= t^\lambda \left( \lambda^2 + \A \lambda + \Cm \right)a_\a
\end{align*}
for $\lambda := i \s -\abs{\a}_p$.

Conversely, let $\s := - i (\lambda + \abs{\a}_p)$, where $\left( \lambda^2 + \A \lambda + \Cm \right)v = 0$ for a $v \in \C^m$.
We then define $u := \Cr_1^{\a_1} \cdots \Cr_n^{\a_n} \left( t^\lambda v \right)$, where $\Cr_1, \hdots, \Cr_n$ are the creation operators.
Since $t^\lambda v$ is a QNM of frequency $-i \lambda$, Theorem~\ref{thm: creation and annihilation} implies that $u$ is a generalized QNM of frequency $\s = - i \left( \lambda + \abs{\a}_p \right)$, completing the proof.
\end{proof}

\begin{example} \label{ex: scalar wave}
If $\P = \Box + \frac{\b}{t^2}$, with $\b \in \R$, then $\A = 0$ and $\Cm = \b$.
By Corollary~\ref{cor: QNM frequencies}, the QNM frequencies of $t^2 \P$ are therefore given by all $- i \left(\lambda + \abs{\a}_p \right)$, where $\a \in \N_0^n$ and $\lambda^2 = \b$.
If $\b \geq 0$, then all QNM frequencies are purely imaginary.
If $\b < 0$, then the QNM frequencies have a non-trivial real part, given by $\pm \sqrt{\abs{\b}}$.
In the special case of the scalar wave equation in de Sitter space, that is $p_1 = \hdots = p_n = 0$ and $\b = 0$, then this set is $- i \N_0$ (c.f.~\cite{H2025}*{Ex.~11.8}).
\end{example}

\begin{remark}
Neither of the above presents a way how to compute the order $k$ of a generalized QNM.
Note, however, that this is in principle possible by inserting \eqref{eq: QNM form} into a geometric wave equation and comparing terms of the same polynomial degree.
\end{remark}

\section{The linearized Einstein equation} \label{sec: QNM lin Ein}

Let in this section $(M, g)$ be a non-flat Kasner spacetime, see Definition~\ref{def: Kasner}. 
In other words, $(M,g)$ is a Kasner-type spacetime satisfying the Kasner relations \eqref{eq: Kasner relations} and $p_j < 1$ for all $j$, which ensures that the Einstein vacuum equation is satisfied and excludes the flat Kasner spacetime.
The purpose of this section is to apply the theory for linear wave equations developed in Section~\ref{sec: asymptotic expansion} and Section~\ref{sec: QNM geometric waves}, specifically Theorem~\ref{thm: main QNM} and Theorem~\ref{thm: main as exp}, to study generalized QNMs for the linearized Einstein equation, i.e.~solutions $u \in C^\infty(M)$ to
\begin{equation} \label{eq: QNM lin Ein}
	\De \Ric_g(u)
		= 0, \qquad
	\left( \L_T + i \s \right)^k u
		= 0,
\end{equation}
for a QNM frequency $\s \in \C$ and order $k \in \N_0$.
The goal is to prove Theorem~\ref{thm: main linearized Einstein, general dim}, Theorem~\ref{thm: self-similar solutions} and Theorem~\ref{thm: main linearized Einstein}.
Note that Theorem~\ref{thm: main linearized Einstein, general dim} says in particular that we get an asymptotic expansion only using QNMs where $i\s \in \R$, meaning it suffices to consider purely imaginary $\s$.
This is a non-trivial fact, as linear wave equations in general require complex frequencies, see Example~\ref{ex: scalar wave}.
It will be a consequence of Theorem~\ref{thm: AVTD system} below, in light of Corollary~\ref{cor: QNM frequencies}.

The first step is to relate QNMs of the linearized Einstein equation, i.e.~solutions to \eqref{eq: QNM lin Ein}, to QNMs for a linear wave equation with the gauge condition \eqref{eq: gauge condition}.
This follows a standard gauge choice for the linearized Einstein equation.
We define the Lichnerowicz-d'Alembert operator as
\[
	\Box_L
		:= \n^*\n - 2 \mathring {\mathrm R},
\]
where $( \mathring {\mathrm R} u)_{ij} := \mathrm R_{kijl}g^{km}g^{lr}u_{mr}$, where $\mathrm R$ denotes the Riemann curvature tensor.
A straightforward computation, see e.g.~\cite{B2008}*{Thm.~1.174} and \cite{P2019}, shows that
\begin{equation} \label{eq: lin Ein equation}
	\De \Ric_g(u)
		= \frac12 \left( \Box_L u + \L_{\div_g\left( u - \frac12\tr_g(u) g \right)^\sharp}g \right).
\end{equation}
Since $(M, g)$ is a globally hyperbolic spacetime, there is a vector field $X \in C^\infty(M)$ satisfying
\begin{equation} \label{eq: constr X}
	\Box_g X
		= \div_g \left( u - \frac12 \tr_g(u) g \right).
\end{equation}
Defining $\tilde u := u + \L_Xg$, a standard computation using $\Ric_g = 0$ shows that
\[
	\div_g\left(\tilde u - \frac12 \tr_g \left( \tilde u \right) g \right)
		= - \Box X + \div_g \left( u - \frac12 \tr_g(u) g \right)
		= 0.
\]
This is the gauge condition \eqref{eq: gauge condition} used in our main results.
Inserting the gauge condition in \eqref{eq: lin Ein equation}, the equation $\De \Ric_g(u) = 0$ reduces to $\Box_L u = 0$.
As we will see, applying Proposition~\ref{prop: QNM wave equations} to \eqref{eq: constr X}, we can similarly fix the gauge for the generalized QNMs, so it suffices to consider solutions $u \in C^\infty(M)$ to 
\begin{equation} \label{eq: Box L QNM}
	\Box_L u
		= 0, 
	\qquad
	\left( \L_T + i \s \right)^k u
		= 0,
\end{equation}
for a QNM frequency $\s \in \C$ and order $k \in \N_0$, satisfying the gauge condition \eqref{eq: gauge condition}.

\subsection{Quasinormal modes with no spatial dependence}

Motivated by the last assertion in Corollary~\ref{cor: structure of QNMs}, a natural starting point is to compute all $x$-independent generalized QNMs to $\Box_L$, i.e.~solutions to \eqref{eq: Box L QNM}, which satisfy the gauge condition \eqref{eq: gauge condition}.

\begin{thm} \label{thm: AVTD system}
Assume that $(M, g)$ is a non-flat Kasner spacetime.
Let $u \in C^\infty(M)$ be a solution to
\[
	\Box_L u
		= 0, 
	\qquad
	\div_g \left(u - \frac12 \tr_g(u) g \right)
		= 0, 
	\qquad
	\L_{\d_{x_j}} u 
		= 0,
\]
for $j = 1, \hdots, n$.
Then $u$ is a linear combination of the following $n(n+2)$ generalized QNMs:
\begin{itemize}
	\item $g'(d) := 2 \log(t) \sum_{j = 1}^n d_j t^{2p_j} \md x_j^2$, where $\sum_{j = 1}^n p_j d_j = \sum_{j = 1}^n d_j = 0$, with frequency $\s = - 2 i$ and order $2$, 
	\item $\L_{\frac1t \d_t}g$, with frequency $\s = 0$ and order $1$,
	\item $\L_{t^{-2p_j}\d_{x_j}}g$ if $p_j \neq 0$, otherwise $\L_{\log(t)\d_{x_j}}g$, with frequency $\s = - (1-p_j)i$ and order $1$, for $j = 1, \hdots, n$,
	\item $\L_{x_j \d_{x_k}}g$, with frequency $\s = (- 2 + p_j - p_k)i$, for $j, k = 1, \hdots, n$, with $j \neq k$,
	\item $\L_{\left( \log(t) - 1 \right)t\d_t + \sum_{j = 1}^n p_j x_j \d_{x_j}}g$, with frequency $\s = -2i$ and order $2$,
	\item $\L_{\sum_{j = 1}^n d_j x_j \d_{x_j}}g$, where $\sum_{j = 1}^n d_j = 0$, with frequency $\s = -2i$ and order $1$,
	\item $\L_{t \d_t} g$, with frequency $\s = -2i$ and order $1$.
\end{itemize} 
\end{thm}

The solutions $g'(d)$ in Theorem~\ref{thm: AVTD system} are simply the linearized Kasner metrics in higher dimensions.
Indeed, if $\sum_{j = 1}^n p_j(r)^2 = \sum_{j = 1}^n p_j(r) = 1$, with $p_j(0) = p_j$, and $p_j'(0) = d_j$, then 
\[
	g'(d)
		= \frac{\md }{\md r}|_{r = 0} \left( -\md t^2 + \sum_{j = 1}^n t^{2p_j(r)}\md x_j^2 \right)
		= 2 \log(t) \sum_{j = 1}^n d_j t^{2p_j} \md x_j^2,
\]
where $\sum_{j = 1}^n p_j d_j = \sum_{j = 1}^n d_j = 0$.

\begin{lemma} \label{le: L-d'Alembert}
Assume that $(M, g)$ is a non-flat Kasner spacetime of dimension $n+1$.
Define the orthonormal frame for $g$ given by $e_0 := \d_t$ and $e_j := \ t^{-p_j}\d_{x_j}$, for $j = 1, \hdots, n$, and the dual frame $e^0 := \md t$ and $e^j := \ t^{p_j}\md x_j$ for $j = 1, \hdots, n$.
Writing $u = \sum_{\a, \b = 0}^n u_{\a \b}e^\a \otimes e^\b$, then
\begin{align*}
	t^2 \Box_L u
		= & \ \left( t^2 \Box u_{00} + 4 \sum_{k = 1}^n p_k t e_k u_{k0} - 2 \left( u_{00} + \sum_{k = 1}^n p_k u_{kk} \right) \right) e^0 \otimes e^0 \\
		& \ + \sum_{j = 1}^n \left( t^2 \Box u_{jj} + 4 p_j t e_j u_{0j} - 2 p_j \left( u_{00} + \sum_{k = 1}^n p_k u_{kk} \right) \right) e^j \otimes e^j \\
		& \ + 2 \sum_{j = 1}^n \left( t^2 \Box u_{0j} + 2 \left( \sum_{l = 1}^n p_l te_l u_{jl} + p_j t e_j u_{00} \right) - \left( p_j + 1 \right)^2 u_{j0} \right) e^0 \otimes_s e^j  \\
		& \ + \sum_{j \neq k} \left( t^2 \Box u_{jk} + 2 \left( p_j t e_j u_{0k} + p_k t e_k u_{0j} \right) - \left( p_j - p_k \right)^2 u_{jk} \right) e^j \otimes e^k,
\end{align*}
where $\Box$ denotes the d'Alembert operator on scalar-valued functions. 
\end{lemma}
\begin{proof}[Proof of Lemma~\ref{le: L-d'Alembert}]
This is a straightforward computation.
\end{proof}

\begin{proof}[Proof of Theorem~\ref{thm: AVTD system}]
From Lemma~\ref{le: L-d'Alembert}, we see that $\Box_L u = 0$ and $\L_{\d_{x_j}} u = 0$, for all $j = 1, \hdots, n$, if and only if the components $u_{\a\b}$ of $u$ are independent of $x_1, \hdots, x_n$ and satisfy
\begin{align*}
	\left( t\d_t \right)^2 u_{00} - 2 \left( u_{00} + \sum_{k = 1}^n p_k u_{kk} \right) 
		&= 0, &&
	\left( t\d_t \right)^2 u_{jj} - 2 p_j \left( u_{00} + \sum_{k = 1}^n p_k u_{kk} \right)
		= 0, \\
		\left( t\d_t \right)^2 u_{0j} - \left( p_j + 1 \right)^2 u_{0j}
		&= 0, &&
	\left( t\d_t \right)^2 u_{jk} - \left( p_j - p_k \right)^2 u_{jk}
		= 0,
\end{align*}
for all $j, k = 1, \hdots, n$ with $j \neq k$.
The general solution to this system of ODE is
\begin{align*}
	u_{00}
		&= a_0 + b_0\log(t) + \frac{c_1}2 t^2 + \frac{c_2}2 t^{-2}, &&
	u_{jj}
		= a_j + b_j\log(t) + \frac{c_1}2 p_j t^2 + \frac{c_2}2p_j t^{-2}, \\
	u_{0j}
		&= a_{0j} t^{1 + p_j} + b_{0j} t^{-(1+p_j)}, && 
	u_{jk}
		= a_{jk}t^{p_j - p_k} + b_{jk} t^{p_k - p_j},
\end{align*}
for some constants $a_0, a_j, b_0, b_j, a_{0, j}, b_{0,j}, a_{jk}, b_{jk}, c_1, c_2 \in \R$, where $j \neq k$, under the algebraic constraints $a_0 + \sum_{k = 1}^n p_j a_j = b_0 + \sum_{k = 1}^n p_j b_j = 0$.

We have so far only used that $\Box_L u = 0$.
One can compute that $\div_g \left(u - \frac12 \tr_g(u) g \right) = 0$ is equivalent to $t \d_t \left( u_{00} + \sum_{j = 1}^n u_{jj} \right) + \left( u_{00} + \sum_{j = 1}^n p_j u_{jj} \right) = 0$ and $\left( t \d_t + (1 + p_j) \right) u_{0j} = 0$, for $j = 1, \hdots, n$.
This implies that in fact $c_1 = 0$ and $a_{0j} = 0$ for all $j = 1, \hdots, n$.
Moreover, we also get the extra constraint that $b_0 + \sum_{j = 1}^n b_j = 0$.
To summarize, we conclude that
\begin{align}
	u_{00}
		&= a_0 + b_0 \log(t) + \frac{c_2}2 t^{-2}, \label{eq: 00 solution} \\
	u_{jj}
		&= a_j + b_j \log(t) + \frac{c_2}2p_j t^{-2}, \label{eq: jj solution} \\
	u_{0j}
		&= b_{0j} t^{-(1+p_j)}, \label{eq: 0j solution} \\
	u_{jk}
		&= a_{jk}t^{p_j - p_k} + b_{jk} t^{p_k - p_j}, \label{eq: jk solution}
\end{align}
for all $j, k = 1, \hdots, n$ with $j \neq k$, together with the algebraic constraints
\begin{equation} \label{eq: alg constraints}
	a_0 + \sum_{k = 1}^n p_k a_j
		= b_0 + \sum_{k = 1}^n p_j b_j
		= b_0 + \sum_{j = 1}^n b_j
		= 0.
\end{equation}
The solutions corresponding to \eqref{eq: 0j solution} are given by $t^{-(1+p_j)} \md t \otimes_s t^{p_j}\md x_j = - \frac1{4p_j}\L_{t^{-2p_j}\d_{x_j}}g$, for $j = 1, \hdots, n$, unless $p_j = 0$, in which case $t^{-1} \md t \otimes_s \md x_j = \frac12 \L_{\log(t)\d_{x_j}}g$
The solutions corresponding to \eqref{eq: jk solution} are $\md x_j \otimes_s t^{2p_k} \md x_k = \frac12 \L_{x_j \d_{x_k}}g$, for $j, k = 1, \hdots, n$, with $j \neq k$, respectively. 
The solution corresponding to the $c_2$-factor in \eqref{eq: 00 solution} and \eqref{eq: jj solution} is $\frac1{t^2} \left( \md t \otimes \md t + \sum_{j = 1}^n p_j t^{2p_j} \md x_j \otimes \md x_j \right) = \frac12 \L_{\frac1t \d_t}g$.
It therefore remains to study solutions with $u_{00} = a_0 + b_0 \log(t)$ and $u_{jj} = a_j + b_j \log(t)$, under the algebraic constraints \eqref{eq: alg constraints}.
Let us first assume that $b_0 = 0$.
Then the constraints for $(b_1, \hdots, b_n) \in \R^n$ are solved precisely by any $(d_1, \hdots, d_n) \in \R^n$ such that $\sum_{j = 1}^n p_j d_j = \sum_{j = 1}^n d_j = 0$.
The corresponding solution is indeed given by a linearization of the Kasner metric $\log(t) \sum_{j = 1}^n d_j t^{2p_j} \md x_j^2 = \frac12 g'(d)$.
There is only one more linearly independent solution to the constraints in $b$, for example $b_0 = - 1$ and $b_j = p_j$ for $j = 1, \hdots, n$.
The corresponding solution is given by $\log(t) \left( - \md t^2 + \sum_{j = 1}^n p_j t^{2p_j} \md x_j^2 \right) = \L_{\left( \log(t) - 1 \right)t\d_t + \sum_{j = 1}^n p_j x_j \d_{x_j}}g$.
We proceed by solving the constraint equations for $a$.
First assuming that $a_0 = 0$ implies that $(a_1, \hdots, a_n) = (d_1, \hdots, d_n)$, for any solution to $\sum_{j = 1}^n d_j p_j = 0$.
The corresponding solution is $\sum_{j = 1}^n c_j t^{2p_j} \md x_j^2 = \L_{\sum_{j = 1}^n c_j x_j \d_{x_j}}g$.
The final solution is when $a_0 = -1$ and $a_j = p_j$ for $j = 1, \hdots, n$.
It is given by $-\md t^2 + \sum_{j = 1}^n p_j t^{2p_j} \md x_j^2 = \L_{t\d_t} g$.
This finishes the computation of the $n(n+2)$ solutions.
That these are generalized QNM, and the value of the frequencies, is readily verified with the formula $\L_T \L_V g = \L_{[T, V]} g + \L_V \L_T g$, for any vector field $V$.
\end{proof}

Using this, we may prove Theorem~\ref{thm: main linearized Einstein, general dim}.

\begin{proof}[Proof of Theorem~\ref{thm: main linearized Einstein, general dim}]
Let $(M, g)$ be a non-flat Kasner spacetime of any spacetime dimension (the Kasner relations force the dimension to be $\geq 4$), and let $l \in \R$.
By assumption, there is a $u \in C^\infty(J_\de^+(\gamma))$ satisfying $\De \Ric_g(u) = 0$.
Since $J_\de^+(\gamma) \subset M$ is closed, there is a smooth extension $u_{ext} \in C^\infty(M)$ such that $u_{ext}|_{J_\de^+(\gamma)} = u$.
Consequently, $\De \Ric_g(u_{ext})|_{J_\de^+(\gamma)} = \De \Ric_g(u) = 0$.
Constructing $X$ as in \eqref{eq: constr X} with $u$ replaced by $u_{ext}$, we see that $\tilde u_{ext} := u_{ext} + \L_Xg$ satisfies $\Box_L \tilde u_{ext} = 0$ and $\div_g \left( \tilde u_{ext} - \frac12\tr_g(u_{ext}) g \right) = 0$.
Defining $\tilde u := \tilde u|_{ext}|_{J_\de^+(\gamma)}$, we note that $\Box_L \tilde u = 0$ and $\tilde u = u + \L_Xg|_{J^+_\de(\gamma)}$.
If we now let $\chi \in C^\infty(\R)$ be such that $\chi(t) = 1$ for $t \leq \frac{\de}3$ and $\chi(t) = 0$ for $t \geq \frac{2\de}3$, then $\Box_L \left( \chi \tilde u \right) = [\Box_L, \chi] \tilde u \in C_c^\infty(J_\de^+(\gamma))$ is compactly supported.
We can therefore find an $f \in C_c^\infty(M)$ such that $f|_{J_\de^+(\gamma)} = [\Box_L, \chi] \tilde u$.
Let us now apply Theorem~\ref{thm: main as exp} to conclude that the unique backward solution $v \in C^\infty(M)$ to $\Box_L v = f$ satisfies
\begin{equation} \label{eq: asymptotic expansion in the proof}
	\left| \L_{t\d_t}^k \L_{t^{1-p}\d_x}^\a \left( v - \sum_{j = 1}^N c_j \u_j \right) (t, x) \right| \leq C t^l
\end{equation}
for fixed $\u_1, \u_2, \hdots \in C^\infty(M)$ associated with $\Box_L$, for some $C > 0$ and unique constants $c_1, \hdots, c_N$ and all $(t, x) \in J^+(\gamma)$.
Since $f|_{J^+(\gamma)} = [\Box_L, \chi] \tilde u$, finite speed of propagation implies that $\chi \tilde u = v|_{J^+_\de(\gamma)}$.
It follows that
\[
	\chi u 
		= \chi \tilde u + \L_Xg |_{J_\de^+(\gamma)} 
		= v|_{J_\de^+(\gamma)} + \L_Xg |_{J_\de^+(\gamma)}.
\]
Since $\chi u$ and $u$ coincide for $t < \frac\de3$, we thus obtain the estimate \eqref{eq: asymptotic expansion higher}.
By construction of the $\u_j$, note that $\Box_L \u_j = 0$ for all $j = 1, \hdots, n$.
Note that $v|_{J^+_{\de/3}(\gamma)} = \tilde u$, so $\div_g\left(v - \frac12 \tr_g(v) g \right)|_{J^+_{\de/3}(\gamma)} = 0$.
Since Theorem~\ref{thm: main as exp} implies that the asymptotic expansion \eqref{eq: asymptotic expansion in the proof} is unique, we conclude that $\div_g\left(\u_j - \frac12 \tr_g(\u_j) g \right)|_{J^+_{\de/3}(\gamma)} = 0$ for $j = 1, \hdots, n$.
By Corollary~\ref{cor: structure of QNMs}, all generalized QNMs of geometric wave equations, like $\u_j$ for $\Box_L$, are real analytically depending in $(t, x)$.
Therefore $\div_g\left(\u_j - \frac12 \tr_g(\u_j) g \right)$ is a real analytic that vanishes on an open subset, and thus vanishes globally.
Finally, combining Theorem~\ref{thm: AVTD system} with Corollary~\ref{cor: QNM frequencies} implies that the quasinormal frequencies $\s_j$ of $\u_j$ are purely imaginary.
This means that all $\lambda_j := - i \s$ in the assertion are real.
\end{proof}

\subsection{The explicit unstable quasinormal modes}

In Theorem~\ref{thm: AVTD system}, we computed all the $x$-independent solutions to the linearized Einstein equation satisfying the gauge condition \eqref{eq: gauge condition}.
The next step is to construct our families of which the unstable solutions in Theorem~\ref{thm: self-similar solutions} will be a subset. 

\subsubsection{Construction of the explicit solutions}

The following give in particular the large class of unstable solutions in spacetime dimension $4$, introduced in Section~\ref{subsubsec: USS}.

\begin{thm} \label{thm: explicit solutions}
Assume that $(M, g)$ is a non-flat Kasner spacetime.
Let $j, m, r \in \{1, \hdots, n\}$ be distinct, assume that $p_m > p_r$, and let $k \in \N_0$ be such that $k < 2 \frac{p_m-p_r}{1-p_j}$.
Define
\[
	v_{jk}(t, x)
		: = \sum_{l = 0}^k a_{lk} t^{l(1-p_j)} \left( x_j \right)^{k-l} t^{2p_r} \md x_m \otimes_s \md x_r,
\]
where $a_{0k} = \frac1{k!}$ and $a_{lk} = \frac1{(k-l)!} \prod_{q = 1}^{\frac l2} \frac1{\left( 2q (1 - p_j) + p_r - p_m \right)^2 - \left( p_r - p_m \right)^2}$, if $l$ is even and positive, and $a_{lk} = 0$ if $l$ is odd.
Then 
\[
	\Box_L v_{jk} 
		= 0,
	\qquad
	\left( \L_\T + i \s \right)v_{jk}
		= 0,
	\qquad
	\div_g \left(v_{jk} - \frac12 \tr_g(v_{jk}) g \right)
		= 0, 
\]
with $\s = \left( - 2 + p_m - p_r - k (1-p_j) \right)i$.
In particular, $\De \Ric_g(v_{jk}) = 0$ for $j,k$ as above.
\end{thm}

Note that for $k = 0$, then $v_{j0} = t^{2p_r} \md x^m \otimes_s \md x^r = \L_{\frac12 x_m \d_{x_r}}g$, so it is a gauge solutions.
However, for $k \geq 1$ we prove in Proposition~\ref{prop: linear independence} that $v_{jk}$ are not gauge solution.

\begin{remark} \label{rmk: relation explicit solutions}
Since $(k-l) a_{lk} = a_{l(k-1)}$, we note that $\L_{\d_{x_j}} v_{jk} = v_{j(k-1)}$, so $v_{jk}$ and $v_{j(k-1)}$ are related via the annihilation operator in Definition~\ref{def: ann + creat}.
\end{remark}

\begin{remark}
Assume that $n = 3$.
For $k \geq 1$, we may write the QNM frequencies in Theorem~\ref{thm: explicit solutions} as
\[
	\s 
		= \left( - 2 + p_m - p_r - k (1-p_j) \right)i
		= \left( - 2 - 2p_r - (k-1) (1-p_j) \right)i
\]
using that $p_j + p_m + p_r = 1$.
Hence, if $\Im(\s) \geq -2$ then $p_r \leq 0$, which for non-flat Kasner spacetimes of dimension $4$ is equivalent to $p_r < 0$ and $p_j, p_m > 0$.
Consequently, if we assume that $p_3 < 0$, the only possibilities for the indices $j$ and $m$ are $j = 1$ and $m = 2$ or $j = 2$ and $m = 1$, leading to the unstable solutions $v_k$ and $w_k$ introduced in Section~\ref{subsubsec: USS}. 
\end{remark}

\begin{proof}[Proof of Theorem~\ref{thm: explicit solutions}]
Firstly, for any smooth function $f$ depending only on $t$ and $x_j$,
\[
	\div_g \left( f(t, x_j) e^m \otimes_s e^r \right)
		= 0, \qquad
	\tr_g \left( f(t, x_j) e^m \otimes_s e^r \right)
		= 0.
\]
Secondly, Lemma~\ref{le: L-d'Alembert} implies that $\Box_L \left( f(t, x_j) e^m \otimes_s e^r \right) = 0$ if and only if
\[
	\left( \left( t\d_t \right)^2 - t^{2(1-p_j)}\d_{x_j}^2 - \left( p_r - p_m \right)^2 \right) f
		= 0.
\] 
In the assertion, $f(t, x_j) = \sum_{l = 0}^k a_{lk} t^{l(1-p_j)} \left( x_j \right)^{k-l} t^{p_r-p_m}$, which is a solution if and only if
\[
	a_{lk} \left( \left( l (1-p_j) + p_r - p_m \right)^2 - \left( p_r - p_m \right)^2 \right) - a_{(l-2)k} (k-l+2) (k-l+1)
		= 0
\]
for every $l = 2, \hdots, k$.
Setting $a_{0k} = \frac1{k!}$, the recursion relation is satisfied by the $a_{lk}$ in the assertion.
The assumption on $k$ ensures that all $a_{lk}$ are well-defined. 
It is straightforward to compute the QNM frequency.
\end{proof}

\subsubsection{An extra solution when two Kasner exponents coincide.}

\begin{thm} \label{thm: LRS case}
Assume that $(M, g)$ is a non-flat Kasner spacetime and let $j,m,r \in \{1, \hdots, n\}$ be distinct such that $p_j = p_m = 1 + p_r$.
(If the spacetime dimension is $4$, then this is precisely when two Kasner exponents coincide, i.e.~up to permutation when $p_j = p_m = \frac23$ and $p_r = -\frac13$.) 
Then
\[
	z(t, x)
		:= x_m \left( \frac{1 + 2p_r}2 x_j^2 t^{2p_r} - 1 \right) \md x_m \otimes_s \md x_r + x_j \md x_j \otimes_s \md x_r
\]
satisfies 
\[
	\De \Ric(z)
		= 0,
	\qquad
	\left( \L_\T + i \s \right) z
		= 0,
\]
with $\s = -(1-3p_r)i$.
In particular, if the spacetime dimension is $4$, and $p_1 = p_2 = \frac23$ and $p_3 = -\frac13$, then
\[
	z(t, x)
		= x_1 \left( \frac16x_2^2t^{- \frac23} - 1 \right) \md x_1 \otimes_s \md x_3 + x_2\md x_2 \otimes_s \md x_3
\]
satisfies $\De \Ric(z) = 0$ and $z$ is a QNM of frequency $\s = -2i$ and of order $1$.
\end{thm}

\begin{proof}
Note that $v_{j2} = \left( \frac12 x_j^2 t^{2p_r} - \frac1{1 + 2p_r} \right) \md x_m \otimes_s \md x_r$, where $v_{j2}$ was defined in Theorem~\ref{thm: explicit solutions}.
It follows that
\begin{align*}
	\L_{\d_{x_m}} z
		= & \ (1 + 2p_r) v_{j2}, \\
	\L_{\d_{x_j}} z
		= & \ (1 + 2p_r) x_m x_j t^{2p_r} \md x_m \otimes_s \md x_r + \md x_j \otimes_s \md x_r \\
		= & \ \frac{1 + 2p_r}4 \L_{x_m^2x_j \d_{x_r}}g - \left( \frac{1 + 2p_r}2 x_m^2 t^{2p_r} - 1 \right) \md x_j \otimes_s \md x_r, \\
		= & \ \frac{1 + 2p_r}4 \L_{x_m^2x_j \d_{x_r}}g - (1 + 2p_r) v_{m2},
\end{align*}
and $\L_{\d_{x_b}} z = 0$ for all $b \neq m, j$.
By Theorem~\ref{thm: explicit solutions}, it follows that $\L_{\d_{x_a}} \De \Ric_g(z) = \De \Ric_g \left( \L_{\d_{x_a}} z \right) = 0$, for all $a = 1, \hdots, n$.
It therefore suffices to check that $\De \Ric_g(z)|_{x = 0}$.
Since $\De \Ric_g$ is a second order differential operator, the first term, that contains third powers of $x$, will give no contribution at $x = 0$.
It therefore suffices to check that $\De \Ric_g(z_0)|_{x = 0}$, where
\begin{align*}
	 z_0
	 	:= & \ - x_m \md x_m \otimes_s \md x_r + x_j \md x_j \otimes_s \md x_r \\
	 	= & \ - x_m t^{-p_m-p_r} e^m \otimes_s e^r + x_j t^{-p_j-p_r} e^j \otimes_s e^r,
\end{align*}
where $e^\a$ are as in Lemma~\ref{le: L-d'Alembert}, for $\a = 0, \hdots, n$.
We first note that 
\[
	\div_g \left( z_0 - \frac12 \tr_g(z_0) g \right)|_{x = 0} 
		= - \md x_m \otimes_s \md r \left( \grad(x_m), \cdot \right) + \md x_j \otimes_s \md r \left( \grad(x_j), \cdot \right)
		= 0,
\]
since $\grad (x_m) = t^{2p_m}\d_{x_m}$ and $\grad (x_j) = t^{2p_j}\d_{x_j}$, and $p_j = p_m$.
Moreover, by Lemma~\ref{le: L-d'Alembert},
\[
	\Box_L z_0 |_{x = 0}
		= 4 \left( - p_m t^{-p_m-p_r} + p_j t^{-p_j-p_r} \right) e^0 \otimes_s e^r
		= 0,
\]
since $p_j = p_m$.
Hence $\De \Ric_g(z)|_{x = 0} = 0$, and therefore $\De \Ric_g(z) = 0$.
A straightforward computation shows that $(\L_\T + 1-3p_r) z = 0$.
\end{proof}

We can modify $z$ in the following way to satisfy the gauge condition \eqref{eq: gauge condition}:

\begin{cor} \label{cor: LRS}
Assume that $(M, g)$ is a non-flat Kasner spacetime and let $j,m,r \in \{1, \hdots, n\}$ be distinct such that $p_j = p_m = 1 + p_r$.
Then $\tilde z(t, x) := z(t, x) + \L_Yg$, where $z$ is as in Theorem~\ref{thm: LRS case} and $Y = \left( f_1(t) x_j^2 + f_2(t) \right) \d_{x_r}$, with
\[
	f_1(t)
		= - \frac{1 + 2p_r}{4p_r} \left( \log(t) + \frac1{2p_r} \right)t^{-2p_r}, 
	\qquad
	f_2(t) 
		= - \frac{1 + 2p_r}{16p_r^3} \left( \log(t) + \frac5{4p_r} \right) t^{-4p_r},
\]
satisfies
\[
	\Box_L \tilde z
		= 0,
	\qquad
	\left( \L_\T + i \s \right)^2 \tilde z
		= 0,
	\qquad
	\div_g \left(\tilde z - \frac12 \tr_g(\tilde z) g \right)
		= 0, 
\]
with $\s = -(1-3p_r)i$.
\end{cor}

\begin{proof}
Note first that $\De \Ric_g(\tilde z) = 0$ by Theorem~\ref{thm: LRS case}.
Since $\left( \L_\T + 1 - 3p_r \right) z = 0$ by Theorem~\ref{thm: LRS case}, it follows that $\left( \L_\T + 1 - 3p_r \right)^2 z = 0$.
Since $\L_\T g = - 2 g$, note that
\[
	\left( \L_\T + 1 - 3p_r \right)^2 \L_Y g 
		= \left( \L_\T + 1 - 3p_r \right) \L_{\left( \L_\T - 1 - 3p_r \right) Y} g 
		= \L_{\left( \L_\T - 1 - 3p_r \right)^2 Y} g
		= 0,
\]
by the form of $f_1$ and $f_2$.

By a standard identity, which uses $\Ric_g = 0$, and a straightforward computation, 
\[
	\div_g \left( \tilde z - \frac12 \tr_g \left( \tilde z \right) g \right)
	= - \Box Y + \div_g\left( z - \frac12 \tr_g(z) g \right)^\sharp 
	= - \Box Y + \frac{1 + 2p_r}2 x_j^2 t^{-2p_r} \d_{x_r},
\]
so the gauge condition is equivalent to $\Box Y = \frac{1 + 2p_r}2 x_j^2 t^{-2p_r} \d_{x_r}$.
One computes that this in turn is equivalent to the equations
\[
	\left( t\d_t \right)^2 f_1 (t) + 2 p_r t \d_t f_1(t) 
		= \frac{1 + 2p_r}2 t^{-2p_r},
	\qquad 
	\left( t\d_t \right)^2 f_2(t) + 2 p_r t \d_t f_2(t) 
		= 2 t^{-2p_r} f_1 (t), 
\]
with are satisfied by the functions $f_1$ and $f_2$ in the assertion. 
\end{proof}

We will not need the precise form of $\tilde z$ in Corollary~\ref{cor: LRS}, but it will be important for the proof of Theorem~\ref{thm: self-similar solutions} below to note that $\tilde z$ coincides to \emph{leading} polynomial order in $x$ with $z$.

\subsubsection{Linear independence of the explicit solutions}

The next step is to prove that the solutions showing up in Theorem~\ref{thm: self-similar solutions} and Theorem~\ref{thm: main linearized Einstein} linearly independent, even modulo the addition of a gauge solution.
This will ensure that the coefficients in the expansions in Theorem~\ref{thm: self-similar solutions} and Theorem~\ref{thm: main linearized Einstein} are unique.

\begin{prop} \label{prop: linear independence}
Assume that $(M, g)$ is a non-flat Kasner spacetime of dimension $4$ with $p_3 < 0$.
Let $K_1, K_2$ be the largest integers such that $K_1 \leq \frac{p_2 - p_3}{1-p_1}$ and $K_2 \leq \frac{p_1 - p_3}{1-p_2}$, respectively.
Assume moreover that there is a vector field $X \in C^\infty(M)$ constants $a, b_1, \hdots, b_{K_1}, c_1, \hdots, c_{K_2}, d \in \R$, such that
\begin{equation}\label{eq: linear gauge combination}
	\left| \L_{t\d_t}^k \L_{t^{1-p}\d_x}^\a \left( \L_X g + a g' + \sum_{k = 1}^{K_1} b_k v_k +  \sum_{k = 2}^{K_2} c_k w_k + d z \right)(t, x) \right|
		\leq C_{k, \a} t^\epsilon,
\end{equation}
for some $C_{k, \a}, \epsilon > 0$ and all $k \in \N_0$ and $\a \in \N_0^n$ and $(t, x) \in J^+(\gamma)$ with $t \leq 1$, and where $d = 0$ if the $p_1, p_2, p_3$ are distinct.
Then $a, b_1, \hdots, b_{K_1}, c_1, \hdots, c_{K_2}, d = 0$.
\end{prop}

\begin{notation}
In order to simplify notation in the coming proofs, let us write $u \in \O_*(t^\b)$ if 
\[
	\left| \L_{t\d_t}^k \L_{t^{1-p}\d_x}^\a u(t, x) \right|
		\leq C_{k, \a} t^\b,
\]
for some $C_{k, \a}, \epsilon > 0$ and all $k \in \N_0$ and $\a \in \N_0^n$ and $(t, x) \in J^+(\gamma)$ with $t \leq 1$, for any tensor or function $u \in C^\infty(M)$.
\end{notation}

\begin{lemma} \label{le: linearized Kasner non gauge}
Assume that $(M, g)$ is a non-flat Kasner spacetime of dimension $4$ with $p_3 < 0$.
There is no vector field $X \in C^\infty(M)$ such that
\begin{equation} \label{eq: g' gauge}
	\L_X g - g' \in \O_* \left( t^{\b} \right) ,
\end{equation}
if $\b > 0$, where $g'$ is as in \eqref{eq: g' def}.
\end{lemma}
\begin{proof}
We assume, to reach a contradiction, that there is a vector field $X = X_0 \d_t + \sum_{j = 1}^3 X_j t^{-p_j} \d_{x_j}$ satisfying \eqref{eq: g' gauge}.
Evaluating \eqref{eq: g' gauge} on $\d_t \otimes \d_t$, $\d_t \otimes t^{-p_j} \d_{x_j}$ and $t^{-p_j} \d_{x_j} \otimes t^{-p_j} \d_{x_j}$, and multiplying by $t$, implies that
\begin{align}
	0
		= & \ t \d_t X_0 + \O_* \left( t^{\b + 1} \right)
		= t^{1 + p_j} \d_t \left( t^{-p_j} X_j \right) - t^{1 - p_j} \d_{x_j} X_0 + \O_* \left(t^{\b + 1} \right), \label{eq: g' first obs} \\
	2 t \log(t) p_j'
		= & \ t^{1-p_j} \d_{x_j} X_j + p_j X_0 + \O_* \left(t^{\b + 1} \right), \label{eq: g' second obs}
\end{align}
for $j = 1, 2, 3$.
It follows that
\begin{equation} \label{eq: g' comp}
\begin{split}
	t^{2-2p_j} \d_{x_j}^2 X_0
		= & \ t^2 \d_t \left( t^{-p_j} \d_{x_j} X_j \right) + \O_* \left(t^{\b + 1} \right)
		=  \left( t \d_t - 1 \right) \left( t^{1-p_j} \d_{x_j} X_j \right) + \O_* \left(t^{\b + 1} \right) \\
		= & \ - \left( t \d_t - 1 \right) \left( p_j X_0 \right) - 2 p_j' \left( t \d_t - 1 \right) \left( t \log(t) \right) + \O_* \left(t^{\b + 1} \right) \\
		= & \ p_j X_0 - 2 p_j' t + \O_* \left(t^{\b + 1} \right),
\end{split}
\end{equation}
for $j = 1, 2, 3$.
Differentiating both sides with respect to $t\d_t$ gives 
\[
	(2-2p_j) \left( t^{2-2p_j} \d_{x_j}^2 X_0 \right)
		= t\d_t \left( t^{2-2p_j} \d_{x_j}^2 X_0 \right) + \O_* \left(t^{\b + 1} \right)
		= - 2 p_j' t + \O_* \left(t^{\b + 1} \right).
\]
Inserting this in \eqref{eq: g' comp}, we conclude that $p_j X_0 = \left( 2 - \frac1{1-p_j} \right) p_j' t + \O_* \left(t^{\b + 1} \right)$, and hence $t^{1 - p_j} \d_{x_j} X_0 \in \O_* \left(t^{\b + 1} \right)$.
Inserting this into \eqref{eq: g' first obs}, it follows that $t^{1 + p_j} \d_t \left( t^{-p_j} X_j \right) \in \O_* \left(t^{\b + 1} \right)$.
Thus dividing \eqref{eq: g' second obs} by $t$ and applying $t\d_t$, we finally conclude that
\[
	2 p_j'
		= t\d_t \left( t^{-p_j} \d_{x_j} X_j \right) + t\d_t \left( \frac{p_j}t X_0 \right) + \O_* \left(t^\b \right)
		= \O(t^\b).
\]
Since $\b > 0$, this is a contradiction unless $p_j' = 0$ for $j = 1, 2, 3$, which is impossible.
\end{proof}

\begin{lemma} \label{le: Bianchi II non-gauge}
Assume that $(M, g)$ is a non-flat Kasner spacetime of dimension $4$ with $p_3 < 0$.
There is no vector field $X \in C^\infty(M)$ such that
\begin{equation} \label{eq: BII gauge}
	\L_X g - v_1 \in \O_* \left( t^{\b} \right) ,
\end{equation}
if $\b > 2p_3$.
\end{lemma}
\begin{proof}
We assume, to reach a contradiction, that there is a vector field $X = X_0 \d_t + \sum_{j = 1}^3 X_j t^{-p_j} \d_{x_j}$ satisfying \eqref{eq: BII gauge}.
Evaluating \eqref{eq: BII gauge} on $\d_t \otimes \d_t$, $\d_t \otimes t^{-p_j}\d_{x_j}$ and $t^{-p_j}\d_{x_j} \otimes t^{-p_j}\d_{x_j}$, and multiplying by $t$, implies that
\begin{equation} \label{eq: first Lie ident}
\begin{split}
	0
		= & \ t\d_t X_0 + \O_* \left(t^{\b+1} \right)
		= t^{1 + p_j} \d_t \left( t^{-p_j} X_j \right) - t^{1 - p_j} \d_{x_j} X_0 + \O_* \left(t^{\b + 1} \right) \\
		= & \ t^{1-p_j} \d_{x_j} X_j + p_j X_0 + \O_* \left(t^{\b + 1} \right),
\end{split}
\end{equation}
for $j = 1, \hdots 3$.
From this, it follows that
\begin{equation} \label{eq: X0 identity}
\begin{split}
	t^{2-2p_j}\d_{x_j}^2 X_0
		= & \ t^2 \d_t \left( t^{-p_j} \d_{x_j} X_j \right) + \O_* \left(t^{\b + 1} \right)
		= t \d_t \left( t^{1-p_j} \d_{x_j} X_j \right) - t^{1-p_j} \d_{x_j} X_j + \O_* \left(t^{\b + 1} \right) \\
		= & \ - t \d_t \left( p_j X_0 + \O_* \left(t^{\b + 1} \right) \right) + p_j X_0 + \O_* \left(t^{\b + 1} \right) 
		= p_j X_0 + \O_* \left(t^{\b+1} \right).
\end{split}
\end{equation}
Applying $t\d_t$ to both sides, we see that $(2-2p_j) t^{2-2p_j} \d_{x_j}^2 X_0 = p_j t \d_t X_0 + \O_* \left(t^{\b+1} \right)$, hence \eqref{eq: first Lie ident} in combination with \eqref{eq: X0 identity} implies that $X_0 \in \O_* \left(t^{\b+1} \right)$.
Using this, \eqref{eq: first Lie ident} implies that $t^{1-p_j}\d_{x_j} X_j \in \O_* \left(t^{\b+1} \right)$ and $t^{1 + p_j} \d_t \left( t^{-p_j} X_j \right) \in \O_* \left(t^{\b+1} \right)$.

We now evaluate \eqref{eq: BII gauge} on $t^{-p_j}\d_{x_j} \otimes t^{-p_k} \d_{x_k}$, for $j \neq k$ and obtain
\begin{align*}
	t^{-p_1} \d_{x_1} X_2 + t^{-p_2} \d_{x_2} X_1 + \O_* \left(t^\b \right)
		= t^{-p_1} \d_{x_1} X_3 + t^{-p_3} \d_{x_3} X_1 + \O_* \left(t^\b \right)
		&= 0, \\
	t^{-p_2} \d_{x_2} X_3 + t^{-p_3} \d_{x_3} X_2 + \O_* \left(t^\b \right)
		&= x_1 t^{p_3 -p_2}.
\end{align*}
Applying $t^{1-p_1} \d_{x_1}$ to the last equation and using what we proved above, we get
\begin{align*}
	t^{2p_3}
		= t^{1-p_1} \d_{x_1} \left( x_1 t^{p_3 - p_2} \right)
		= & \ t^{1-p_2-p_1} \d_{x_2} \d_{x_1} X_3 + t^{1-p_3-p_1} \d_{x_3} \d_{x_1} X_2 + \O_* \left(t^\b \right) \\
		= & \ 2 t^{1-p_2-p_3} \d_{x_2} \d_{x_3} X_1 + \O_* \left(t^\b \right)
		= 2 t^{p_1} \d_{x_2} \d_{x_3} X_1 + \O_* \left(t^\b \right).
\end{align*}
Finally, apply $t^{1 + p_1}\d_t t^{-2p_1}$ to both sides to obtain
\begin{align*}
	2(p_3 - p_1) t^{-p_1 + 2p_3}
		= & \ t^{1 + p_1}\d_t \left( t^{-2p_1} t^{2p_3} \right)
		= 2 \d_{x_2} \d_{x_3} t^{1 + p_1}\d_t \left( t^{-p_1} X_1 \right) + \O_* \left(t^{\b-p_1} \right) \\
		= & \ \d_{x_2} \d_{x_3} \O_* \left( t^{\b + 1} \right) + \O_* \left(t^{\b-p_1} \right)
		= \O_* \left(t^{\b-p_1} \right).
\end{align*}
This implies that $2|p_3 - p_1| t^{2p_3 - p_1}\leq C t^{\b - p_1}$, but this is a contradiction since $\b > 2p_3$, unless $p_3 = p_1$ which cannot happen. 
\end{proof}

\begin{proof}[Proof of Proposition~\ref{prop: linear independence}]
Let us begin by showing that $d = 0$ also when $p_1 = p_2 = 1 + p_3 = \frac23$.
Note that $\L_{\d_{x_1}}\L_{\d_{x_2}} \left( \L_X g + a g' + \sum_{k = 1}^{K_1} b_k v_k +  \sum_{k = 2}^{K_2} c_k w_k + d z \right)
	=  \L_{\L_{\d_{x_1}}\L_{\d_{x_2}}X} g + d \L_{\d_{x_1}}\L_{\d_{x_2}} z$.
By \eqref{eq: linear gauge combination}, and since $\L_{\d_{x_1}}\L_{\d_{x_2}} z = \frac13 x_2t^{2p_3} \md x_1 \otimes_s \md x_3 = \frac13 w_1 = - \frac13 v_1 + \frac16 \L_{x_1 x_2 \d_{x_3}}g$, we conclude that there is a smooth vector field $Y$ such that $\L_Y g + d v_1 \in \O_*\left(t^{\epsilon + p_1 + p_2 - 2}\right) = \O_*\left( t^{\epsilon + 2 p_3} \right)$.
This is a contradiction to Lemma~\ref{le: Bianchi II non-gauge} unless $d = 0$.

Next, since $\L_{\d_{x_1}}^{K_1-1} \left( \L_X g + a g' + \sum_{k = 1}^{K_1} b_k v_k +  \sum_{k = 2}^{K_2} c_k w_k \right) =  \L_{\L_{\d_{x_1}}^{K_1-1}X} g + (K_1-1)! \cdot b_{K_1} v_1$ for any $K_1 \geq 2$, there is a smooth vector field $Y$ such that $\L_Yg + b_{K_1} v_1 \in \O_*\left( t^{\epsilon - (K_1-1)(1-p_1)} \right)$.
By the assumption on $K_1$, $- (K_1-1) (1-p_1) \geq - (p_2 - p_3) + 1 - p_1 = 2p_3$, so this is contradicting Lemma~\ref{le: Bianchi II non-gauge} for any $K_1 \geq 2$ unless $b_{K_1} = 0$.
We can now iterate this argument in order to conclude that $b_k = 0$ for all $k \geq 2$.
An analogous argument shows that $c_k = 0$ for all $k \geq 2$.
For the remaining expression, we use that $\left(\L_\T + 2\right)^2 g' = 0$, and compute that
\[
	\left( \L_\T + 2 \right)^2 \left( \L_X g + a g' + b_1 v_1 \right)
		= \L_{\left( \L_\T + 2 \right)^2 X} g + b_1 p_3^2 v_1.
\]
Since $p_3 \neq 0$ and $\T = - t\d_t - \sum_{j = 1}^3 (1-p_j)x_j \d_{x_j}$, we may conclude by \eqref{eq: linear gauge combination} that also $b_1 = 0$, for otherwise we would have a contradiction to Lemma~\ref{le: Bianchi II non-gauge}.
Finally, Lemma~\ref{le: linearized Kasner non gauge} implies that $a = 0$, which finishes the proof.
\end{proof}

\subsection{Characterization of unstable quasinormal modes}

We now turn to the proofs of Theorem~\ref{thm: self-similar solutions} and Theorem~\ref{thm: main linearized Einstein}.
In view of Theorem~\ref{thm: main linearized Einstein, general dim}, the idea is to compute all generalized QNM with $\Im(\s) \geq - 2$.
Viewing linear wave equations as systems of wave equations, Corollary~\ref{cor: structure of QNMs} implies that a $\u$ satisfying $\Box_L \u = 0$, the gauge condition $\div_g \left( \u - \frac12 \tr_g(\u) g \right) = 0$ and $\left( \L_\T + i \s \right)^k \u = 0$, for some $k \in \N$, takes the form
\begin{equation} \label{eq: QNM lin ein general form}
	\u(t, x)
		= t^{i\s} \sum_{\abs{\o}_p \leq K} t^{-\abs{\o}_p} x^\o \u_{\o},
\end{equation}
where $\left( \L_{t\d_t} \right)^k \u_{\o} = 0$ for some $k \in \N$ and $\L_{\d_{x_j}}\u_{\o} = 0$ for $j = 1, \hdots, n$.
The main point in the our proofs of Theorem~\ref{thm: self-similar solutions} and Theorem~\ref{thm: main linearized Einstein} is the second assertion of Corollary~\ref{cor: structure of QNMs}, which here says that
\[
	\Box_L \left( t^{i\s - \abs{\a}_p} \u_\a \right)
		= 0, 
	\qquad
	\L_{\d_{x_j}} \left( t^{i\s - \abs{\a}_p} \u_\a \right)
		= 0,
	\qquad
	\L_{\d_{x_j}} \left( t^{i\s - \abs{\a}_p} \u_\a \right)
		= 0,
\]
for any $\a \in \N_0^n$ such that $\u_\o = 0$ for all $\o \in \N_0^n$ with $\o > \a$, and for $j = 1, \hdots, n$.
This means that all such $t^{i\s - \abs{\a}_p} \u_\a$ are characterized in Theorem~\ref{thm: AVTD system}.
This is our starting point.

\subsubsection{Construction of gauge solutions}

The idea in the proof of Theorem~\ref{thm: self-similar solutions} is to successively reduce the polynomial order in $x$ of $\u$ in \eqref{eq: QNM lin ein general form}, starting with the highest order terms that we have characterized in Theorem~\ref{thm: AVTD system}.
The following lemma deals with most of the possible cases.

\begin{lemma} \label{le: gauge solution reduction}
Let $n \in \N$ and assume that $(M, g)$ is a non-flat Kasner spacetime of dimension $n+1$.
Let $W$ be any of the vector fields
\[
	\frac1t \d_t, \quad \left( \log(t) - 1 \right)t\d_t + \sum_{j = 1}^n p_j x_j \d_{x_j}, \quad \sum_{j = 1}^n d_j x_j \d_{x_j}, \quad t \d_t,
\]
or $t^{-2p_j}\d_{x_j}$ if $p_j \neq 0$ and $\log(t) \d_{x_j}$ if $p_j = 0$, where here $\sum_{j = 1}^n d_j = 0$.
Let $\a \in \N_0^n$ and define $\s \in \C$ by $\L_\T \left( x^\a \L_W g\right) = - i \s x^\a \L_Wg$.
Then either $\Im(\s) < -2$ or there is a vector field $X \in C^\infty(M)$ satisfying $\Box X = 0$, such that
\[
	x^\a \L_Wg - \L_X g
		= t^{i \s} \sum_{\abs{\o}_p < \abs{\a}_p} t^{-\abs{\o}_p} x^\o \mathfrak v_\o,
\]
where $\L_{\d_{x_j}} \mathfrak v_\o = 0$ for all $j$ and all $\o$.
\end{lemma}
\begin{proof}
We first note that $\L_\T \left( x^\a \L_W g\right) = x^\a \left( - \abs{\a}_p \L_Wg + \L_\T \L_Wg \right)$.
It therefore follows that $\s = - \abs{\a}_p i + \tilde \s$, where $\tilde \s \in \C$ is the QNM frequency of $\L_Wg$ given by Theorem~\ref{thm: AVTD system}.
Since the QNM frequency of
\[
	\L_{\left( \log(t) - 1 \right)t\d_t + \sum_{j = 1}^n p_j x_j \d_{x_j}}g, \quad \L_{\sum_{j = 1}^n d_j x_j \d_{x_j}}g, \quad \L_{t \d_t} g
\]
is $\tilde \s = -2i$, it follows that either $\Im(\s) < -2$ or $\a = 0$.
In case $\a = 0$, then $x^\a \L_W g = \L_W g$. Since Theorem~\ref{thm: AVTD system} implies that $\Box W = 0$ for all these three vector fields, we can simply choose $X := W$ in these cases.
We are thus left with $W$ being either of the vector fields
\begin{equation} \label{eq: remaining VFs}
	\frac1t \d_t, \quad t^{-2p_j}\d_{x_j} \ (\text{if } p_j \neq 0), \quad \log(t) \d_{x_j} \ (\text{if } p_j = 0).
\end{equation}
Since Theorem~\ref{thm: AVTD system} implies that $\Box W = 0$ and $W$ are generalized QNMs themselves, Theorem~\ref{thm: creation and annihilation} implies that $X := \Cr_1^{\a_1} \hdots \Cr_1^{\a_n} (W)$ solves $\Box X = 0$ and is a generalized QNM.
By the properties of the annihilation and creation operators in Theorem~\ref{thm: creation and annihilation}, and since $\d_{x_1}, \hdots, \d_{x_n}$ are Killing vector fields, we note that
\[
	\left( \L_{\d_{x_1}} \right)^{\a_1} \hdots \left( \L_{\d_{x_n}} \right)^{\a_n} \L_X g
		= \L_{\L_{\d_{x_1}}^{\a_1} \hdots \L_{\d_{x_n}}^{\a_n}X} g
		= \L_W g,
\]
and $\left( \L_{\d_{x_1}} \right)^{\o_1} \hdots \left( \L_{\d_{x_n}} \right)^{\o_n} \L_X g = 0$, for all $\o \in \N_0^n$ such that $\o \nleq \a$.
This proves the assertion for the remaining vector fields in \eqref{eq: remaining VFs}.
\end{proof}

\subsubsection{Constructing non-gauge solutions}

Two classes of $x$-invariant solutions in Theorem~\ref{thm: AVTD system} have not been dealt with yet, namely the linearized Kasner metric $g'$ and $\L_{x_j \d_{x_k}}g$, for $j \neq k$.
The $g'$ is a generalized QNM of frequency $\s = - 2i$, so any multiplication by a polynomial would make it stable.
We therefore do not need to treat this case.
The next lemma deals with the remaining case $\L_{x_j \d_{x_k}}g$, for $j \neq k$.
Instead of getting gauge solutions as in Lemma~\ref{le: gauge solution reduction}, we here get the explicit solutions from Theorem~\ref{thm: explicit solutions} and Corollary~\ref{cor: LRS}.
This is where we need to restrict the spacetime dimension $4$.

\begin{lemma} \label{le: mixed terms reduction}
Assume that $(M, g)$ is a non-flat Kasner spacetime of dimension $4$ with $p_3 < 0$.
Let $j, k, l \in \{1, 2, 3\}$ be distinct, and let $\a \in \N_0^3$.
Define $\s \in \C$ by $\L_\T \left( x^\a \L_{x_j \d_{x_k}} g\right) = - i \s x^\a \L_{x_j \d_{x_k}} g$.
Then either $\Im(\s) < -2$ or there are constants $b_{\a_l}, b_{\a_{l+1}}, c_{\a_l}, c_{\a_{l+1}}, d \in \R$ a vector field $X \in C^\infty(M)$ satisfying $\Box X = 0$, such that
\[
	x^\a \L_{x_j \d_{x_k}} g - \L_X g - b_{\a_l} v_{\a_l} - b_{\a_j+1}v_{\a_j+1} - c_{\a_l} w_{\a_l} - c_{\a_{j+1}}w_{\a_j+1} - d \tilde z
		= t^{i\s} \sum_{\abs{\o}_p < \abs{\a}_p} t^{-\abs{\o}_p} x^\o \mathfrak v_\o,
\]
where $\L_{\d_{x_j}} \mathfrak v_\o = 0$ for all $j$ and $\o$.
Moreover,
\begin{itemize}
	\item $d = 0$ unless $\a_1 = 1$, $\a_2 = 2$ and $\a_3 = 0$,
	\item $b_{\a_l} = c_{\a_l} = 0$, unless $\a_j = \a_3 = 0$ and $1 \leq \a_l \leq \frac{p_j - p_k}{1-p_l}$, 
	\item $b_{\a_j+1} = c_{\a_j+1} = 0$, unless $\a_l = \a_3 = 0$ and $2 \leq \a_j + 1 \leq \frac{p_l - p_k}{1-p_j}$.
\end{itemize}
\end{lemma}
\begin{proof}
We first note that
\begin{align*}
	\L_\T \left( x^\a \L_{x_j \d_{x_k}} g\right)
		= & \ T(x^\a) \L_{x_j \d_{x_k}} g + \L_{[\T, x_j \d_{x_k}]} g + x^\a \L_{x_j \d_{x_k}} \L_\T g \\
		= & \ \left( - \abs{\a}_p + p_j - p_k - 2 \right) x^\a \L_{x_j \d_{x_k}} g.
\end{align*}
It follows that $\s = \left( - 2 + p_j - p_k - \abs{\a}_p \right) i$, hence $\Im(\s) \geq -2$ is equivalent to
\begin{equation} \label{eq: non-gauge sigma condition}
	- \abs{\a}_p + p_j - p_k
		\geq 0.
\end{equation}
If $\Im(\s) \geq - 2$, it follows that $p_j \geq p_k$.
In case $p_j = p_k$, then $\a = 0$, and we may choose $X = x_j \d_{x_k}$ in the assertion.
We may therefore assume from now on that $p_j > p_k$.

\textbf{Case 1:} $\a_k > 0$.
In this case, $- \abs{\a}_p + p_j - p_k \leq -(1 - p_k) + p_j - p_k \leq p_j - 1 < 0$, which contradicts \eqref{eq: non-gauge sigma condition}.

\textbf{Case 2:} $\a_k = 0$ and $\a_l = 0$.
Then $x^\a \L_{x_j \d_{x_k}} g = (x_j)^{\a_j} \L_{x_j \d_{x_k}} g = \L_{\frac1{\a_j + 1}(x_j)^{\a_j + 1} \d_{x_k}} g$, so we can choose $X = \a_j! \left( \Cr_j \right)^{\a_j} (x_j \d_{x_k})$ in this case, which satisfies $\Box X = 0$ and coincides with $\frac1{\a_j + 1}(x_j)^{\a_j + 1} \d_{x_k}$ to highest polynomial order.

It remains to study the cases when $\a_k = 0$ and $\a_l \geq 1$.
Now \eqref{eq: non-gauge sigma condition} implies that $0 \leq - (1-p_l) + p_j - p_k = - 2 p_k$, and it follows that $k = 3$, i.e.~$p_k = p_3 < 0$.

\textbf{Case 3:} $\a_3 = 0$, $\a_l \geq 2$ and $\a_j \geq 1$.
Using that $p_1 + p_2 + p_3 = 1$,
\[
	- \abs{\a}_p + p_j - p_k
		\leq - 2(1-p_l) - (1-p_j) + p_j - p_k
		\leq - 1 - 3 p_k,
\]
which is only non-positive when $p_k = p_3 = -\frac13$, and $j = 1$ and $l = 2$, in which case $p_1 = p_2 = \frac23$, and we have $\a_1 = 1$ and $\a_2 = 2$.
Corollary~\ref{cor: LRS} implies that $\tilde z$, defined there, satisfies $\Box_L \tilde z = 0$, $\div_g \left( \tilde z - \frac12 \tr_g(\tilde z)g \right) = 0$ and $\left( \L + 2 \right)^2 \tilde z = 0$, and $\frac12 \tilde z$ coincides with 
\[
	x^\a \L_{x_j \d_{x_k}} g
		= x_1 x_2^2 \L_{x_1 \d_{x_3}} g
		= 2 x_1 x_2^2 t^{-\frac23} \md x_1 \otimes_s \md x_3
\]
to highest polynomial order.

\textbf{Case 4:} $\a_3 = 0$, $\a_l \geq 1$ and $\a_j = 0$.
We have $x^\a \L_{x_j \d_{x_k}} g = (x_l)^{\a_l} \L_{x_j \d_{x_3}} g = (x_l)^{\a_l} t^{2p_k} \md x_j \otimes_s \md x_k$.
The condition \eqref{eq: non-gauge sigma condition} is now equivalent to $\a_l \leq \frac{p_j - p_3}{1-p_l}$ when $j,l \in \{1,2\}$.
Thus, in this case, $x^\a \L_{x_j \d_{x_k}} g$ coincides to leading order with $\a_l! \cdot v_{\a_l}$, if $l = 1$, or $\a_l! \cdot w_{\a_l}$, if $l = 2$, which are generalized QNMs satisfying our gauge condition.

\textbf{Case 5:} $\a_3 = 0$, $\a_l = 1$ and $\a_j \geq 1$.
We note that
\[
	x^\a x_l \L_{x_j \d_{x_k}} g
		= (x_j)^{\a_j} x_l \L_{x_j \d_{x_3}} g
		= \L_{\frac1{\a_j + 1} (x_j)^{\a_j + 1}x_l \d_{x_3}} g - (x_j)^{\a_j+1} \L_{x_l \d_{x_3}} g.
\]
The first term can be dealt with by choosing $X = \a_j! \left( \Cr_j \right)^{\a_j + 1} \left( x_l \d_{x_k} \right)$, similar to \textbf{Case 2}.
The second term reduces to \textbf{Case 4}, with $j$ and $l$ interchanged and with $\a_j + 1$ in place of $\a_l$.
\end{proof}

\subsubsection{Proving the characterization of self-similar solutions}

\begin{proof}[Proof of Theorem~\ref{thm: self-similar solutions}]
By assumption, $\u \in C^\infty(J^+(\gamma))$ satisfies $\De \Ric_g(\u) = 0$, and $\left( \L_\T - \lambda \right)^k \u = 0$ for some $k \in \N$ and $\lambda \geq -2$.
In other words, $\u$ is a generalized QNM with frequency $\s = \lambda i$ with $\Im(\s) = \lambda \geq -2$.
By Proposition~\ref{prop: QNM wave equations}, there is a smooth vector field $X_1$ defined in of $\mY_a$, where $\mY_a \subset M$ was defined in \eqref{eq: Y a} for some $a > 0$, satisfying $\Box_g X_1 = \div_g \left( \u + \frac12 \tr_g \left( \u \right) g \right)$ and $\left( \L_\T + i \s \right)^{k'} \L_{X_1}g = 0$ for some $k \leq k' \in \N$.
By Corollary~\ref{cor: structure of QNMs}, $X_1$ extends uniquely by analyticity to a generalized QNM on all of $M$, satisfying the same equations.
This implies that $\tilde \u := \u + \L_{X_1} g$ satisfies $\Box_L \tilde \u = 0$, $\div_g \left( \tilde \u + \frac12 \tr_g \left( \tilde \u \right) g \right) = 0$, and $\left( \L_\T + i \s \right)^{k'} \tilde u = 0$.

By Corollary~\ref{cor: structure of QNMs}, it follows that $\tilde \u(t, x) = t^{i\s} \sum_{\abs \o_p \leq K} t^{ - \abs \o_p} x^\o \tilde \u_\o$, and for all $\a \in \N_0^3$ with $\abs \a_p = K$ that $\Box_L \left( t^{i\s - \abs \a_p} \tilde \u_\a \right) = 0$, and that $\L_{\d_{x_j}} \tilde \u_\a = 0$ for all $j = 1, 2, 3$.
Moreover, we note that
\[
	0
		= \left( \L_{\d_{x_1}}\right)^{\a_1} \hdots \left( \L_{\d_{x_n}}\right)^{\a_n} \div_g \left(\tilde \u - \frac12 \tr_g(\tilde \u) g \right)
		= \div_g \left( t^{i\s - \abs \a_p} \tilde \u_\a - \frac12 \tr_g\left( t^{i\s - \abs \a_p} \tilde \u_\a \right) g \right).
\]
By combining Theorem~\ref{thm: AVTD system} with  Lemma~\ref{le: gauge solution reduction} and Lemma~\ref{le: mixed terms reduction}, we can write $\tilde \u$ as a linear combination of a gauge solution (satisfying the gauge condition \eqref{eq: gauge condition}) and $g'$, $v_1, \hdots, v_{K_1}$, $w_1, \hdots, w_{K_2}$ and $\tilde z$ and a sum of the form $t^{i\s} \sum_{\abs \o_p \leq K-1} t^{ - \abs \o_p} x^\o  \mathfrak v_\o$. 
Iterating this, we prove that $\tilde u$ is the linear combination of a gauge solution and $g'$, $v_1, \hdots, v_{K_1}$, $w_1, \hdots, w_{K_2}$ and $\tilde z$.
By construction of $\tilde z$ in Corollary~\ref{cor: LRS}, $\tilde z$ is related by a gauge solution to $z$, so it can be replaced by $z$.
Moreover, by \eqref{eq: v1 w1}, $w_1$ differs to $v_1$ only by addition of a gauge solution and can thus be removed.
Finally, a straightforward computation using the Kasner relations \eqref{eq: Kasner relations} shows that if $p_1 > p_2$, then $K_2 = 1$ for $p_3 \in \left( - \frac27, 0 \right)$, $K_2 = 2$ for $p_3 \in \left( -\frac13, \frac27 \right]$, and $K_1 = K_2 = 3$ if $p_3 = -\frac13$.
This establishes the form $\u$ in Theorem~\ref{thm: self-similar solutions}.

The uniqueness of the constants for a given $\u$ is an immediate consequence of Proposition~\ref{prop: linear independence}.
\end{proof}

\subsubsection{Proving the characterization of linear instability} \label{subsubsec: proof instability}

\begin{proof}[Proof of Theorem~\ref{thm: main linearized Einstein}]
The existence of the asymptotic expansions is a direct consequence of Theorem~\ref{thm: main linearized Einstein, general dim} and Theorem~\ref{thm: self-similar solutions}.
The uniqueness of the constants for a given $\u$ is an immediate consequence of Proposition~\ref{prop: linear independence}.
\end{proof}

\begin{appendix}
\section{Microlocal saddle point estimates} 
\label{sec: saddle point estimates}

In this section, we state and prove the saddle point estimate, which is by now standard in propagation of singularities since the work of Vasy in \cite{V2013} and subsequent variations. 
However, we are not aware of any explicit reference to precisely this formulation, other than Exercise~10.11 in \cite{H2025}. 
Since this is the main microlocal analysis ingredient (in addition to H\"ormander's propagation of singularities and elliptic regularity theory) that goes into the arguments in Section~\ref{sec: asymptotic expansion}, we formulate the details here.

Let $M$ be a smooth manifold.
Let $\rho \in C^\infty(\oT M \backslash 0)$ be a boundary defining function of $S^*M = \d \oT M$ in $\oT M$, which is positive on $T^* M \backslash 0$.
This is equivalent to defining $0 < \rho \in S^{-1}_{cl}(T^*M)$ as in \cite{H2025}*{p.~306}.
(In our application in Section~\ref{sec: asymptotic expansion}, $M = \R^n$ and we thus have a global coordinate system, so we may choose $\rho = \abs{\xi}^{-1}$.)
Let $P$ be a differential operator of order $m \in \N_0$, acting on sections in a vector bundle $F \to M$, which is equipped with a positive definite metric.
We assume throughout this section that $P$ is of real principal type, meaning that the principal symbol of $P$ is given by $p \cdot \id_F$, for a real valued smooth function $p$ on $T^*M$.
We will henceforth refer to $p$ as the principal symbol of $P$.
Let $H_p$ denote the Hamiltonian vector field with respect to $p$.
For general orders $m$, $p$ and $H_p$ do not extend smoothly to fiber $S^*M$. 
We therefore define the rescaled versions
\begin{equation} \label{eq: rescaling}
	\tilde p
		:= \rho^m p, 
	\qquad
	\tilde H_p
		:= \rho^{m-1}H_p,
\end{equation}
which smoothly extend to $\oT M$.
One can check that $\tilde H_p$ in fact is tangent to $\tilde H_p$.
We define the characteristic set of $P$ to be 
\[
	\Char(P)
		:= \tilde p^{-1}(0) \cap S^*M \subseteq S^*M.
\]
If $P_\hbar$ is a \emph{semiclassical} differential operator of order $m$ with $h$-independent real principal symbol $\p_\hbar$, then the same rescaling as in \eqref{eq: rescaling} gives $\tilde p_\hbar$ and $\tilde H_{p_\hbar}$ that extend smoothly to $\oT M$.
Analogously, $\tilde H_{p_\hbar}$ is tangent to $S^*M$.
We define the (semiclassical) characteristic set of $P_{p_\hbar}$ to be
\[
	\Char(P_\hbar)
		:= \tilde p_\hbar ^{-1}(0) \subseteq \oT M.
\]

\begin{definition} \label{def: saddle manifold}
A smooth compact submanifold $\cR \subset \oT M$ is called a \emph{saddle manifold for $P$} if the following conditions are satisfied:
\begin{itemize}
	\item[(A.1)] $\tilde H_p$ is tangent to $\cR$,
	\item[(A.2)] $\md \tilde p|_{\cR} \neq 0$, 
	\item[(A.3)] There are two functions $\rho_1, \rho_2 \in C^\infty(S^*M)$, such that $(\rho_1 + \rho_2)|_{\Char(P)}$ is a quadratic defining function of $\cR$ in $\Char(P)$ (see \cite{H2025}*{Def.~10.3} for the definition of a quadratic defining function),  satisfying
	\begin{equation} \label{eq: Hp saddle}
		\qquad 
		\tilde H_p \rho_1 
			\geq \b_1 \rho_1 + \Phi_1, 
		\qquad 
		\tilde H_p \rho_2 
			\leq - \b_2 \rho_2 + \Phi_2
	\end{equation}
	on $\Char(P) \cap \{ \rho_1 + \rho_2 < \de \}$, for some $\de > 0$, where $\b_1, \b_2, \Phi_1, \Phi_2 \in C^\infty(S^*M)$, such that $\b_1, \b_2|_{\cR} > 0$ and $\Phi_1, \Phi_2$ vanish cubically at $\cR$.
	\item[(A.4)] $\tilde H_p \rho = \b_0 \rho$ in an open neighbourhood of $\cR$ with $\b_0|_{\cR} > 0$.
\end{itemize}
Similarly, $\cR$ is called a \emph{saddle manifold for the semiclassical operator $P_\hbar$} if the same conditions hold with $p$ replaced by $p_\hbar$ (and correspondingly for the rescaled quantities) and $\Char(P)$ is replaced by $\Char(P_\hbar) \cap S^*M$.
\end{definition}

Assumption (A.2) implies that $\Char(P)$ (resp.~$\Char(P_\hbar)$) is a smooth submanifold in an open neighbourhood of $\cR$.

\begin{remark}
In the special case that $\rho_2 = 0$, then the saddle manifold is a \emph{normal source} in the sense of \cite{H2025}*{p.~306-307} and \cite{H2025}*{p.~319} for $P$ (and \emph{normal sink} for $-P$).
\end{remark}

A peculiar feature of saddle point estimates (in particular, of source/sink estimates) is that the subprincipal part of the operator gives a threshold on the a priori regularity one needs to assume.
Fix a volume density $\Vol$ on $M$.
Given a saddle manifold $\cR \subset S^*M$ with respect to $P$, as in Definition~\ref{def: saddle manifold}, let $p_1$ denote the principal symbol of $\frac1{2i}\left( P - P^* \right)$, where the formal adjoint $P^*$ is defined with respect to $\Vol$ and the positive definite metric on $F$.
We define $\tilde p_1 := \rho^{m-1} p_1$, which extends smoothly to all of $\oT M$.
In the case of a semiclassical operator $P_\hbar$, we instead let $p_{1, \hbar}$ denote the principal symbol of $\frac1{2ih}\left( P_\hbar - P_\hbar^* \right)$ and define the rescaling $\tilde p_{1, \hbar} := \rho^{m-1}p_{1, \hbar}$ analogously.

\begin{definition} \label{def: saddle manifold threshold}
Define the smooth endomorphism $\tilde \b$ of $F$ in an open neighbourhood of $\cR$ in $S^*M$ by $\tilde p_1 = \b_0 \tilde \b$, where $\b_0$ was defined (A.4) in Definition~\ref{def: saddle manifold}.
We define $\tilde \b_\hbar$ similarly.
\end{definition}

This will give the threshold regularity in our saddle point estimates:

\begin{thm}[Non-semiclassical saddle point estimate] \label{thm: saddle points}
Assume that $P$ is a linear differential operator of real principal type and of order $m \in \N_0$.
Assume that $\cR \subset \Char(P) \subset S^*M$ is a saddle manifold for $P$ as in Definition~\ref{def: saddle manifold}.
Then there is a $\de_0 > 0$, such that for all $\de \in (0, \de_0)$ the following holds:
Fix $\chi \in C_c^\infty(M)$, such that $\chi = 1$ in an open neighbourhood of 
\[
	K 
		:= \{\tilde p \leq 2\de^2 \} \cap \left\{ \rho \leq 2 \de^2 \right\} \cap \{ \rho_1 + \rho_2 \leq 2 \de \} .
\]
Let $u \in \D'(M)$.
Let $B, G, E \in \Psi^0(M)$ be such that the Schwartz kernels of $B, G$ and $E$ are contained in $K \times K$, and $\WF'(B) \cup \cR \subset \Ell(G)$.
\begin{enumerate}
	\item (Forward propagation through $\cR$.) Assume that 
	\[
		\left\{ \tilde p \leq \de^2 \right\} \cap \left\{\rho \leq \de^2 \right\} \cap \{\rho_1 \leq \de \} \cap \{\de/2 \leq \rho_2 \leq \de\} 
			\subset \Ell(E).
	\]
	Then for all $s, s_0, N \in \R$, with $s > s_0 > \frac{m-1}2 + \tilde \b$ on $\cR$, there exists a $C > 0$, such that
	\[
		\norm{B u}_{H^s}
			\leq C \left( \norm{G P u}_{H^{s-m+1}} + \norm{G u}_{H^{s_0}} + \norm{E u}_{H^s}+ \norm{\chi u}_{H^{-N}} \right).
	\]
	\item (Backward propagation through $\cR$.) Assume that 
	\[
		\left\{ \tilde p \leq \de^2 \right\} \cap \left\{\rho \leq \de^2 \right\} \cap \{\rho_1 \leq \de \} \cap \{\de/2 \leq \rho_2 \leq \de\} 
			\subset \Ell(E).
	\]
	Then for all $s, N \in \R$, with $s < \frac{m-1}2 + \tilde \b$ on $\cR$, there exists a $C > 0$, such that
	\[
		\norm{B u}_{H^s}
			\leq C \left( \norm{G P u}_{H^{s-m+1}} + \norm{E u}_{H^s} + \norm{\chi u}_{H^{-N}} \right).
	\]
\end{enumerate}
Both these estimate hold in the strong sense: If $u$ is such that the right-hand side is finite, then so is the left-hand side, and the estimate holds. 
\end{thm}

We will also need a semiclassical version of this estimate.

\begin{thm}[Semiclassical saddle point estimate] \label{thm: saddle points semiclassical}
Assume that $P_\hbar$ is a linear semiclassical differential operator with $h$-independent real principal symbol $\p_\hbar$.
Assume that $\cR \subset \Char_\hbar (\P) \cap S^*M$ and $\P_\hbar$ is a saddle manifold for $P_\hbar$ as in Definition~\ref{def: saddle manifold}. 
Then there is a $\de_0 > 0$, such that for all $\de \in (0, \de_0)$ the following holds: 
Fix $\chi \in C_c^\infty(M)$, such that $\chi = 1$ near 
\[
	K := \{\tilde p_\hbar \leq 2\de^2 \} \cap \left\{ \rho \leq 2 \de^2 \right\} \cap \{ \rho_1 + \rho_2 \leq 2 \de \}.
\] 
Let $u \in \D'(M)$.
Let $B, G, E \in \Psi_\hbar(M)$ be such that the Schwartz kernels of $B, G$ and $E$ are contained in $K \times K$, and $\WF_\hbar(B) \cup \cR \subset \Ell_\hbar(G)$.
\begin{enumerate}
	\item (Forward propagation through $\cR$.) Assume that 
	\[
		\{\tilde p_\hbar \leq \de^2 \} \cap \left\{\rho \leq \de^2 \right\} \cap \{\rho_1 \leq \de \} \cap \{\de/2 \leq \rho_2 \leq \de\} 
			\subset \Ell_\hbar(E).
	\]
	Then for all $s, s_0, N \in \R$, with $s > s_0 > \frac{m-1}2 + \tilde \b$ on $\cR$, there exist $h_1, C > 0$, such that
	\[
		\norm{B u}_{H_h^s}
			\leq C \left( h^{-1}\norm{G P u}_{H_h^{s-m+1}} + h^N \norm{G u}_{H_h^{s_0}} + \norm{E u}_{H_h^s} + h^N \norm{\chi u}_{H_h^{-N}} \right)
	\]
	for all $h \in (0, h_1)$ in the strong sense.
	\item (Backward propagation through $\cR$.) Assume that 
	\[
		\left\{ \tilde p_\hbar \leq \de^2 \right\} \cap \left\{\rho \leq \de^2 \right\} \cap \{\rho_1 \leq \de \} \cap \{\de/2 \leq \rho_2 \leq \de\} 
			\subset \Ell_\hbar (E).
	\]
	Then for all $s, N \in \R$, with $s < \frac{m-1}2 + \tilde \b$ on $\cR$, there exist $h_1, C > 0$, such that
	\[
		\norm{B u}_{H_h^s}
			\leq C \left( h^{-1}\norm{G P u}_{H_h^{s-m+1}} + \norm{E u}_{H_h^s} + \norm{\chi u}_{H_h^{-N}} \right)
	\]
	for all $h \in (0, h_1)$ in the strong sense.
\end{enumerate}
\end{thm}

\begin{proof}[Proof of Theorem~\ref{thm: saddle points} and Theorem~\ref{thm: saddle points semiclassical}]
The proof structure is the same as the proofs of \cite{H2025}*{Thm.~10.5} and \cite{H2025}*{Thm.~10.17}.
The only significant difference is that a slightly more general commutatant will be used.
The same commutant will work in both the non-semiclassical and in the semiclassical setting.
Let $\de > 0$, that we eventually will choose small enough.
For $r > 0$ define the commuatant, including regularizing factor, as
\[
	a_r
		:= \rho^{-2s + m - 1} \Psi_r(\rho)^2 \phi \left( \tilde \p \right)^2 \phi(\rho)^2 \psi(\rho_1)^2 \psi(\rho_2)^2 \cdot \id_F,
\]
where $\psi \in C_c^\infty(\R, [0, 1])$, such that $\psi(x) = 0$ for $\abs x \geq 2\de$ and $\psi(x) = 1$ for $\abs x \leq \de$ and $\sqrt{-\psi' \psi} \in C^\infty(\R, [0, 1])$, $\phi(x) := \psi \left( \frac x \de \right)$ for all $x \in \R$, and $\Psi_r(\rho)$ is the regularizer $\Psi_r(\rho) := \left( 1 + r \rho^{-1} \right)^{-\Theta}$ for any $\Theta > 0$ and all $r > 0$.
Note that $\phi$ is localizing more than $\psi$ when $\de > 0$ is small.
In the following computations, we will omit the identity map $\id_F$ for convenience.
Similarly to the proofs of \cite{H2025}*{Thm.~10.5 and Thm.~10.17}, 
\begin{align*}
	H_p a_r + 2 p_1 a_r
		= & \ \rho^{-m+1}\left( \tilde H_p a_r + 2 \tilde p_1 a_r \right) \\
		= & \ \rho^{-2s} \Psi_r(\rho)^2 \Big( \b_0 \left( - 2s + m - 1 + 2 \tilde \b + 2 f_r \right) \phi(\tilde p)^2 \phi(\rho)^2 \psi(\rho_1)^2 \psi(\rho_2)^2 \\*
		& \hspace{20mm} + 2 m \b_0 \tilde p \phi'(\tilde p) \phi(\tilde p) \phi(\rho)^2  \psi(\rho_1)^2 \psi(\rho_2)^2 \\*
		& \hspace{20mm} + 2 \b_0 \rho \phi(\tilde p)^2 \phi'(\rho) \phi(\rho) \psi(\rho_1)^2 \psi(\rho_2)^2 \\*
		& \hspace{20mm} + 2 \left( \tilde H_p \rho_1 \right) \phi(\tilde p)^2 \phi(\rho)^2 \psi'(\rho_1) \psi(\rho_1) \psi(\rho_2)^2 \\*
		& \hspace{20mm} + 2 \left( \tilde H_p \rho_2 \right) \phi(\tilde p)^2 \phi(\rho)^2 \psi(\rho_1)^2 \psi'(\rho_2)  \psi(\rho_2) \Big) ,
\end{align*}
where $f_r(\rho) = \Theta \frac{r \rho^{-1}}{1 + r \rho^{-1}}$.
We thus have
\[
	H_\p a_r + 2 p_1 a_r
		= \ - 2 \de \rho^{2s - 2m + 2} a_r^2 - b_r^2 + h_r p - b_{\infty, r}^2 - b_{1, r}^2 + b_{2, r}^2,
\]
where
\begin{align*}
	b_r
		:= & \ \rho^{-s} \Psi_r(\rho) \phi(\tilde \p) \phi(\rho) \psi(\rho_1) \psi(\rho_2) \\*
		& \ \times \sqrt{ \b_0 \left( 2s - m + 1 - 2 \tilde \b - 2 f_r(\rho) \right) - 2 \de \Psi_r(\rho)^2 \phi(\tilde \p)^2 \phi(\rho)^2 \psi(\rho_1)^2\psi(\rho_2)^2 }, \\
	h_r
		:= & \ 2 m \b_0 \rho^{-2s + m} \Psi_r(\rho)^2  \phi'(\tilde p) \phi(\tilde p) \phi(\rho)^2  \psi(\rho_1)^2 \psi(\rho_2)^2, \\
	b_{\infty, r}
		:= & \ \rho^{-s} \Psi_r(\rho) \phi(\tilde \p) \psi(\rho_1) \psi(\rho_2) \sqrt{ 2 \b_0 \rho \left( - \phi'(\rho)\phi(\rho)  \right) }, \\
	b_{1, r}
		:= & \ \rho^{-s} \Psi_r(\rho) \phi(\tilde \p) \psi(\rho_2) \psi(\rho) \sqrt{2 \left( \tilde H_p \rho_1 \right) \left( - \psi'(\rho_1) \psi(\rho_1) \right) }, \\
	b_{2, r}
		:= & \ \rho^{-s} \Psi_r(\rho) \psi(\rho_1) \phi(\tilde \p) \psi(\rho) \sqrt{ 2 \left( - \tilde H_p \rho_2 \right) \left( - \psi'(\rho_2)\psi(\rho_2)  \right) }.
\end{align*}
In order for this to make sense, we claim that there is a small enough $\de_0 > 0$, such that if $\de \in (0, \de_0)$, then $\tilde H_p \rho_1$ is positive, and $\tilde H_p \rho_2$ is negative, on the support of $\phi(\tilde p) \phi(\phi) \psi'(\rho_1) \psi(\rho_2)$, and on the support of $\phi(\tilde p) \phi(\rho) \psi(\rho_1) \psi'(\rho_2)$, respectively.
The support of $\phi(\tilde p) \phi(\phi) \psi'(\rho_1) \psi(\rho_2)$ is $\{\tilde p \leq \de^2 \} \cap \{ \rho \leq \de^2 \} \cap \{\de/2 \leq \rho_1 \leq \de\} \cap \{\rho_2 \leq \de\}$.
Let us restrict to the non-semiclassical setting, as the semiclassical setting is analogous.
Since \eqref{eq: Hp saddle} implies that $\tilde H_p \rho_1 \geq \b_1 \rho_1 + \Phi_1$ on $\Char(P) \cap \{ \rho_1 + \rho_2 < \de \} \subseteq S^*M$, where $\Phi_1$ is cubically vanishing at $\cR$, we conclude that $\tilde H_p \rho_1 \geq \b_1 \rho_1 + \Phi_{1,1} + \tilde p \Phi_{1, 2} + \rho \Phi_{1, 3}$, for smooth functions $\Phi_{1, 1}, \Phi_{1, 2}, \Phi_{1, 3}$ defined in some open neighbourhood of $\cR$ in $\oT M$, where $\Phi_{1,1}$ vanish cubically at $\cR$.
Now, on the support of $\phi(\tilde p) \phi(\phi) \psi'(\rho_1) \psi(\rho_2)$, we note that $\rho_1 \geq \de/2$, and hence $\rho_2 \leq \de \leq 2 \rho_1$, $\tilde p \leq \de^2 \leq 4 \rho_1^2$ and $\rho \leq \de^2 \leq 4 \rho_1^2$.
Using this, we conclude that
\begin{align*}
	\tilde H_p \rho_1 
		\geq & \ \b_1 \rho_1 + \Phi_{1,1} + \tilde p \Phi_{1, 2} + \rho \Phi_{1, 3} 
		\geq \ \b_1 \rho_1 - C \left( \rho_1 + \rho_2 + \rho + \tilde p \right)^{\frac32} - C \rho_1^2 \\
		\geq & \ \b_1 \rho_1 - C \left( \rho_1 + \rho_1^2 \right)^{\frac32} - C \rho_1^2
		\geq \frac12 \b_1 \rho_1,
\end{align*}
on the support of $\phi(\tilde p) \phi(\phi) \psi'(\rho_1) \psi(\rho_2)$, if $\de > 0$ is small enough.
We similarly conclude that $\tilde H_p \rho_1 \leq - \frac12 \b_2 \rho_2$ on the support of $\phi(\tilde p) \phi(\phi) \psi(\rho_1) \psi'(\rho_2)$, if $\de > 0$ is small enough.
The symbols $b_{1, r}$ and $b_{2, r}$ are therefore well-defined. 
The rest of the proof, in particular the quantization of this commutator identity, proceeds completely analogous to the proofs of \cite{H2025}*{Thm.~10.5} and \cite{H2025}*{Thm.~10.17}, where $F$ by assumption is elliptic on the essential support of $b_{2, r}$ and $b_{1, r}$, respectively.
\end{proof}
\end{appendix}

\end{document}